\documentclass[twoside,leqno,10pt, A4]{amsart}
\usepackage{amsfonts}
\usepackage{amsmath}
\usepackage{amscd}
\usepackage{amssymb}
\usepackage{amsthm}
\usepackage{amsrefs}
\usepackage{latexsym}
\usepackage{mathrsfs}
\usepackage{bbm}
\usepackage{enumerate}
\usepackage{graphicx}

\usepackage{amsfonts}
\usepackage{amsmath}
\usepackage{amscd}
\usepackage{amssymb}
\usepackage{amsthm}
\usepackage{amsrefs}
\usepackage{latexsym}
\usepackage{mathrsfs}
\usepackage{bbm}
\usepackage{amscd}
\usepackage{amssymb}
\usepackage{amsthm}
\usepackage{amsrefs}
\usepackage{latexsym}
\usepackage{mathrsfs}
\usepackage{bbm}
\usepackage{enumerate}
\usepackage{graphicx}
\usepackage{color}
\begin{document}

\newtheorem{theorem}[subsection]{Theorem}
\newtheorem{proposition}[subsection]{Proposition}
\newtheorem{lemma}[subsection]{Lemma}
\newtheorem{corollary}[subsection]{Corollary}
\newtheorem{conjecture}[subsection]{Conjecture}
\newtheorem{prop}[subsection]{Proposition}
\numberwithin{equation}{section}
\newcommand{\mr}{\ensuremath{\mathbb R}}
\newcommand{\mc}{\ensuremath{\mathbb C}}
\newcommand{\dif}{\mathrm{d}}
\newcommand{\intz}{\mathbb{Z}}
\newcommand{\ratq}{\mathbb{Q}}
\newcommand{\natn}{\mathbb{N}}
\newcommand{\comc}{\mathbb{C}}
\newcommand{\rear}{\mathbb{R}}
\newcommand{\prip}{\mathbb{P}}
\newcommand{\uph}{\mathbb{H}}
\newcommand{\fief}{\mathbb{F}}
\newcommand{\majorarc}{\mathfrak{M}}
\newcommand{\minorarc}{\mathfrak{m}}
\newcommand{\sings}{\mathfrak{S}}
\newcommand{\fA}{\ensuremath{\mathfrak A}}
\newcommand{\mn}{\ensuremath{\mathbb N}}
\newcommand{\mq}{\ensuremath{\mathbb Q}}
\newcommand{\half}{\tfrac{1}{2}}
\newcommand{\f}{f\times \chi}
\newcommand{\summ}{\mathop{{\sum}^{\star}}}
\newcommand{\chiq}{\chi \bmod q}
\newcommand{\chidb}{\chi \bmod db}
\newcommand{\chid}{\chi \bmod d}
\newcommand{\sym}{\text{sym}^2}
\newcommand{\hhalf}{\tfrac{1}{2}}
\newcommand{\sumstar}{\sideset{}{^*}\sum}
\newcommand{\sumprime}{\sideset{}{'}\sum}
\newcommand{\sumprimeprime}{\sideset{}{''}\sum}
\newcommand{\sumflat}{\sideset{}{^\flat}\sum}
\newcommand{\shortmod}{\ensuremath{\negthickspace \negthickspace \negthickspace \pmod}}
\newcommand{\V}{V\left(\frac{nm}{q^2}\right)}
\newcommand{\sumi}{\mathop{{\sum}^{\dagger}}}
\newcommand{\mz}{\ensuremath{\mathbb Z}}
\newcommand{\leg}[2]{\left(\frac{#1}{#2}\right)}
\newcommand{\muK}{\mu_{\omega}}
\newcommand{\thalf}{\tfrac12}
\newcommand{\lp}{\left(}
\newcommand{\rp}{\right)}
\newcommand{\Lam}{\Lambda_{[i]}}
\newcommand{\lam}{\lambda}

\theoremstyle{plain}
\newtheorem{conj}{Conjecture}
\newtheorem{remark}[subsection]{Remark}

\makeatletter
\def\widebreve{\mathpalette\wide@breve}
\def\wide@breve#1#2{\sbox\z@{$#1#2$}%
     \mathop{\vbox{\m@th\ialign{##\crcr
\kern0.08em\brevefill#1{0.8\wd\z@}\crcr\noalign{\nointerlineskip}%
                    $\hss#1#2\hss$\crcr}}}\limits}
\def\brevefill#1#2{$\m@th\sbox\tw@{$#1($}%
  \hss\resizebox{#2}{\wd\tw@}{\rotatebox[origin=c]{90}{\upshape(}}\hss$}
\makeatletter

\title[Non-vanishing of central values of quadratic Hecke $L$-functions of prime moduli in the Gaussian field]{Non-vanishing of central values of quadratic Hecke $L$-functions of prime moduli in the Gaussian field}

%%\date{\today}
\author{Peng Gao}
\address{School of Mathematical Sciences, Beihang University, Beijing 100191, P. R. China}
\email{penggao@buaa.edu.cn}
\begin{abstract}
 We study the first and second mollified moments of central values of a quadratic family of Hecke $L$-functions of prime moduli to show that more than nine percent of the members of this family do not vanish at the central values.
\end{abstract}

\maketitle



\section{Introduction}
\label{sec 1}

  Central values of $L$-functions have received considerable attention in the literature as they carry rich arithmetic information. In general,  an $L$-function is expected to vanish at the central value for a reason. Such a reason may simply come from an observation on the sign of the functional equation or arise from ties with deep assertions such as the Birch and Swinnerton-Dyer conjecture.

  For the case of Dirichlet $L$-functions, it is believed that $L(\frac{1}{2},\chi) \neq 0$ for any Dirichlet character $\chi$. When $\chi$ is a primitive quadratic character, this is a conjecture of S. Chowla \cite{Chow}. In \cite{Jutila}, M. Jutila initiated the study on the first two  moments of the family of quadratic Dirichlet $L$-functions. His results imply that Chowla's conjecture is true for infinitely many such $L$-functions.  By further evaluating the mollified moments, K. Soundararajan \cite{sound1} showed that at least $87.5\%$ of the members of the same family do not vanish at the central value.

  Other than the entire family of quadratic Dirichlet $L$-functions, it is also intriguing  to investigate the non-vanishing issue of the family of quadratic Dirichlet $L$-functions over prime moduli. In this case, it again follows from the work on Jutila in \cite{Jutila}, who also obtained the first moment of the family of quadratic Dirichlet $L$-functions of prime moduli, that there are infinitely many such $L$-functions having non-vanishing central values.  With more efforts, by combining an evaluation on the mollified first moment and an upper bound for the mollified second moment using sieve methods, S. Baluyot and K. Pratt \cite{B&P} were able to show that more than nine percent of the members of the quadratic family of Dirichlet $L$-functions of prime moduli do not vanish at the central value.

  In \cite{Gao1}, the author studied the mollified first and second moments of a family of quadratic Hecke $L$-functions in the Gaussian field to show that at least $87.5\%$ of the members of that family do not vanish at the central value. This can be regarded as an analogue to the result of Soundararajan in \cite{sound1}. In this paper, motivated by the above mentioned result of Baluyot and Pratt in \cite{B&P}, it is our goal to investigate the non-vanishing issue of the quadratic family of Hecke $L$-functions introduced in \cite{Gao1} with a further restriction to prime moduli.

   Throughout the paper, we denote $K=\mq(i)$ for the Gaussian field and $\mathcal{O}_K=\mz[i]$ for the ring of integers of $K$.  We say an element $d \in \mathcal O_K$ is odd if $(d,2)=1$. We also denote $L(s,\chi)$ for the $L$-function associated to a Hecke character $\chi$ and $\zeta_{K}(s)$ for the Dedekind zeta function of $K$. We denote $\varpi$ for a prime element in $\mathcal O_K$, by which we mean that the ideal $(\varpi)$ is a prime ideal in $\mathcal O_K$.  The expression $\chi_c$ is reserved for the quadratic residue symbol $\leg {c}{\cdot}$ defined in Section \ref{sec2.4}.

  A Hecke character $\chi$ of $K$ is said to be of trivial infinite type if its component at infinite places of $K$ is trivial and a Hecke character $\chi$ is said to be primitive modulo $q$ if it does not factor through $\left (\mathcal{O}_K / (q') \right )^{\times}$ for any  divisor $q'$ of $q$ such that $N(q')<N(q)$. It is shown in \cite[Section 2.1]{Gao2} that $\chi_{(1+i)^5\varpi}$ defines a primitive quadratic Hecke character modulo $(1+i)^5\varpi$ of trivial infinite type for any odd prime $\varpi \in \mathcal{O}_K$. Thus, for technical reasons, instead of considering a family of $L$-functions of $\{ L(s,\chi_{\varpi}) \}$ for primes $\varpi$ satisfying certain congruence conditions (so that the corresponding $L$-functions become primitive), we considering the following family of $L$-functions:
\begin{align}
\label{Lfamily}
 \mathcal F = \big\{ L(\thalf, \chi_{(1+i)^5\varpi}) :  \varpi  \hspace{0.05in} \text{odd prime} \big\}.
\end{align}

  We aim to prove the following result in the paper, which shows that more than nine percent of the members of the above family have non-vanishing central values.
\begin{theorem}\label{thm: pos prop}
For all large $X$, we have
\begin{align*}
\sum_{\substack{(\varpi, 2)=1 \\ N(\varpi) \leq X  \\ L(\frac{1}{2},\chi_{(1+i)^5\varpi})\neq 0}}1 &\geq 0.0964 \sum_{\substack{(\varpi, 2)=1 \\ N(\varpi) \leq X }} 1.
\end{align*}
\end{theorem}

 Notice here the percentage we obtain in Theorem \ref{thm: pos prop} is exactly the same as the one given in \cite[Theorem 1.1]{B&P}. This is not surprising since our proof of Theorem \ref{thm: pos prop} follows closely the proof of \cite[Theorem 1.1]{B&P} by Baluyot and Pratt. We now briefly outline the approach of the proof. Let $X$ be a large number and for some fixed $\theta, \vartheta \in \left(0 ,\tfrac{1}{2}\right)$, we define
\begin{align}\label{eq:outline section, defn of M, length of mollifier}
M = X^\theta,  \quad R = X^\vartheta.
\end{align}

  We fix a smooth function $\Phi(x)$, compactly supported in $[\frac{1}{2},1]$, satisfying $\Phi(x) = 1$ for $x \in [\frac{1}{2} + \frac{1}{\log X},1 - \frac{1}{\log X}]$ and $\Phi^{(j)}(x) \ll_j (\log X)^j$ for all $j \geq 0$. Let $H(t)$ be another smooth function to be optimized later such that it is compactly supported in $[-1,1]$. For $m \in \mathcal O_K$, we set
\begin{align}\label{eq: defn of mollifier coeffs bm}
b_m = \mu_{[i]}(m) H \left(\frac{\log N(m)}{\log M} \right),
\end{align}
 where we denote $N(m)$ for the norm of $m$ and $\mu_{[i]}$ for the analogue on $\mathcal O_K$ of the usual M\"obius function on $\mz$. We use $b_m$ to define the mollifier function $M(\varpi)$ for every odd prime $\varpi$ by
\begin{align}\label{eq: defn of mollifier}
M(\varpi) = \sum_{\substack{N(m) \leq M \\ m \equiv 1 \bmod {(1+i)^3} }} \frac{b_m}{\sqrt{N(m)}} \chi_{(1+i)^5\varpi}(m).
\end{align}

  Here we recall from \cite[Section 2.1]{G&Zhao4} that every ideal in $\mathcal{O}_K$ co-prime to $2$ has a unique generator congruent to $1$ modulo $(1+i)^3$ which is called primary. Hence in \eqref{eq: defn of mollifier} and in what follows, a sum of the form $\sum_{\substack{ m \equiv 1 \bmod {(1+i)^3} }}$  indicates that we are summing over primary elements in $\mathcal O_K$.

  Now we introduce the mollified first moment $S_1$ and the mollified second moment $S_2$ of the family $\mathcal F$ given in \eqref{Lfamily} as follows:
\begin{equation}\label{eq: defn moment sums S1 and S2}
\begin{split}
S_1 & = \sum_{\substack{(\varpi, 2)=1}} \log N(\varpi) \Phi \left( \frac{N(\varpi)}{X}\right) L\left(\tfrac{1}{2},\chi_{(1+i)^5\varpi}\right) M(\varpi), \\
S_2 &= \sum_{\substack{(\varpi, 2)=1}} \log N(\varpi) \Phi \left( \frac{N(\varpi)}{X}\right) L\left(\tfrac{1}{2},\chi_{(1+i)^5\varpi} \right)^2 M(\varpi)^2.
\end{split}
\end{equation}

  Our aim is to evaluate both $S_1$ and $S_2$ asymptotically. The evaluation of $S_1$ is relatively easy, which is performed in Section \ref{sec:mollified first moment} and our result is summarized in the following proposition.
\begin{prop}\label{prop: asymptotic for S1}
Let $0 < \theta < \frac{1}{2}$ be fixed. For all large $X$, we have
\begin{align*}
S_1 &= 2\left(H(0) - \frac{1}{2\theta}H'(0) \right) X  + O \left(\frac{X}{\log X} \right).
\end{align*}
\end{prop}

   It is a challenging task to evaluate $S_2$. However, for our purpose, all we need is an upper bound of the right order of magnitude for $S_2$. Thus, instead of evaluating $S_2$ asymptotically, we follow the approach in \cite{B&P} to use sieves to derive an upper bound for $S_2$ in Section \ref{sec:mollified second moment}. The result is given in the following Proposition.
\begin{prop}\label{prop:upper bound for S2}
Let $\delta > 0$ be small and fixed, and let $\theta,\vartheta$ satisfy $\theta + 2\vartheta < \frac{1}{2}$. If $X \geq X_0(\delta,\theta,\vartheta)$, then
\begin{align}
\label{S2bound}
S_2 \leq (1+\delta) \frac{\mathfrak{I}}{\vartheta} 2X,
\end{align}
where
\begin{align*}
\mathfrak{I} =& -2\int_0^1 H(x)H'(x) dx + \frac{1}{\theta}\int_0^1 H(x) H''(x) dx + \frac{1}{\theta}\int_0^1 H'(x)^2 dx \\
&- \frac{1}{2\theta^2}\int_0^1 H'(x)H''(x)dx+\frac{1}{24\theta^3}\int_0^1 H''(x)^2 dx.
\end{align*}
\end{prop}

  In our proof of Proposition \ref{prop:upper bound for S2}, the use of sieves allows us to reduce the difficulty of estimating certain character sums\ over primes to an evaluation on a character sum over algebraic integers in $\mathcal O_K$. We then adapt the methods developed by Soundararajan in \cite{sound1} to treat the resulting sum. The most delicate part of the treatment consists of applying a two-dimensional Poisson summation to convert the desired character sum to its dual sum. A careful analysis on the dual sum ultimately leads to the bound of $S_2$ given in  \eqref{S2bound}.

   With both Proposition \ref{prop: asymptotic for S1} and Proposition \ref{prop:upper bound for S2} available, we apply the Cauchy-Schwarz inequality to see that
\begin{align}\label{eq: Cauchy Schwarz first second moment inequality}
\sum_{\substack{(\varpi, 2)=1 \\ L(\frac{1}{2},\chi_{(1+i)^5\varpi}) \neq 0}} \log N(\varpi)\Phi \left( \frac{N(\varpi)}{X}\right) \geq \frac{S_1^2}{S_2}.
\end{align}
  The optimal choice of $H(t)$ in the above estimation has already been determined in \cite[Section 8]{B&P}, which allows us to deduce the conclusion given in Theorem \ref{thm: pos prop}.

\section{Preliminaries}
\label{sec 2}

 As tools needed in the rest of the paper, we gather here some auxiliary results.

%%----------------------------------------------------------------------------
\subsection{Residue symbol and Gauss sum}
\label{sec2.4}
%%----------------------------------------------------------------------------
   It is well-known that the Gaussian field $K=\mq(i)$ has class number one.  We denote $U_K=\{ \pm 1, \pm i \}$ for the group of units in $\mathcal{O}_K$ and $D_{K}=-4$ for the discriminant of $K$.  We say an element $d \in \mathcal O_K$ is a perfect square if $d=n^2$ for some $n \in \mathcal O_K$ and we denote it by writing $d = \square$. We say an element $d \in \mathcal O_K$ is square-free if the ideal $(d)$ is not divisible by the square of any prime ideal in $\mathcal{O}_K$.

  Recall that every ideal in $\mathcal{O}_K$ co-prime to $2$ has a unique generator called primary. As $(1+i)$ is the only prime ideal in $\mathcal{O}_K$ that lies above the integral ideal $(2) \in \mz$, we can fix a generator for every prime ideal $(\varpi) \in \mathcal{O}_K$ by taking $\varpi$ to be $1+i$ for the ideal $(1+i)$ and by taking $\varpi$ to be primary otherwise. By further taking $1$ as the generator for the ring $\mz[i]$ itself, we
extend the choice of the generator for any ideal of $\mathcal{O}_K$ multiplicatively. We denote $G$ for this set of generators. For $a, b \in \mathcal O_K$, we write $[a,b]$ for their least common multiple such that $[a,b]\in G$. Similarly, we write $(a,b)$ for their greatest common divisor such that $(a, b) \in G$.

   For an odd $n \in \mathcal{O}_{K}$, the quadratic residue symbol $\leg{\cdot}{n}$ modulo $n$ is first defined when $n=\varpi$ is a prime. In this case, for any $a \in
\mathcal{O}_{K}$,  we have $\leg{a}{\varpi} =0$ when $\varpi | a$ and $\leg{a}{\varpi} \equiv
a^{(N(\varpi)-1)/2} \pmod{\varpi}$ with $\leg{a}{\varpi} \in \{
\pm 1 \}$ when $(a, \varpi)=1$.
 Then the definition is extended
to any composite $n$ multiplicatively. Moreover, for $n \in U_K$, we define $\leg {\cdot}{n}=1$.

 For two co-prime primary elements $m, n \in \mathcal{O}_{K}$, we have the following quadratic reciprocity law (see \cite[(2.1)]{G&Zhao4})
\begin{align}
\label{quadrec}
 \leg{m}{n} = \leg{n}{m}.
\end{align}

  Recall that $\chi_{c}$ is reserved for the quadratic residue symbol $\leg {c}{\cdot}$. For odd $c$, it is shown in \cite[Section 2.1]{G&Zhao2019} that $\chi_{c}$ can be regarded as a Hecke characters of trivial infinite type modulo $16c$, provided that we define $\leg {c}{a}=0$ when $1+i|a$. We shall henceforth view $\chi_{c}$ as a Hecke character whose conductor dividing $16c$. We make one exception here that we regard $\chi_{\pm 1}$ as a principal character modulo $1$ (so we have $\chi_{\pm 1}(a)=1$ for all $a \in \mathcal O_K$).  This further implies that we have $L(s, \chi_{\pm 1})=\zeta_K(s)$.

 For any complex number $z$, we write $e(z) = e^{2\pi i z}$ and we denote
\begin{align*}
 \widetilde{e}(z) =\exp \left( 2\pi i  \left( \frac {z}{2i} - \frac {\bar{z}}{2i} \right) \right) .
\end{align*}
  For $r, n \in \mathcal O_K$ with $(n,2)=1$, the quadratic Gauss sum $g(r, n)$ is defined by
\begin{align*}
%%\label{g2}
  g(r,n) = \sum_{x \bmod{n}} \leg{x}{n} \widetilde{e}\leg{rx}{n}.
\end{align*}

 Let $\varphi_{[i]}(n)$ denote the number of elements in the reduced residue class of $\mathcal{O}_K/(n)$,
we recall the following explicitly evaluations of $g(r,n)$ from \cite[Lemma 2.2]{G&Zhao4}.
\begin{lemma} \label{Gausssum}
\begin{enumerate}[(i)]
\item  We have
\begin{align*}
%%\label{2.7}
 g(rs,n) & = \overline{\leg{s}{n}} g(r,n), \qquad (s,n)=1, \\
   g(k,mn) & = g(k,m)g(k,n),   \qquad  m,n \text{ primary and } (m , n)=1 .
\end{align*}
\item Let $\varpi$ be a primary prime in $\mathcal{O}_K$. Suppose $\varpi^{h}$ is the largest power of $\varpi$ dividing $k$. (If $k = 0$ then set $h = \infty$.) Then for $l \geq 1$,
\begin{align*}
g(k, \varpi^l)& =\begin{cases}
    0 \qquad & \text{if} \qquad l \leq h \qquad \text{is odd},\\
    \varphi_{[i]}(\varpi^l) \qquad & \text{if} \qquad l \leq h \qquad \text{is even},\\
    -N(\varpi)^{l-1} & \text{if} \qquad l= h+1 \qquad \text{is even},\\
    \leg {ik\varpi^{-h}}{\varpi}N(\varpi)^{l-1/2} \qquad & \text{if} \qquad l= h+1 \qquad \text{is odd},\\
    0, \qquad & \text{if} \qquad l \geq h+2.
\end{cases}
\end{align*}
\end{enumerate}
\end{lemma}

%%----------------------------------------------------------------
\subsection{The approximate functional equation}
\label{sect: apprfcneqn}
%%-----------------------------------------------------------------
    Let $\chi$ be a primitive quadratic Hecke character $\chi$ of $K$ of trivial infinite type. A well-known result of E. Hecke says that $L(s, \chi)$ has an analytic continuation to the entire complex plane and satisfies the following
functional equation (see \cite[Theorem 3.8]{iwakow})
\begin{align}
\label{fneqn}
  \Lambda(s, \chi) = W(\chi)(N(m))^{-1/2}\Lambda(1-s, \chi),
\end{align}
   where $m$ is the conductor of $\chi$, $|W(\chi)|=(N(m))^{1/2}$ and
\begin{align}
\label{Lambda}
  \Lambda(s, \chi) = (|D_K|N(m))^{s/2}(2\pi)^{-s}\Gamma(s)L(s, \chi).
\end{align}

   In particular, we have the following functional equation for $\zeta_K(s)$:
\begin{align}
\label{fcneqnforzeta}
\pi^{-s}\Gamma(s)\zeta_K(s)=\pi^{-(1-s)}\Gamma(1-s)\zeta_K(1-s).
\end{align}

By combining \cite[Theorem 3.8]{iwakow} and \cite[Lemma 2.2]{Gao2}, we see that
$W(\chi_{(1+i)^5d})=g(\chi_{(1+i)^5d})=\sqrt{N((1+i)^5d)}$ when $\chi=\chi_{(1+i)^5d}$ for an odd, square-free $d \in \mathcal{O}_K$. In this case \eqref{fneqn} becomes
\begin{align*}
%%\label{fneqnquad}
  \Lambda(s, \chi_{(1+i)^5d}) = \Lambda(1-s, \chi_{(1+i)^5d}).
\end{align*}

   For $n \in \mathcal{O}_{K}$ and $j \in \mz, j \geq 1$, we denote $d_{[i], j}(n)$ for the analogue on $\mathcal{O}_K$ of the usual function $d_k$ on $\mz$, so that $d_{[i],j}(n)$ equals the coefficient of $N(n)^{-s}$ in the Dirichlet series expansion of the $j$-th power of $\zeta_K(s)$. It follows that $d_{[i],1}(n)=1$ and when $n$ is primary,
\begin{align*}
  d_{[i],j}(n)=\sum_{\substack{ a_1\cdots a_j=n \\ a_i \equiv 1 \bmod {(1+i)^3}, 1 \leq i \leq j } }1.
\end{align*}

  We further denote for $j \in \mz, j \geq 1$ and any real number $t>0$,
\begin{align}
\label{eq:Vdef}
 V_j(t) = \frac{1}{2 \pi i} \int\limits\limits_{(2)}  w_j(s) t^{-s} \frac {ds}{s}, \quad w_j(s) = \left(\frac{2^{5/2}}{\pi}\right)^{js}
\left ( \frac {\Gamma(\frac{1}{2}+s)}{\Gamma(\frac{1}{2})} \right )^j.
\end{align}
  We then note the following approximate functional equation for $L(\half, \chi_{(1+i)^5d})^j$ from \cite[Lemma 2.2]{Gao2}.
\begin{lemma}[Approximate functional equation]
\label{lem:AFE}
  For any odd, square-free $d \in \mathcal{O}_K$, we have for $j=1,2$,
\begin{align*}
%%\label{fcneqnL}
\begin{split}
 L(\half, \chi_{(1+i)^5d})^j = & 2\sum_{\substack{n \equiv 1 \bmod {(1+i)^3}}} \frac{\chi_{(1+i)^5d}(n)d_{[i],j}(n)}{N(n)^{\frac{1}{2}}} V_j
\left(\frac{ N(n)}{N(d)^{j/2}} \right).
\end{split}
\end{align*}
\end{lemma}

  The next lemma gives the behaviors of $V_j(t)$ defined in \eqref{eq:Vdef} for $t \rightarrow 0^+$ or $t \rightarrow \infty$, which can be established similar to  \cite[Lemma 2.1]{sound1}.
\begin{lemma}\label{lem: properties of omega j}
Let $j = 1,2$. The function $V_j(\xi)$ is real-valued and smooth on $(0,\infty)$. We have
\begin{align*}
V_j(\xi) = 1 + O_\varepsilon(\xi^{\frac{1}{2}-\varepsilon}).
\end{align*}
For any fixed integer $\nu \geq 0$ and large $\xi$, we have
\begin{align*}
V_j^{(\nu)}(\xi) \ll  \xi^{\nu+\frac 32} \exp\left( - \frac{j}{2} \xi^{\frac{2}{j}}\right) \ll_\nu \exp\left(-\frac{j}{4}\xi^{\frac{2}{j}} \right).
\end{align*}
\end{lemma}

%%----------------------------------------------------------------------------
\subsection{Poisson summation}
\label{sec Poisson}
%%----------------------------------------------------------------------------

   A key ingredient needed in our treatment of the paper is the following two dimensional Poisson summation, which follows from \cite[Lemma 2.7, Corollary 2.8]{G&Zhao4}.
\begin{lemma}
\label{Poissonsumformodd} Let $n \in \mathcal{O}_K$ be primary and $\leg {\cdot}{n}$ be the quadratic residue symbol modulo $n$. For any smooth function $W:\mr^{+} \rightarrow \mr$ of compact support,  we have for $X>0$,
\begin{align*}
   \sum_{\substack {m \in \mathcal{O}_K \\ (m,1+i)=1}}\leg {m}{n} W\left(\frac {N(m)}{X}\right)=\frac {X}{2N(n)}\leg {1+i}{n}\sum_{k \in
   \mathcal{O}_K}(-1)^{N(k)} g(k,n)\widetilde{W}\left(\sqrt{\frac {N(k)X}{2N(n)}}\right),
\end{align*}
   where
\begin{align}
\label{Wtdef}
   \widetilde{W}(t) =& \int\limits^{\infty}_{-\infty}\int\limits^{\infty}_{-\infty}W(N(x+yi))\widetilde{e}\left(- t(x+yi)\right)\dif x \dif y, \quad t \geq 0.
\end{align}
\end{lemma}

\subsection{Estimations related to character sums}
  We now collect two lemmas on estimations related to character sums. They are consequences of a large sieve result of K. Onodera \cite{Onodera} on quadratic residue symbols in the Gaussian field, which is a generalization in $K$ of the well-known large sieve result on quadratic Dirichlet characters of D. R. Heath-Brown \cite{DRHB}. The first lemma can be obtained by applying the large sieve result in \cite{Onodera} in the proof of \cite[Corollary 2]{DRHB} and \cite[Lemma 2.4]{sound1}.
\begin{lemma}\label{lem: estimates for character sums}
  Let $N, Q$ be positive integers, and let $a_1, \cdots, a_n$ be arbitrary complex numbers. Let $S(Q)$ denote the set of $\chi_m$ for square-free $m$ satisfying $N(m) \leq Q$. Then for any $\epsilon > 0$,
\begin{align*}
 \sum_{\chi \in S(Q)}\Big | \sum_{\substack{n \equiv 1 \bmod {(1+i)^3} \\ N(n) \leq N}}\mu^2_{[i]}(n) a_n\chi(n) \Big |^2 \ll_{\epsilon}  (QN)^{\epsilon}(Q + N) \sum_{\substack{n_1, n_2 \equiv 1 \bmod {(1+i)^3} \\ N(n_1), N(n_2) \leq N \\ n_1n_2=\square}}|a_{n_1}a_{n_2}|.
\end{align*}
 Let $M$ be a positive integer, and for each $m \in \mathcal O_K$ satisfying $N(m) \leq M$, we write $m = m_1m_2^2$ with $m_1$ square-free and $m_2 \in G$. Suppose the sequence $a_n$ satisfies $|a_n| \ll N(n)^\varepsilon$, then
\begin{align*}
\sum_{N(m) \leq M} \frac{1}{N(m_2)} \left|\sum_{N(n) \leq N} a_n \left( \frac{m}{n}\right) \right|^2 \ll (MN)^\varepsilon N (M+N).
\end{align*}
\end{lemma}

  The second lemma is a combination of \cite[Lemma 2.15]{Gao2} and the proof of \cite[Lemma 2.5]{sound1}.
\begin{lemma}
\label{lem:2.3}
Suppose $\sigma+it$ is a complex number with $\sigma \geq \frac{1}{2}$. Then
\begin{align*}
%%\label{equ:13.6}
\sumstar_{\substack{(d,2)=1 \\ N(d) \leq X}} |L(\sigma+it,\chi_{(1+i)^5d})| ^4
\ll X^{1+\varepsilon} (1+|t|^2)^{1+\varepsilon}.
\end{align*}
  Here the ``$*$'' on the sum over $d$ means that the sum is restricted to square-free elements $d$ in $\mathcal{O}_K$.
\end{lemma}

\subsection{Analytical behaviors of certain functions}
\label{sect: alybehv}

   In this section, we discuss the analytical behaviors of certain functions that are needed in the paper.
First, we have the following result that can be established similar to  \cite[Lemma 5.3]{sound1}.
\begin{lemma}
\label{lem: nu-sum as an Euler product}
  Let $\alpha, d, \ell \in \mathcal O_K$ be primary and let
\begin{equation}
\label{eq: defn of d 1}
d_1 = \frac{d}{(d,\alpha)}.
\end{equation}
  For each $k \in \mathcal O_K, k \neq 0$, we write $kd_1$ uniquely by
\begin{equation}\label{eq: defn of k1 and k2}
 kd_1 = k_1 k_2^2,
\end{equation}
 with $k_1$ square-free and $k_2 \in G$. For $\Re(s)>1$, we have
$$
\sum_{\substack{\nu \equiv 1 \bmod {(1+i)^3} \\ (\nu,\alpha d)=1}} \frac {d_{[i],2}(\nu)}{N(\nu)^{s}} \left( \frac{d_1}{\nu}\right) \frac {g(k,\ell \nu)}{N(\nu)^{1/2}} \ = \ L(s,\chi_{ik_1})^2\prod_{\varpi \in G}\mathcal{G}_{0,\varpi}(s;k,\ell,\alpha, d) \ =: \ L(s,\chi_{ik_1})^2\mathcal{G}_0(s;k,\ell,\alpha, d),
$$
where $\mathcal{G}_{0,\varpi}(s;k,\ell,\alpha, d)$ is defined by
$$
\mathcal{G}_{0,\varpi}(s;k,\ell,\alpha, d) \ = \ \Bigg( 1-\frac{1}{N(\varpi)^{s}}\Bigg(\frac{ik_1}{\varpi}\Bigg)\Bigg)^2 \ \ \ \ \text{if }\varpi|2\alpha d, \ \ \
$$
$$
\mathcal{G}_{0,\varpi}(s;k,\ell,\alpha, d) \ = \ \Bigg( 1-\frac{1}{N(\varpi)^{s}}\Bigg(\frac{i k_1}{\varpi}\Bigg)\Bigg)^2\sum_{r=0}^{\infty} \frac{r+1 }{N(\varpi)^{rs}}\left( \frac{d_1}{\varpi^r}\right)\frac{g(k, \varpi^{r+\text{\upshape{ord}}_\varpi(\ell)})}{N(\varpi)^{r/2}}\ \ \ \ \text{if }\varpi \nmid 2\alpha d.
$$
The function $\mathcal{G}_0(s;k,\ell,\alpha, d)$ is holomorphic for $\Re(s)>\frac{1}{2}$. Furthermore, on writing
\begin{equation*}
%%\label{eq: defn of k3 and k4}
k = k_3 k_4^2,
\end{equation*}
with $k_3$ square-free and $k_4 \in G$, we have uniformly for $\Re(s)\geq \frac{1}{2}+\varepsilon$,
$$
\mathcal{G}_0(s;k,\ell,\alpha, d) \ll_{\varepsilon} N(\alpha dk\ell)^{\varepsilon} N(\ell)^{1/2} N((\ell,k_4^2))^{1/2}.
$$
\end{lemma}

 Next, let $\Phi(t)$ be the smooth function appearing in the definition of $S_1$ and $S_2$ given in \eqref{eq: defn moment sums S1 and S2} and $V_2(t)$ be given in \eqref{eq:Vdef}, we define
\begin{equation}\label{Fdef}
F_{y}(t) = \Phi(t) V_2\left( \frac{y}{tX}\right).
\end{equation}

  We further define for $\xi>0$ and $\Re(w)>0$,
\begin{equation}
\label{h}
h(\xi,w) = \int_0^{\infty} \widetilde{F}_{t}\left( \left(\frac{\xi}{t} \right )^{1/2} \right) t^{w-1} \,dt.
\end{equation}

  Before we state our next lemma, we would like to recall that the Mellin transform $\widehat g(s)$ of a function $g$ is given by
\begin{align*}
 \widehat g(s) = \int_0^\infty g(t) t^{s-1} dt.
\end{align*}

  Now, we are ready to present a result concerning some analytic properties of $h(\xi,w)$.
\begin{lemma}\label{lem: properties of h(xi,w)}
Let $F_t$ be defined by \eqref{Fdef} and let $\xi >0$. The function $h(\xi,w)$ is an entire function of $w$ in $\Re(w)>-1$ such that
\begin{align}\label{eq: integral in lemma for h(xi,w)}
\begin{split}
  h(\xi,w)
=&  \widehat{\Phi}(1+w) \xi^{w} \frac {\pi }{2\pi i}
\int\limits\limits_{(c)}\left(\frac{2^{5/2}}{\pi}\right)^{2s}
\left ( \frac {\Gamma(\frac{1}{2}+s)}{\Gamma(\frac{1}{2})} \right )^2 \left(\frac {X}{\xi} \right )^s (\pi)^{-2s+2w}\frac{\Gamma (s-w)}{\Gamma (1-s+w)} \frac {ds}{s}
\end{split}
\end{align}
for any $c$ with $c>\max\{0,\Re(w)\}$. Moreover, in the region $1\geq \Re(w)> -1$, it satisfies the bound
\begin{align}
\label{hbound}
h(\xi,w) \ll (1+|w|)^{3-2\Re(w)} \exp \Bigg( -\frac{1}{10}\frac{\xi^{1/4}}{X^{1/4}(|w|+1)^{1/2}}\Bigg) \xi^{\Re(w)} |\widehat{\Phi}(1+w)|.
\end{align}
\end{lemma}
\begin{proof}
  We recall that for any smooth function $W$, the function $\widetilde{W}(t)$ defined in \eqref{Wtdef} can be evaluated in polar coordinates as
\begin{align*}
     \widetilde{W}(t) =& 4\int\limits^{\pi/2}_0\int\limits^{\infty}_0\cos (2\pi t r\sin \theta)W(r^2) \ r \dif r \dif \theta.
\end{align*}
  It follows from this and the definition of $V_2(t)$ in \eqref{eq:Vdef} that we have, for $c_s>2$,
\begin{align*}
& h(\xi,w) \\
=& 4\int\limits_0^{\infty} \int\limits^{\pi/2}_0 \int\limits^{\infty}_0\cos (2\pi \left(\frac{\xi}{t} \right )^{1/2} r \sin \theta)\Phi(r^2) V_2\left( \frac{t}{r^2X}\right) \ r\dif r \dif \theta t^{w-1} \,dt  \\
=& 2\int\limits_0^{\infty} \int\limits^{\pi/2}_0 \int\limits^{\infty}_0\cos (r)\Phi((\frac {rt^{1/2}}{2\pi \xi^{1/2} \sin \theta})^2) V_2\left( \frac{t}{X}(\frac {rt^{1/2}}{2\pi \xi^{1/2} \sin \theta})^{-2} \right) \ \dif (\frac {rt^{1/2}}{2\pi \xi^{1/2} \sin \theta})^{2} \dif \theta t^{w-1} \,dt  \\
=& 2 \int\limits^{\pi/2}_0 \int\limits^{\infty}_0\cos (r) V_2\left( \frac{1}{X}(\frac {r}{2\pi \xi^{1/2} \sin \theta})^{-2} \right) \ \dif (\frac {r}{2\pi \xi^{1/2} \sin \theta})^{2} \dif \theta \int\limits_0^{\infty} \Phi((\frac {rt^{1/2}}{2\pi \xi^{1/2} \sin \theta})^2) t^{w+1} \,\frac {dt}{t}  \\
=& 2 \widehat{\Phi}(1+w) \int\limits^{\pi/2}_0 \int\limits^{\infty}_0\cos (r) V_2\left( \frac{1}{X}(\frac {r}{2\pi \xi^{1/2} \sin \theta})^{-2} \right)(\frac {r}{2\pi \xi^{1/2} \sin \theta})^{-2w-2} \ \dif (\frac {r}{2\pi \xi^{1/2} \sin \theta})^{2} \dif \theta   \\
=& 2 \widehat{\Phi}(1+w) \xi^{w} \int\limits^{\pi/2}_0 \int\limits^{\infty}_0\cos (r) \frac {1}{2\pi i}
\int\limits\limits_{(c_s)}\left(\frac{2^{5/2}}{\pi}\right)^{2s}
\left ( \frac {\Gamma(\frac{1}{2}+s)}{\Gamma(\frac{1}{2})} \right )^2 \left( \frac{1}{X}(\frac {r}{2\pi \xi^{1/2} \sin \theta})^{-2} \right)^{-s}(\frac {r}{2\pi \sin \theta})^{-2w-2} \ \frac {ds}{s} \dif (\frac {r}{2\pi  \sin \theta})^{2} \dif \theta   \\
=& 4 \widehat{\Phi}(1+w) \xi^{w} \frac {1}{2\pi i}
\int\limits\limits_{(c_s)}\left(\frac{2^{5/2}}{\pi}\right)^{2s}
\left ( \frac {\Gamma(\frac{1}{2}+s)}{\Gamma(\frac{1}{2})} \right )^2 \left(\frac {X}{\xi} \right )^s (2\pi)^{-2s+2w} \int\limits^{\pi/2}_0 \int\limits^{\infty}_0\cos (r)r^{2s-2w-1}  \left(  \sin \theta \right)^{-2s+2w} \frac {ds}{s} \dif r \dif \theta .
\end{align*}

  Now, applying the relation (see \cite[Section 2.4]{Gao1})
\begin{align*}
\int\limits^{\pi/2}_0 (\sin \theta )^{-u} \dif \theta
\int\limits^{\infty}_0\cos (r)r^{u}\frac {\dif r}{r}=\frac {\pi}{2}2^{u-1}\frac{\Gamma \leg{u}{2}}{\Gamma\leg{2-u}{2}},
\end{align*}
  we see that
\begin{align*}
\int\limits^{\pi/2}_0 (\sin \theta )^{-(2s-2w)} \dif \theta
\int\limits^{\infty}_0\cos (r)r^{2s-2w}\frac {\dif r}{r}=\frac {\pi}{2}2^{2s-2w-1}\frac{\Gamma (s-w)}{\Gamma (1-s+w)}.
\end{align*}

  Substituting this into the above expression for $h(\xi,w)$, we immediately obtain \eqref{eq: integral in lemma for h(xi,w)} by noticing that we can now take $c_s>\max\{0,\Re(w)\}$. This also implies that
$h(\xi,w)$ is an entire function of $w$ in $\Re(w)>-1$.

  It remains to establish \eqref{hbound}. For this, we let $c=\Re(s)$ and we may assume that $c \geq 2$ here. By apply Stirling's formula (see \cite[(5.112)]{iwakow}), we deduce that
\begin{align*}
 & \left(\frac{2^{5/2}}{\pi}\right)^{2s}
\left ( \frac {\Gamma(\frac{1}{2}+s)}{\Gamma(\frac{1}{2})} \right )^2  (\pi)^{-2s+2w}\frac{\Gamma (s-w)}{\Gamma (1-s+w)} \frac {1}{s} \\
\ll & \left(\frac{2^{5/2}}{\pi^2}\right)^{2c}(1+|s|)^{2c-1}e^{-\frac {\pi}{2}|\Im(s)|}(1+|s-w|)^{2c-2\Re(w)-1} \\
 \ll &
(\frac {|s|}{e^{1/2}})^{2c-1}e^{-\frac {\pi}{2}|\Im(s)|}(1+|s|)^{2c-2\Re(w)-1} (1+|w|)^{2c-2\Re(w)-1}.
\end{align*}
  It follows that
\begin{align}
\label{gbound}
\begin{split}
& \int\limits\limits_{(c)}\left(\frac{2^{5/2}}{\pi}\right)^{2s}
\left ( \frac {\Gamma(\frac{1}{2}+s)}{\Gamma(\frac{1}{2})} \right )^2(\pi)^{-2s+2w}\frac{\Gamma (s-w)}{\Gamma (1-s+w)} \frac {ds}{s}
\ll e^{-c}|c|^{4c-2\Re(w)-2}(1+|w|)^{2c-2\Re(w)-1}.
\end{split}
\end{align}
  We now take
\begin{align*}
\begin{split}
 c=\max (2, \frac {\xi^{1/4}}{X^{1/4}(1+|w|)^{1/2}})
\end{split}
\end{align*}
 to see that the bound given \eqref{hbound} follows.
\end{proof}

  Our next two lemmas provide bounds for certain dyadic sums involving $\mathcal{G}_0$ and $h(\xi,w)$.
\begin{lemma}\label{lem: version of Lemma 5.5 of Sound}
Let $K, J \geq 1$ be two integers and let $k_2$ be defined in \eqref{eq: defn of k1 and k2}. Then for $\Re(w)=-\frac{1}{2}+\varepsilon$ and any sequence of complex numbers $\delta_{\ell}$ satisfying $\delta_{\ell}\ll N(\ell)^{\varepsilon}$, we have
\begin{equation*}
\sum_{K\leq N(k)< 2K}\frac{1}{N(k_2)} \left| \sum_{\substack{N(\ell)=J \\ (\ell,2\alpha d)=1}}^{2J-1} \frac{\delta_{\ell}}{\sqrt{N(\ell)}} \mathcal{G}_0(1+w;k,\ell,\alpha, d)\right|^2 \ll_{\varepsilon} (N(\alpha d) JK)^{\varepsilon} J(J+K).
\end{equation*}
\end{lemma}
\begin{proof}
We write any $k \neq 0, k \in \mathcal O_K$ as $k=u_k \displaystyle \prod_{\varpi \in G, \ a_i\geq 1} \varpi_i^{a_i}$ with $u_k \in U_K$. For those $\varpi_i$ appearing in this product, we define
\begin{equation}\label{eq: defn of a(k) and b(k)}
a(k)=\prod_i \varpi_i^{a_i+1} \ \ \ \ \text{and} \ \ \ \ b(k)=\prod_{a_i=1} \varpi_i \prod_{a_i\geq 2}\varpi_i^{a_i-1}.
\end{equation}
 It follows from part (ii) of Lemma \ref{Gausssum} and the definition of $\mathcal{G}_0$ in Lemma~\ref{lem: nu-sum as an Euler product} that $\mathcal{G}_0(1+w;k,\ell,\alpha, d)=0$ unless $\ell=gm$ with $g|a(k)$, $(m,k)=1$ and $m$ square-free.  We then deduce from this and Lemma~\ref{Gausssum} that when $(\ell,2\alpha d)=1$,
\begin{equation*}
%%\label{eq: G0 in terms of g and m}
\mathcal{G}_0(1+w;k,\ell,\alpha, d)=\sqrt{N(m)}\left( \frac{ik}{m}\right) \prod_{\substack{\varpi \in G \\ \varpi |m}} \left( 1+ \frac{2}{N(\varpi)^{1+w}}\left( \frac{i k_1}{\varpi}\right)\right)^{-1} \mathcal{G}_{0}(1+w;k,g,\alpha, d).
\end{equation*}
 We apply the above and the Cauchy-Schwarz inequality to see that
\begin{equation*}
%%\label{eq: split G0 sum into 3|m and 3 not divides m}
\sum_{K\leq N(k)< 2K}\frac{1}{N(k_2)} \left| \sum_{\substack{N(\ell)=J\\ (\ell,2\alpha d)=1}}^{2J-1} \frac{\delta_{\ell}}{\sqrt{N(\ell)}} \mathcal{G}_0(1+w;k,\ell,\alpha, d)\right|^2 \ll_{\varepsilon} K^{\varepsilon} \sum_{K\leq N(k)< 2K}\frac{1}{N(k_2)}  \sum_{\substack{g|a(k) \\ N(g)<2J}} \Psi(k,g),
\end{equation*}
where
\begin{equation*}
\Psi(k,g) = \Bigg| \sum_{\substack{\frac{J}{N(g)}\leq  N(m) <\frac{2J}{N(g)} \\ (m,2\alpha d)=1}} \frac{\mu^2_{[i]}(m) \delta_{gm}}{\sqrt{N(g)}}  \mathcal{G}_{0}(1+w;k,g,\alpha, d) \left( \frac{i k}{m}\right) \prod_{\substack{\varpi \in G \\ \varpi |m}} \left( 1+ \frac{2}{N(\varpi)^{1+w}}\left( \frac{i k_1}{\varpi}\right)\right)^{-1} \Bigg|^2.
\end{equation*}
  Applying the bound for $\mathcal{G}_{0}$ in Lemma~\ref{lem: nu-sum as an Euler product} in the above expression, we obtain that
\begin{equation}\label{eq: Psi 1 after factoring out G0}
\Psi(k,g) \ll_{\varepsilon} N(\alpha d K)^{\varepsilon}N(g)^{1+\varepsilon}\Bigg| \sum_{\substack{\frac{J}{N(g)}\leq N(m)<\frac{2J}{N(g)} \\ (m,2\alpha d)=1 }} \mu^2_{[i]}(m) \delta_{gm}  \left( \frac{i k}{m}\right) \prod_{\substack{\varpi \in G \\ \varpi |m}} \left( 1+ \frac{2}{N(\varpi)^{1+w}}\left( \frac{i k_1}{\varpi}\right)\right)^{-1} \Bigg|^2.
\end{equation}
 Note that as $\left( \frac{i k}{m}\right)\neq 0$,
\begin{align*}
\prod_{\substack{\varpi \in G \\ \varpi |m}}  \left( 1+ \frac{2}{N(\varpi)^{1+w}}\left( \frac{i k_1}{\varpi }\right)\right)^{-1}
& = \prod_{\substack{\varpi \in G \\ \varpi |m}}  \left( 1- \frac{4}{N(\varpi)^{2+2w}}\right)^{-1} \prod_{\substack{\varpi \in G \\ \varpi |m}}  \left( 1- \frac{2}{N(\varpi)^{1+w}}\left( \frac{i k_1}{\varpi}\right)\right) \\
& = \prod_{\substack{\varpi \in G \\ \varpi |m}}  \left( 1- \frac{4}{N(\varpi)^{2+2w}}\right)^{-1}\sum_{\substack{j|m \\ j \equiv 1 \bmod {(1+i)^3}}} \frac{\mu_{[i]}(j) d_{[i],2}(j)}{N(j)^{1+w}} \left( \frac{i k_1}{j}\right).
\end{align*}
 Using this in \eqref{eq: Psi 1 after factoring out G0}, we see via another application of Cauchy-Schwarz that
\begin{equation*}
\Psi(k,g) \ll_{\varepsilon} N(\alpha d K)^{\varepsilon}N(g)^{1+\varepsilon} \sum_{N(j)<\frac{2J}{N(g)} } \Bigg| \sum_{\substack{\frac{J}{N(g)}\leq N(m)<\frac{2J}{N(g)} \\ (m, 2\alpha d)=1 \\ j|m }} \mu^2_{[i]}(m) \delta_{gm} \left( \frac{i k}{m}\right) \prod_{\substack{\varpi \in G \\ \varpi |m}}  \left( 1- \frac{4}{N(\varpi)^{2+2w}} \right)^{-1}\Bigg|^2.
\end{equation*}

  Relabelling $m$ by $jm$ while noting that $N(\varpi)>3$ for all $\varpi |m$, we deduce that for $\Re(w)\geq -\frac{1}{2}+\varepsilon$,
\begin{align*}
 \mu^2_{[i]}(j)\left( \frac{i k}{j}\right) \prod_{\substack{\varpi \in G \\ \varpi |j }} \left( 1- \frac{4}{N(\varpi)^{2+2w}} \right)^{-1} \ll N(j)^{\varepsilon}.
\end{align*}
  Using this, we see that
\begin{equation}\label{eq: after relabeling m as jm in Psi 1}
\Psi(k,g) \ll_{\varepsilon} N(\alpha d JK)^{\varepsilon}N(g)^{1+\varepsilon} \sum_{N(j)<\frac{2J}{N(g)} } \Bigg| \sum_{\substack{\frac{J}{N(gj)}\leq N(m)<\frac{2J}{N(gj)} \\ (m,2\alpha dj)=1}} \mu^2_{[i]}(m) \delta_{gjm} \left( \frac{i k}{m}\right) \prod_{\substack{\varpi \in G \\ \varpi |m}} \left( 1- \frac{4}{N(\varpi)^{2+2w}} \right)^{-1}\Bigg|^2.
\end{equation}

 Observe that $g|a(k)$ implies $b(g)|k$ by \eqref{eq: defn of a(k) and b(k)}. For such $k$, we relabel it as $f b(g)$ to obtain from \eqref{eq: after relabeling m as jm in Psi 1} that
\begin{align}
\label{eq: after relabeling k as fb(g)}
\begin{split}
 & \sum_{K\leq N(k)< 2K}\frac{1}{N(k_2)}  \sum_{\substack{g|a(k) \\ N(g)<2J}} \Psi(k,g) \\
\leq & \sum_{N(g)<2J} \sum_{\substack{K\leq N(k)< 2K \\ b(g)|k} }\frac{1}{N(k_2)}  \Psi(k,g)
= \sum_{N(g)<2J}\sum_{\frac{K}{N(b(g))}\leq N(f)< \frac{2K}{N(b(g))}} \frac{1}{N(k_2)}  \Psi(f b(g),g) \\
\ll_{\varepsilon} &  N(\alpha d J K)^{\varepsilon}\sum_{N(g)<2J}N(g)^{1+\varepsilon} \sum_{\frac{K}{N(b(g))}\leq N(f)< \frac{2K}{N(b(g))}} \frac{1}{N(k_2)} \\
& \times \sum_{N(j)<\frac{2J}{N(g)} } \Bigg| \sum_{\substack{\frac{J}{N(gj)}\leq N(m)<\frac{2J}{N(gj)} \\ (m,2\alpha dj)=1}} \mu^2_{[i]}(m) \delta_{gjm} \left( \frac{i f b(g)}{m}\right) \prod_{\substack{\varpi \in G \\ \varpi |m}} \left( 1- \frac{4}{N(\varpi)^{2+2w}} \right)^{-1}\Bigg|^2 .
\end{split}
\end{align}

 We write $f=f_1f_2^2$ with $f_1$ square-free and $f_2 \in G$ to see that the relation $f b(g) d_1=k_1k_2^2$ implies that $f_2| k_2$, so that $N(k_2)^{-1}\ll N(f_2)^{-1}$. Applying this in \eqref{eq: after relabeling k as fb(g)}, we see that the assertion of the lemma follows from Lemma~\ref{lem: estimates for character sums}.
\end{proof}

\begin{lemma}\label{lem: version of Lemma 5.4 of Sound}
Let $N(\alpha)\leq Y$ and let $K, J \geq 1$ be two integers. Then for $\Re(w)=-\frac{1}{2}+\varepsilon$ and any $\gamma_{\ell} \in \mc$ satisfying $|\gamma_{\ell}|\leq 1$,
\begin{align}
\label{Gklsum}
\sum_{\substack{ K\leq N(k)<2K  }}\frac{1}{N(k_2)} \Bigg| \sum_{\substack{N(\ell)=J \\ (\ell,2\alpha d)=1}}^{2J-1} \frac{\gamma_{\ell}}{N(\ell)}  \mathcal{G}_0(1+w;k,\ell,\alpha,d) h(\frac{N(k)X}{2N(\alpha^2d_1\ell)}, w) \Bigg|^2
\end{align}
is bounded by
\begin{align*}
\ll_{\varepsilon} |\widehat{\Phi}(1+w)|^2 (1+|w|)^{8+\varepsilon}  \frac{N(d_1)^2N(\alpha)^{2+\varepsilon}J^{2+\varepsilon}K^{\varepsilon}N(d)^{\varepsilon} }{X^{1-\varepsilon}} \exp\Bigg( -\frac{1}{20}\frac{\sqrt[4]{K}}{\sqrt[4]{N(\alpha)^2N(d_1) J(1+|w|^2)}}\Bigg),
\end{align*}
and also by
\begin{align*}
\ll_{\varepsilon} ((1+|w|)N(\alpha d) JKX)^{\varepsilon}|\widehat{\Phi}(1+w)|^2  \frac{N(\alpha)^2 N(d_1)(JK+J^2)}{KX}.
\end{align*}
\end{lemma}
\begin{proof}
We apply Lemma~\ref{lem: nu-sum as an Euler product} and Lemma~\ref{lem: properties of h(xi,w)} to bound respectively $\mathcal{G}_0$ and $h(\xi,w)$ to see that the expression in \eqref{Gklsum} is
\begin{align*}
\ll & |\widehat{\Phi}(1+w)|^2(1+|w|)^{8+\varepsilon}\frac{N(d_1)N(\alpha)^{2+\varepsilon}J^{\varepsilon}K^{\varepsilon}N(d)^{\varepsilon} }{X^{1-\varepsilon}} \exp\Bigg( -\frac{1}{20}\frac{\sqrt[4]{K}}{\sqrt[4]{N(\alpha)^2N(d_1) J(1+|w|^2)}}\Bigg) \\
& \times \sum_{\substack{ K\leq N(k)<2K }}\frac{1}{N(k)N(k_2)} \Bigg(\sum_{\substack{N(\ell)=J \\ (\ell,2\alpha d)=1}}^{2J-1} N((\ell,k_4^2))^{\frac{1}{2}} \Bigg)^2 \\
\ll & |\widehat{\Phi}(1+w)|^2(1+|w|)^{8+\varepsilon}\frac{N(d_1)^2N(\alpha)^{2+\varepsilon}J^{\varepsilon}K^{\varepsilon}N(d)^{\varepsilon} }{X^{1-\varepsilon}} \exp\Bigg( -\frac{1}{20}\frac{\sqrt[4]{K}}{\sqrt[4]{N(\alpha)^2N(d_1) J(1+|w|^2)}}\Bigg) \\
& \times \sum_{\substack{ K\leq N(k)<2K }}\frac{1}{N(kd_1)N(k_2)} \Bigg(\sum_{\substack{N(\ell)=J \\ (\ell,2\alpha d)=1}}^{2J-1} N(k_2) \Bigg)^2,
\end{align*}
 where the last estimation above follows from the observation that $N((\ell,k_4^2)) \leq N(k_4)^2\leq N(k_2)^2$. By further writing $N(kd_1)=N(k_1k^2_2)$, we obtain the first bound of the lemma from the above estimation.

 We now derive the second bound by setting $c=\varepsilon$ to write the integral \eqref{eq: integral in lemma for h(xi,w)} as
\begin{align*}
\frac{1}{2\pi i} \int\limits_{(\varepsilon)} g(s,w)\left(\frac{X}{\xi}\right)^s\,ds.
\end{align*}
 This allows us to see that
\begin{align*}
& \Bigg| \sum_{\substack{N(\ell)=J \\ (\ell,2\alpha d)=1}}^{2J-1} \frac{\gamma_{\ell}}{N(\ell)}  \mathcal{G}_0(1+w;k,\ell,\alpha,d) h(\frac{N(k)X}{2N(\alpha^2d_1\ell)}, w) \Bigg| \\
 \ll & |\widehat{\Phi}(1+w)|\left(\frac{N(\alpha)^{1+\varepsilon} N(d_1)^{\frac{1}{2}+\varepsilon} }{N(k)^{\frac{1}{2}-\varepsilon}X^{\frac{1}{2}-\varepsilon}}\right)\int\limits_{(\varepsilon)}\Bigg| g(s,w)\sum_{\substack{N(\ell)=J \\ (\ell,2\alpha d)=1}}^{2J-1}  \frac{\gamma_{\ell}}{N(\ell)^{1+w-s}} \mathcal{G}_0(1+w;k,\ell,\alpha,d)   \Bigg|\,|ds|.
\end{align*}
 Note that \eqref{gbound} is still valid when $c=\varepsilon$ and $\Re(w)=-\frac{1}{2}+\varepsilon$ so that it implies that $g(s,w; k)\ll_{\varepsilon}(1+|w|)^{\varepsilon}\exp(-(\frac{\pi}{2}-\varepsilon) |\Im(s)|)$. We then deduce via the Cauchy-Schwarz inequality that
\begin{align*}
\begin{split}
& \Bigg| \sum_{\substack{\ell=J \\ (\ell,2\alpha d)=1}}^{2J-1}  \frac{\gamma_{\ell}}{N(\ell)}  \mathcal{G}_0(1+w;k,\ell,\alpha,d) h(\frac{N(k)X}{2N(\alpha^2d_1\ell)}, w)\Bigg|^2 \\
\ll &  (1+|w|)^{\varepsilon}|\widehat{\Phi}(1+w)|^2 \left(\frac{N(\alpha)^{2+\varepsilon} N(d_1)^{1+\varepsilon} }{N(k)^{1-\varepsilon}X^{1-\varepsilon}}\right)  \int\limits_{(\varepsilon)}\exp(-(\tfrac{\pi}{2}-\varepsilon)|\Im(s)|)\Bigg| \sum_{\substack{N(\ell)=J \\ (\ell,2\alpha d)=1}}^{2J-1}  \frac{\gamma_{\ell}}{N(\ell)^{1+w-s}} \mathcal{G}_0(1+w;k,\ell,\alpha,d)   \Bigg|^2\,|ds|.
\end{split}
\end{align*}
 Inserting this into \eqref{Gklsum} and applying Lemma~\ref{lem: version of Lemma 5.5 of Sound}, we readily deduce the second bound of the lemma.
\end{proof}

\section{The mollified first moment}
\label{sec:mollified first moment}

  In this section, we prove Proposition \ref{prop: asymptotic for S1}.  By applying Lemma \ref{prop: asymptotic for S1} and the definition of $M(\varpi)$ in \eqref{eq: defn of mollifier} in the expression of $S_1$ in \eqref{eq: defn moment sums S1 and S2}, we see that
 \begin{align*}
S_1 &= 2\sum_{\substack{ N(m) \leq M \\ m \equiv 1 \bmod {(1+i)^3}}} \frac{b_m}{\sqrt{N(m)}}\sum_{\substack{n \equiv 1 \bmod {(1+i)^3}}} \frac{1}{\sqrt{N(n)}}   \sum_{(\varpi,2)=1  } \log N(\varpi) \Phi\left(\frac{N(\varpi)}{X}\right)V_1\left( \frac{N(n)}{\sqrt{N(\varpi)}} \right) \leg{(1+i)^5\varpi}{mn} \\
 =& S_1^\square + S_1^{\neq},
\end{align*}
where
\begin{align}\label{eq: splitting up first moment into diagonal and off diagonal}
\begin{split}
S_1^\square =& 2 \sum_{\substack{ N(m) \leq M \\ m \equiv 1 \bmod {(1+i)^3}}}\sum_{\substack{n \equiv 1 \bmod {(1+i)^3}\\ mn = \square}} \frac{1}{\sqrt{N(n)}} \frac{b_m}{\sqrt{N(m)}}  \sum_{(\varpi,2)=1 } \log N(\varpi) \Phi\left(\frac{N(\varpi)}{X}\right)V_1\left( \frac{N(n)}{\sqrt{N(\varpi)}} \right) \leg{(1+i)^5\varpi}{mn}, \\
S_1^{\neq} =& 2\sum_{\substack{ N(m) \leq M \\ m \equiv 1 \bmod {(1+i)^3}}}\sum_{\substack{n \equiv 1 \bmod {(1+i)^3}\\ mn \neq \square}} \frac{1}{\sqrt{N(n)}}  \frac{b_m}{\sqrt{N(m)}}  \sum_{(\varpi,2)=1  } \log N(\varpi) \Phi\left(\frac{N(\varpi)}{X}\right)V_1\left( \frac{N(n)}{\sqrt{N(\varpi)}} \right) \leg{(1+i)^5\varpi}{mn} .
\end{split}
\end{align}

\subsection{Evaluation of $S_1^\square$}

As $b_m$ is supported on square-free elements in $\mathcal O_K$, we see that $mn = \square$ if and only if $n = mk^2$ with $k \in \mathcal O_K$. By making a change of variable $n=mk^2$ with $k$ being primary, we deduce that
\begin{align*}
S_1^{\square} = 2 \sum_{(\varpi,2)=1  } \log N(\varpi) \Phi\left(\frac{N(\varpi)}{X}\right) \sum_{\substack{N(m) \leq M\\ m \equiv 1 \bmod {(1+i)^3} \\ (m,\varpi)=1}} \frac{b_m}{N(m)} \sum_{\substack{k \equiv 1 \bmod {(1+i)^3}  \\ (k,\varpi)=1 }} \frac{1}{N(k)}  V_1\left( \frac{N(mk^2)}{\sqrt{N(\varpi)}}\right) .
\end{align*}
 The rapid decay of $V_1$ given in Lemma \ref{lem: properties of omega j} implies that the contribution from those $k$ with $(k,\varpi) \neq 1$ is $O_A(X^{-A})$ for any large number $A$. Moreover, the condition $(m,\varpi)=1$ is automatically satisfied as $N(m) \leq M < N(\varpi)$. We may thus ignore these two conditions and apply the definition \eqref{eq:Vdef} of $V_1(\xi)$ to see that
\begin{align*}
&\sum_{\substack{k \equiv 1 \bmod {(1+i)^3}  }} \frac{1}{N(k)}  V_1\left( \frac{N(mk^2)}{\sqrt{N(\varpi)}}\right) = \  \frac{1}{2 \pi i} \int\limits\limits_{(2)}  w_1(s)   \left( 1-\frac{1}{2^{1+2s}}\right) \zeta_K(1+2s) N(\varpi )^{s/2}N(m)^{-s}\, \frac {ds}{s} .
\end{align*}
  Now we note the following convexity bounds for $\zeta_K(s)$ and $\zeta_K'(s)$ for $0< \Re(s)<1$:
 \begin{align}
\label{zetaconvexitybounds}
\begin{split}
 \zeta_K(s), \ \zeta_K'(s) \ll (1+|s|^2)^{\frac {1-\Re(s)}{2}+\varepsilon} .
\end{split}
\end{align}
  Here the bound for $\zeta_K(s)$ follows from \cite[Exercise 3, p. 100]{iwakow} and the bound for $\zeta'_K(s)$ can be obtained via a similar convexity principle.

 By moving the line of integration to $\Re(s)=-\frac{1}{2}+\varepsilon$, we apply \eqref{zetaconvexitybounds} along with the rapid decay of the gamma function in vertical strips to see that the integral on this line is $\ll_{\varepsilon}\left(N(\varpi)^{-\frac{1}{4}+\varepsilon} N(m)^{\frac{1}{2}-\varepsilon}\right)$ and this contributes an error term of size $\ll \ X^{\frac{3}{4}+\varepsilon} M^{\frac{1}{2}} =O(X^{1-\varepsilon})$ to $S_1^{\square}$ by noticing the definition of $M$ in \eqref{eq:outline section, defn of M, length of mollifier} and the estimation that $b_m\ll 1$.

 We also obtain a contribution from the residue of a pole at $s=0$, which we write as an integral along a small circle around $0$ to see that
\begin{align}\label{eq: first moment S 1 square}
\begin{split}
S_1^\square =& 2 \sum_{(\varpi,2)=1  } \log N(\varpi) \Phi\left(\frac{N(\varpi)}{X}\right) \sum_{\substack{N(m) \leq M\\ m \equiv 1 \bmod {(1+i)^3}}} \frac{b_m}{N(m)} \\
&\times \frac{1}{2\pi i} \oint\limits_{|s|=\frac{1}{2\log X}}  w_1(s)   \left( 1-\frac{1}{2^{1+2s}}\right) \zeta_K(1+2s) N(\varpi )^{s/2}N(m)^{-s}\, \frac {ds}{s}+O(X^{1-\varepsilon}).
\end{split}
\end{align}

 Notice that $b_m=\mu_{[i]}(m)H\left( \frac{\log N(m)}{\log M}\right)$ and we have by the Fourier inversion formula that
\begin{align}
\label{eq: write H as Fourier integral}
H(t) \ = \ \int_{-\infty}^{\infty} h(z) e^{-t(1+iz)}\,dz,
\end{align}
 where
\begin{align*}
%%\label{eq: write h as Fourier transform}
h(z) \ = \frac 1{2\pi} \int_{-\infty}^{\infty} e^t H(t) e^{izt}\,dt.
\end{align*}
  Applying the above, we see that
\begin{align*}
& \sum_{\substack{N(m) \leq M\\ m \equiv 1 \bmod {(1+i)^3}}} \frac{b_m}{N(m)^{1+s}}\ = \ \int_{-\infty}^{\infty} h(z) \sum_{\substack{m \equiv 1 \bmod {(1+i)^3}}}\frac{\mu_{[i]}(m)}{N(m)^{1+s+\frac{1+iz}{\log M}}}\,dz \\
  =&  \ \int_{-\infty}^{\infty} h(z)\left(1-\frac{1}{2^{1+s+\frac{1+iz}{\log M}}} \right)^{-1}  \zeta_K^{-1}\left( 1+s+\frac{1+iz}{\log M}\right) \,dz.
\end{align*}
 Integration by parts shows that
\begin{align}\label{eq: h has rapid decay}
h(z) \ll_j \frac{1}{(1+|z|)^j},
\end{align}
 which implies that we have for any large number $A$,
\begin{align*}
\sum_{\substack{N(m) \leq M\\ m \equiv 1 \bmod {(1+i)^3}}} \frac{b_m}{N(m)^{1+s}} \ = \ \int\limits_{|z| \leq \sqrt{\log M}} &h(z)\left(1-\frac{1}{2^{1+s+\frac{1+iz}{\log M}}} \right)^{-1}  \zeta_K^{-1}\left( 1+s+\frac{1+iz}{\log M}\right) \,dz \ + \ O_A\left(\frac{1}{(\log X)^A}\right).
\end{align*}
For $|s|=\frac{1}{2\log X}$ and $|z|\leq \sqrt{\log M}$, we expand $\left(1-\frac{1}{2^{1+s+\frac{1+iz}{\log M}}} \right)^{-1}  \zeta_K^{-1}\left( 1+s+\frac{1+iz}{\log M}\right)$ into a power series (note that $\zeta_K(s)^{-1} \sim \frac 4{\pi}(s-1)$ when $s$ is near $1$) to obtain that
\begin{align*}
\sum_{\substack{N(m) \leq M\\ m \equiv 1 \bmod {(1+i)^3}}} \frac{b_m}{N(m)^{1+s}} & = \ \frac {8}{\pi} \int\limits_{|z| \leq \sqrt{\log M}} h(z)  \left( s+\frac{1+iz}{\log M}\right) \,dz \ \ + \ O\left( \frac{1}{(\log X)^{2}}\right).
\end{align*}
  Again by \eqref{eq: h has rapid decay}, we may extend back the integration above to all $z$ with an negligible error. Then we have
\begin{align*}
\int_{-\infty}^{\infty} h(z)  \left( s+\frac{1+iz}{\log M}\right) \,dz \ = \ sH(0)-\frac{1}{\log M} H'(0),
\end{align*}
 since by the expression for $H(t)$ given in \eqref{eq: write H as Fourier integral}, we have that
\begin{align*}
H'(t) \ = \ -(1+iz)\int_{-\infty}^{\infty} h(z) e^{-t(1+iz)}\,dz.
\end{align*}
 Thus we conclude that
\begin{align}\label{eq: first moment m sum main term}
\sum_{\substack{N(m) \leq M\\ m \equiv 1 \bmod {(1+i)^3}}} \frac{b_m}{N(m)^{1+s}} \ = \ \frac 8{\pi} \left ( sH(0)-\frac{1}{\log M} H'(0) \right ) \ \ + \ O\left( \frac{1}{(\log X)^2}\right).
\end{align}

 We apply \eqref{eq: first moment m sum main term} to \eqref{eq: first moment S 1 square} and use the definition of $\omega_1(s)$ given in \eqref{eq:Vdef} to see that
\begin{align*}
S_1^\square =&  \frac {16}{\pi} \sum_{(\varpi, 2)= 1 } \log N(\varpi) \Phi\left(\frac{N(\varpi)}{X}\right)  \\
&\times \frac{1}{2\pi i} \oint\limits_{|s|=\frac{1}{2\log X}}  w_1(s)   \left( 1-\frac{1}{2^{1+2s}}\right) \zeta_K(1+2s) N(\varpi )^{s/2} \left(sH(0)-\frac{1}{\log M} H'(0) \right) \,\frac{ds}{s}  + O\left( \frac{X}{\log X}\right)\\
=&   \frac {16}{\pi}\sum_{(\varpi, 2)= 1 } \log N(\varpi) \Phi\left(\frac{N(\varpi)}{X}\right)  \\
&\times \frac{1}{2\pi i} \oint\limits_{|s|=\frac{1}{2\log X}} \left(\frac{2^{5/2}}{\pi}\right)^{s}\left ( \frac {\Gamma(\frac{1}{2}+s)}{\Gamma(\frac{1}{2})} \right )   \left( 1-\frac{1}{2^{1+2s}}\right) \zeta_K(1+2s) N(\varpi )^{s/2} \left(sH(0)-\frac{1}{\log M} H'(0) \right) \,\frac{ds}{s}+ O\left( \frac{X}{\log X}\right).
\end{align*}
 Now, the integral above can be evaluated according to the formula
\begin{equation}\label{eq: residue as derivative}
\underset{s=0}{\mbox{Res}}\, g(s) = \frac{1}{(n-1)!} \frac{d^{n-1}}{ds^{n-1}} s^n g(s) \Bigg|_{s=0}
\end{equation}
for a function $g(s)$  having a pole of order at most $n$ at $s=0$. We then arrive at
\begin{align*}
S_1^\square &= \ \sum_{(\varpi, 2)= 1 } \log N(\varpi) \Phi\left(\frac{N(\varpi)}{X}\right) \ \left(H(0)-\frac{\log N(\varpi)}{2\log M} H'(0) \right) + O\left( \frac{X}{\log X}\right).
\end{align*}
 We may replace $\log N(\varpi)/\log M$ in the sum above by  $\log X/\log M$ in view of the support of $\Phi$. Then applying the prime ideal theorem and partial summation, we see that
\begin{align}\label{eq: S1 square gives main term}
S_1^\square &=4 \left(H(0)-\frac{\log X}{2\log M} H'(0) \right) \widehat{\Phi}(1)X + O\left( \frac{X}{\log X}\right).
\end{align}

\subsection{Decomposition of  $S_1^{\neq}$}

Our remaining task is to bound $S_1^{\neq}$. By writing $n = rk^2$ with $r$ being primary, square-free and $k$ primary, we see that the condition $mn\neq \square$ in \eqref{eq: splitting up first moment into diagonal and off diagonal} is equivalent to $m\neq r$, as both $m$ and $r$ are primary and square-free. This allows us to recast $S_1^{\neq}$ as
\begin{align*}
S_1^{\neq} =& 2\sum_{\substack{ N(m) \leq M \\ m \equiv 1 \bmod {(1+i)^3}}} \frac{b_m}{\sqrt{N(m)}} \sum_{\substack{ r \equiv 1 \bmod {(1+i)^3} \\ r \neq m}}  \sum_{\substack{ k \equiv 1 \bmod {(1+i)^3} }}  \frac{\mu^2_{[i]}(r)}{N(k)\sqrt{N(r)}} \\
& \times \sum_{(\varpi, 2)= 1 } \log N(\varpi) \Phi\left(\frac{N(\varpi)}{X}\right)V_1\left( \frac{N(rk^2)}{\sqrt{N(\varpi)}} \right) \leg{(1+i)^5\varpi}{mrk^2}.
\end{align*}

  We make changes of variables $m \rightarrow gm, r \rightarrow gr$ with $g=(m,r)$ to further recast $S_1^{\neq}$ as
\begin{align*}
S_1^{\neq} = 2 &\sum_{\substack{g \equiv 1 \bmod {(1+i)^3} }}\frac{\mu^2_{[i]}(g)}{N(g)} \sum_{\substack{N(m) \leq M/N(g) \\ m \equiv 1 \bmod {(1+i)^3} \\ (m,2g)=1}} \frac{b_{mg}}{\sqrt{N(m)}} \sum_{\substack{ r \equiv 1 \bmod {(1+i)^3} \\ (r,mg)=1 \\ N(mr)>1 }} \frac{\mu^2_{[i]}(r)}{\sqrt{N(r)}}\sum_{\substack{ k \equiv 1 \bmod {(1+i)^3} }}\frac{1}{N(k)} \\
& \times \sum_{(\varpi, 2)= 1 } \log N(\varpi) \Phi\left(\frac{N(\varpi)}{X}\right)V_1\left( \frac{N(grk^2)}{\sqrt{N(\varpi)}} \right) \leg{(1+i)^5\varpi}{mrg^2k^2}.
\end{align*}

We deduce from Lemma \ref{lem: properties of omega j} that we may truncate the sums over $k, r$ to $N(k) \leq X^{\frac{1}{4}+\varepsilon}$ and $N(r) \leq X^{\frac{1}{2}+\varepsilon}$ with negligible errors. Notice that this also implies that $N(k)<N(\varpi)$. As we also have $N(g) \leq M<N(\varpi)$, we conclude that we have $\left( \frac {\varpi}{g^2k^2}\right)=1$. We then extend the sum on $k$ to infinity again by Lemma \ref{lem: properties of omega j} to see that
\begin{equation*}
%%\label{eq: first moment S1 neq remove coprim conditions}
\begin{split}
S_1^{\neq} =& 2 \sum_{\substack{g \equiv 1 \bmod {(1+i)^3} }}\frac{\mu^2_{[i]}(g)}{N(g)} \sum_{\substack{N(m) \leq M/N(g) \\ m \equiv 1 \bmod {(1+i)^3} \\ (m,2g)=1}} \frac{b_{mg}}{\sqrt{N(m)}}\leg{1+i}{m} \sum_{\substack{ r \equiv 1 \bmod {(1+i)^3} \\ N(r) \leq X^{1/2+\varepsilon}\\ (r,mg)=1 \\ N(mr)>1 }} \frac{\mu^2_{[i]}(r)}{\sqrt{N(r)}}\leg{1+i}{r}\sum_{\substack{ k \equiv 1 \bmod {(1+i)^3} }}\frac{1}{N(k)} \\
& \times \sum_{(\varpi, 2)= 1  } \log N(\varpi) \Phi\left(\frac{N(\varpi)}{X}\right)V_1\left( \frac{N(grk^2)}{\sqrt{N(\varpi)}} \right) \chi_{mr}(\varpi)+O(X^{-1}).
\end{split}
\end{equation*}

   We may now write $\varpi=u_{\varpi} \varpi'$ with $u_{\varpi} \in U_K$. As the treatments are similar, we may further assume that $\varpi$ is primary, so that we can apply the quadratic reciprocity law \eqref{quadrec} with \eqref{eq:Vdef} to obtain that
\begin{align*}
%%\label{eq: first moment S1 neq remove coprim conditions1}
\begin{split}
S_1^{\neq}
=& 2 \sum_{\substack{g \equiv 1 \bmod {(1+i)^3} }}\frac{\mu^2_{[i]}(g)}{N(g)} \sum_{\substack{N(m) \leq M/N(g) \\ m \equiv 1 \bmod {(1+i)^3} \\ (m,2g)=1}} \frac{b_{mg}}{\sqrt{N(m)}}\leg{1+i}{m} \sum_{\substack{ r \equiv 1 \bmod {(1+i)^3} \\ N(r) \leq X^{1/2+\varepsilon}\\ (r,mg)=1 \\ N(mr)>1 }} \frac{\mu^2_{[i]}(r)}{\sqrt{N(r)}}\leg{1+i}{r}\sum_{\substack{ k \equiv 1 \bmod {(1+i)^3} }}\frac{1}{N(k)} \\
& \times \sum_{\varpi \equiv 1 \bmod {(1+i)^3} } \log N(\varpi) \Phi\left(\frac{N(\varpi)}{X}\right) \int\limits_{(\frac 1{\log X})} \left(\frac{2^{5/2}}{\pi}\right)^{s}\left ( \frac {\Gamma(\frac{1}{2}+s)}{\Gamma(\frac{1}{2})} \right )  \left( \frac{N(grk^2)}{\sqrt{N(\varpi)}} \right)^{-s} \chi_{mr}(\varpi)\frac {ds}{s}+O(X^{-1}) \\
=& 2 \sum_{\substack{g \equiv 1 \bmod {(1+i)^3} }}\frac{\mu^2_{[i]}(g)}{N(g)} \sum_{\substack{N(m) \leq M/N(g) \\ m \equiv 1 \bmod {(1+i)^3} \\ (m,2g)=1}} \frac{b_{mg}}{\sqrt{N(m)}}\leg{1+i}{m} \sum_{\substack{ r \equiv 1 \bmod {(1+i)^3} \\ N(r) \leq X^{1/2+\varepsilon}\\ (r,mg)=1 \\ N(mr)>1 }} \frac{\mu^2_{[i]}(r)}{\sqrt{N(r)}}\leg{1+i}{r} \\
& \times  \int\limits_{(\frac 1{\log X})} \left(\frac{2^{5/2}}{\pi}\right)^{s}\left ( \frac {\Gamma(\frac{1}{2}+s)}{\Gamma(\frac{1}{2})} \right )\left( 1-\frac{1}{2^{1+2s}}\right) \zeta_K(1+2s) N(gr)^{-s} \\
& \times \sum_{\varpi \equiv 1 \bmod {(1+i)^3} } \log N(\varpi) \Phi\left(\frac{N(\varpi)}{X}\right)N(\varpi)^{s/2}\chi_{mr}(\varpi)\frac {ds}{s}+O(X^{-1}).
\end{split}
\end{align*}

   We denote $\Lambda_{[i]}(n)$ for the von Mangoldt function on $K$, so that $\Lambda_{[i]}(n)$ equals the coefficient of $N(n)^{-s}$ in the Dirichlet series expansion of $\zeta^{'}_K(s)/\zeta_K(s)$.  Then we have
\begin{align*}
\sum_{\varpi \in G } \log N(\varpi) \Phi\left(\frac{N(\varpi)}{X}\right) \chi_{mr}(\varpi) N(\varpi)^{s/2} = \sum_{n \in G} \Lambda_{[i]}(n) \Phi \left( \frac{N(n)}{X}\right) \chi_{ mr}(n) N(n)^{s/2} + O(X^{1/2}).
\end{align*}
   It is easy to check that the contribution of the error term above to $S_1^{\neq}$ is $O(X^{1-\varepsilon})$ for sufficiently small $\varepsilon = \varepsilon(\theta) >0$. Now, we define for any function $g(t)$ and any complex number $s$,
\begin{align*}
g_s(t) = g(t) t^{s/2}.
\end{align*}
  Using this notation, we have
\begin{equation}\label{eq: first moment S1 neq all done massaging}
\begin{split}
S_1^{\neq} =& 2 \sum_{\substack{g \equiv 1 \bmod {(1+i)^3} }}\frac{\mu^2_{[i]}(g)}{N(g)} \sum_{\substack{N(m) \leq M/N(g) \\ m \equiv 1 \bmod {(1+i)^3} \\ (m,2g)=1}} \frac{b_{mg}}{\sqrt{N(m)}}\leg{1+i}{m} \sum_{\substack{ r \equiv 1 \bmod {(1+i)^3} \\ N(r) \leq X^{1/2+\varepsilon}\\ (r,mg)=1 \\ N(mr)>1 }} \frac{\mu^2_{[i]}(r)}{\sqrt{N(r)}}\leg{1+i}{r} \\
& \times \frac{1}{2\pi i} \displaystyle \int_{\frac{1}{\log X} - i(\log X)^2}^{\frac{1}{\log X} + i(\log X)^2} \left(\frac{2^{5/2}}{\pi}\right)^{s}\left ( \frac {\Gamma(\frac{1}{2}+s)}{\Gamma(\frac{1}{2})} \right )\left( 1-\frac{1}{2^{1+2s}}\right) \zeta_K(1+2s)X^{s/2} N(gr)^{-s} \\
& \times \sum_{n \in G} \Lambda_{[i]}(n) \Phi_s \left( \frac{N(n)}{X}\right) \chi_{ mr}(n) \frac {ds}{s}+O(X^{1-\varepsilon}),
\end{split}
\end{equation}
  where we truncate the integral to $|{\Im}(s)| \leq (\log X)^2$ due to the rapid decay of the gamma function in vertical strips. We further split the first expression on the right-hand side above into a sum of two terms:
\begin{align}\label{eq: first moment decompose S1 neq into E1 and E2}
S_1^{\neq} = E_1 + E_2+O(X^{1-\varepsilon}),
\end{align}
  with $E_1$ restricting the sums over $m, r$ to $N(mr) \ll \exp(w \sqrt{\log X})$ and $E_2$ the opposite, where $w >0$  is a fixed sufficiently small constant.

\subsection{Evaluation of $E_1$}

 In this section we estimate $E_1$, which is given by
\begin{equation}\label{eq: first moment defn of E1}
\begin{split}
E_1 =& 2 \sum_{\substack{g \equiv 1 \bmod {(1+i)^3} }}\frac{\mu^2_{[i]}(g)}{N(g)} \sum_{\substack{N(m) \leq M/N(g) \\ m \equiv 1 \bmod {(1+i)^3} \\ (m,2g)=1}} \frac{b_{mg}}{\sqrt{N(m)}}\leg{1+i}{m} \sum_{\substack{ r \equiv 1 \bmod {(1+i)^3} \\ N(r) \leq X^{1/2+\varepsilon}\\ (r,mg)=1 \\ 1 < N(mr) \ll \exp(w \sqrt{\log X}) }} \frac{\mu^2_{[i]}(r)}{\sqrt{N(r)}}\leg{1+i}{r} \\
& \times \frac{1}{2\pi i} \int_{\frac{1}{\log X} - i(\log X)^2}^{\frac{1}{\log X} + i(\log X)^2} \left(\frac{2^{5/2}}{\pi}\right)^{s}\left ( \frac {\Gamma(\frac{1}{2}+s)}{\Gamma(\frac{1}{2})} \right )\left( 1-\frac{1}{2^{1+2s}}\right) \zeta_K(1+2s)X^{s/2} N(gr)^{-s} \\
& \times \sum_{n \in G} \Lambda_{[i]}(n) \Phi_s \left( \frac{N(n)}{X}\right) \chi_{ mr}(n) \frac {ds}{s}.
\end{split}
\end{equation}

 By partial summation, we have that
\begin{align*}
%%\label{eq: first moment partial summation before pnt}
\sum_{n \in G} \Lambda_{[i]}(n) \Phi_s \left( \frac{N(n)}{X}\right) \chi_{ mr}(n) &= -\int_0^{\infty} \frac{1}{X}\Phi_s'\left(\frac{u}{X}\right) \Bigg(\sum_{\substack{n \in G \\ N(n) \leq u}}\Lambda_{[i]}(n) \chi_{mr}(n)\Bigg) \,du.
\end{align*}

  Combining \cite[Theorem 5.13]{iwakow} and \cite[Theorem 5.35]{iwakow} together with \cite[(5.52)]{iwakow}, by noting that the conductor of the primitive character $\chi_{mr}$ is $\ll \exp(w\sqrt{\log X}) \leq \exp(2w \sqrt{\log X})$, we see that
\begin{align}\label{eq: first moment davenport pnt}
\sum_{\substack{n \in G \\ N(n) \leq u}}\Lambda_{[i]}(n) \chi_{mr}(n) = -\frac{u^{\beta_1}}{\beta_1} +O\left(u \exp(w \sqrt{\log X})\exp(-c_1 \sqrt{\log u}) \right),
\end{align}
  for an absolute constant $c_1>0$. Here the term $-u^{\beta_1}/\beta_1$ appears only when $L(s,\chi_{mr})$ has a real zero $\beta_1 > 1- \frac{c_2}{\log N(mr)}$ for a positive constant $c_2$.

  Notice that for $\Re(s)$ bounded, we have uniformly in $s$ that
\begin{align*}
%%\label{eq: first moment bound on f s '}
\int_0^{\infty } \frac{1}{X}\left| \Phi_s'\left( \frac{u}{X}\right)\right|\,du \ = \ \int_0^{\infty} |\Phi_s'(u)|\,du \ \ll \ |s|+1.
\end{align*}

  Applying the above estimation, we see from \eqref{eq: first moment defn of E1} that the contribution of the error term to $E_1$ is
\begin{align}\label{eq: first moment regime I ordinary zeros}
\ll X \exp(c_3(w-c_1)\sqrt{\log X})
\end{align}
 for some absolute constant $c_3 > 0$.

By an analogue of Page's theorem for the family of quadratic Hecke $L$-functions (which can be established by using combining the arguments in \cite[\S 14, page 95]{Da}, \cite[Theorem 5.28 (1)]{iwakow} and \cite[Lemma 5.9]{iwakow}), there exists a fixed absolute constant $c_4>0$ such that we have at most one character $\chi_{mr}$ (notice as shown above, the conductor of $\chi_{mr}$ is $\leq \exp(2w \sqrt{\log X})$) for which the $L$-function $L(s,\chi_{ mr})$ has a real zero $\beta_1$ satisfying
\begin{align*}
%%\label{eq: first moment lower bound for exceptional beta 1}
\beta_1 > 1 - \frac{c_4}{2w \sqrt{\log X}}.
\end{align*}

  We may assume such a real zero exists and denote $q^*$ for the modulus of the exceptional character $\chi_{mr}$. It remains to estimate the contribution of the term $ -\frac{u^{\beta_1}}{\beta_1}$ in \eqref{eq: first moment davenport pnt} to $E_1$. For this, we apply integrate by parts to see that
\begin{align*}
\int_0^{\infty} \frac{u^{\beta_1}}{\beta_1}\frac{1}{X}\Phi_s'\left(\frac{u}{X}\right) \,du= X^{\beta_1}\int_0^{\infty} \frac{u^{\beta_1}}{\beta_1}\Phi_s'(u) \,du \ = \ -X^{\beta_1}\int_0^{\infty}\Phi_s(u)u^{\beta_1-1}\,du \ = \ -X^{\beta_1} \widehat \Phi \left( \frac{s}{2}+\beta_1\right).
\end{align*}

 Now, by choosing $w > 0$ sufficiently small in terms of $c_1$ in \eqref{eq: first moment regime I ordinary zeros} and applying the above to  \eqref{eq: first moment defn of E1}, we see that for some positive constant $c_5$ and some bounded power of two denoted by $\gamma^*$,
\begin{align*}
\label{eq: first moment exception E 1}
\begin{split}
E_1
=& -\frac{2}{2\pi i}  \frac{\sqrt{N(\gamma^*)} X^{\beta_1}}{\sqrt{N(q^*)}} \int_{\frac{1}{\log X} -i (\log X)^2}^{\frac{1}{\log X} +i (\log X)^2} \left(\frac{2^{5/2}}{\pi}\right)^{s}\left ( \frac {\Gamma(\frac{1}{2}+s)}{\Gamma(\frac{1}{2})} \right )\left( 1-\frac{1}{2^{1+2s}}\right) \\
&\times X ^{s/2}\zeta_K(1+2s) \widehat  \Phi \left( \frac{s}{2}+\beta_1\right) \mathop{\sum\sum}_{\substack{1 < N(mr) \ll \exp(w \sqrt{\log X}) \\ (mr,2)=1 \\ (m,r)=1 \\ \gamma^* mr = q^*}}\frac{\mu^2_{[i]}(r)}{N(r)^s}\leg{1+i}{rm}\sum_{\substack{g \equiv 1 \bmod {(1+i)^3} \\ (g,mr)=1}} \frac{\mu^2_{[i]}(g)b_{mg}}{N(g)^{1+s}}\frac{ds}{s} \\
&+O\left(X\exp(-c_5\sqrt{\log X})\right).
\end{split}
\end{align*}

Using $b_{mg} = \mu_{[i]}(mg) H(\frac{\log N(mg)}{\log M})$ and applying Fourier inversion given in \eqref{eq: write H as Fourier integral}, we obtain
\begin{equation}\label{eq: first moment exceptional mollifier}
\begin{split}
\sum_{\substack{g \equiv 1 \bmod {(1+i)^3} \\ (g,mr)=1}} \frac{\mu^2_{[i]}(g)b_{mg}}{N(g)^{1+s}}\frac{ds}{s} \ &= \ \mu_{[i]}(m) \int_{-\infty}^{\infty} \frac{1}{N(m)^{\frac{1+iz}{\log M}}} h(z) \prod_{\substack{\varpi \in G \\ \varpi|2mr} } \left(1-\frac{1}{N(\varpi)^{1+s+\frac{1+iz}{\log M}}} \right)^{-1}  \zeta^{-1}_K\left( 1+s+\frac{1+iz}{\log M}\right) \,dz.
\end{split}
\end{equation}
 We may truncate the above integral in \eqref{eq: first moment exceptional mollifier} to $|z| \leq \sqrt{\log M}$  with a negligible error using \eqref{eq: h has rapid decay}. This leads to
\begin{equation}\label{eq: first moment exceptional E 1 opened mollifier}
\begin{split}
E_1 &= -2  \frac{\sqrt{N(\gamma^*)} X^{\beta_1}}{\sqrt{N(q^*)}} \mathop{\sum\sum}_{\substack{N(m) \leq M, N(r) \leq X^{\frac{1}{2}+\varepsilon} \\ (mr,2)=1 \\ (m,r)=1\\ \gamma^* mr=q^* }} \mu_{[i]}(m)\mu^2_{[i]}(r) \ \frac{1}{2\pi i} \int_{\frac{1}{\log X} -i (\log X)^2}^{\frac{1}{\log X} +i (\log X)^2}\left(\frac{2^{5/2}}{\pi}\right)^{s}\left ( \frac {\Gamma(\frac{1}{2}+s)}{\Gamma(\frac{1}{2})} \right ) N(r)^{-s}\leg{1+i}{rm}  \\
& \times  \left( 1-\frac{1}{2^{1+2s}}\right)\zeta_K(1+2s) X^{s/2} \widehat \Phi \left( \frac{s}{2}+\beta_1\right) \,\int\limits_{|z| \leq \sqrt{\log M}} \frac{1}{N(m)^{\frac{1+iz}{\log M}}} \ h(z) \\
&\times \prod_{\substack{\varpi \in G \\ \varpi|2mr} } \left(1-\frac{1}{N(\varpi)^{1+s+\frac{1+iz}{\log M}}} \right)^{-1}  \zeta^{-1}_K\left( 1+s+\frac{1+iz}{\log M}\right) \,dz \frac{ds}{s} + O\left(\frac{X}{\log X}\right).
\end{split}
\end{equation}

 We move line of integration over $s$ in \eqref{eq: first moment exceptional E 1 opened mollifier} to $\Re(s) = -\frac{c_6}{\log \log X}$, for some small $c_6>0$ such that there is no zero of $\zeta_K(1+s+\frac{1+iz}{\log M})$ in the region $\Re(s) \geq -\frac{c_6}{\log \log X}, \Im(s) \leq (\log X)^2$. The contribution to $E_1$ of the integration over $s$  on the new line of integration is $O(X/\log X)$. There is also a contribution from the reside at $s=0$, which we write as an integral along a small circle around $0$ to see that
\begin{equation}\label{eq: first moment exceptional E 1 moved contour}
\begin{split}
E_1 &= -2  \frac{\sqrt{N(\gamma^*)} X^{\beta_1}}{\sqrt{N(q^*)}} \mathop{\sum\sum}_{\substack{N(m) \leq M, N(r) \leq X^{\frac{1}{2}+\varepsilon} \\ (mr,2)=1 \\ (m,r)=1\\ \gamma^* mr=q^* }} \mu_{[i]}(m)\mu^2_{[i]}(r) \leg{1+i}{rm} \ \frac{1}{2\pi i}   \oint\limits_{|s| = \frac{1}{\log X}}\left(\frac{2^{5/2}}{\pi}\right)^{s}\left ( \frac {\Gamma(\frac{1}{2}+s)}{\Gamma(\frac{1}{2})} \right ) N(r)^{-s} \\
& \times  \zeta_K(1+2s) X^{s/2} \widehat \Phi \left( \frac{s}{2}+\beta_1\right) \,\int\limits_{|z| \leq \sqrt{\log M}} \frac{1}{N(m)^{\frac{1+iz}{\log M}}} \ h(z) \\
&\times \prod_{\substack{\varpi \in G\\ \varpi|2mr} } \left(1-\frac{1}{N(\varpi)^{1+s+\frac{1+iz}{\log M}}} \right)^{-1}  \zeta^{-1}_K\left( 1+s+\frac{1+iz}{\log M}\right) \,dz \frac{ds}{s} + O\left(\frac{X}{\log X}\right).
\end{split}
\end{equation}

   Notice that when $|s| = \frac{1}{\log X}$, we have
\begin{align*}
\zeta_K(1+2s) \ll \log X, \ \ \ \ \ \ \zeta^{-1}_K\left( 1+s+\frac{1+iz}{\log M}\right) \ll \frac{1+|z|}{\log X}.
\end{align*}
  Applying the above bounds in \eqref{eq: first moment exceptional E 1 moved contour} and estimating things trivially, we deduce that
\begin{align}
\label{E1bound}
  E_1 \ll \frac{X^{\beta_1}}{N(q^*)^{1/2-\varepsilon}}+O\left(\frac{X}{\log X}\right).
\end{align}

   Now we need an upper bound for $\beta_1$. For this, we recall a result of Landau \cite{Landau} says that for an algebraic number field $F$ of degree $n=2$ and any primitive ideal character $\chi$ of $F$ with conductor $q$, we have for $X > 1$,
\begin{align*}
%%\label{PVnf}
  \sum_{N_F(I) \leq X} \chi(I) \ll |N_F(q)\cdot D_F|^{1/3}\log^2( |N_F(q)\cdot D_F|)X^{1/3},
\end{align*}
where $N_F(q), N_F(I)$ denotes the norm of $q$ and $I$ respectively, $D_F$ denotes the discriminant of $F$ and $I$ runs over integral ideas of $F$.

  It follows from this that we have, for any Hecke character $\chi$ modulo $q$ of trivial infinite type in $K$,
\begin{align}
\label{L1bound}
  L(1, \chi)=\sum_{\substack{ I \neq 0 \\ N(I) \leq N(q) }}\frac {\chi(I)}{N(I)}+\sum_{\substack{ N(I) > N(q) }}\frac {\chi(I)}{N(I)} \ll
  \sum_{\substack{ I \neq 0 \\ N(I) \leq N(q) }}\frac {1}{N(I)}+\int^{\infty}_{N(q)}\frac 1u d  ( \sum_{N(I) \leq u} \chi(I) )
  \ll \log N(q).
\end{align}
   Similarly, we have that
\begin{align}
\label{L1primebound}
  L'(1, \chi) \ll \log^2 N(q).
\end{align}

   We then deduce from \eqref{L1bound}, \eqref{L1primebound} and the proof of \cite[Theorem 5.28 (2)]{iwakow} that we have the following analogue in $K$ of Siegel's theorem, i.e. for any primitive
quadratic Hecke character $\chi$ modulo $q$ of trivial infinite type, we have that
\begin{align}\label{eq: first moment upper bound for beta 1}
\beta_1 < 1 - \frac{c_7(\varepsilon)}{N(q)^{\varepsilon}},
\end{align}
where $c_7(\varepsilon) > 0$ is an ineffective constant depending only on $\varepsilon$.

   Applying \eqref{eq: first moment upper bound for beta 1} in \eqref{E1bound}, then treating the resulting upper bound of $E_1$ according to whether $N(q^*)$ is $\leq (\log X)^3$ or not, we see that
\begin{align}
\label{eq: first moment bounds for E1}
  E_1 \ll \frac{X^{1 - \frac{c_7(\varepsilon)}{N(q)^{\varepsilon}}}}{N(q^*)^{1/2-\varepsilon}}+O\left(\frac{X}{\log X}\right) \ll \frac{X}{\log X}.
\end{align}

\subsection{Evaluation of $E_2$}

In this section we estimate $E_2$. We let $q=mr$ in \eqref{eq: first moment S1 neq all done massaging} and we break $N(q)$ into dyadic segments to see that
\begin{align}
\label{E2bound}
E_2 &\ll (\log X)^{O(1)} \sum_{\substack{Q=2^j \\  \exp(w \sqrt{\log X}) \ll Q \ll M X^{1/2+\varepsilon}}} \mathcal{E}(Q),
\end{align}
where
\begin{align*}
\mathcal{E}(Q) = Q^{-\frac{1}{2}+\varepsilon} \sum_{\chi \in S(Q)} \left|\sum_{n \in G}  \Lambda_{[i]} (n) \Phi_{s_0} \left( \frac{n}{X}\right) \chi(n) \right|.
\end{align*}
Here $S(Q)$ is defined as in Lemma \ref{lem: estimates for character sums}, $s_0 \in \mc$ such that $\Re(s_0) = \frac{1}{\log X}$ and $|\Im(s_0)| \leq (\log X)^2$.

    We now employ zero-density estimates to estimate $\mathcal{E}(Q)$. For this, we write $\Phi_{s_0}$ using inverse Mellin transform to obtain that
\begin{align}
\label{psiintegral}
\sum_{n \in G} \Lambda_{[i]} (n) \, \Phi_{s_0}\left(\frac{N(n)}{X}\right) \chi(n) &= \frac{1}{2\pi i} \int\limits_{(2)} X^w \widehat \Phi \left(w+\frac{s_0}{2}\right) \left(-\frac{L'}{L}\left(w,\chi \right)\right) dw.
\end{align}

   Note that integration by parts implies that for every non-negative integer $j$, we have
\begin{align}
\label{eq: f dagger rapid decay}
\left| \widehat \Phi(\sigma + it + \frac{s_0}{2})\right| \ll_{\sigma,j} (\log X)^j \left(1+\left|t-\frac{\text{Im}(s_0)}{2}\right|\right)^{-j}.
\end{align}

  On the other hand, it follows from the functional equation \eqref{Lambda} that for any primitive Hecke character $\chi$ modulo $m$ of trivial infinite type in $K$, we have
\begin{align*}
%%\label{1.1}
  L(s, \chi)= W(\chi)(N(m))^{-1/2}(|D_K|N(m))^{(1-2s)/2}(2\pi)^{1-2s}\frac {\Gamma(s)}{\Gamma(1-s)}L(1-s, \chi).
\end{align*}
   Upon taking logarithmic derivative on both sides above, we see that
\begin{align*}
%%\label{1.1}
   \frac {L'(s, \chi)}{L(s, \chi)} = -\log (|D_K|N(m))-\log (2\pi)+\frac {\Gamma'(s)}{\Gamma(s)}-\frac {\Gamma'(1-s)}{\Gamma(1-s)}-\frac {L'(1-s, \chi)}{L(1-s, \chi)}.
\end{align*}
   Applying the estimation (see \cite[Theorem C.1]{MVa1}) that for $|s|>\delta, |\arg s|<\pi- \delta$,
\begin{align*}
%%\label{1.1}
   \frac {\Gamma'(s)}{\Gamma(s)}=\log s+O(\frac 1{|s|}),
\end{align*}
  we deduce that on the line $\Re(w) = -\frac{1}{2}$ we have
\begin{align*}
\left|\frac{L'}{L}(w,\chi) \right| \ll \log(N(q)|w|).
\end{align*}

 We now shift the line of integration in \eqref{psiintegral} to $\Re(w) = -\frac{1}{2}$ and we apply the above bound of $L'/L$ and \eqref{eq: f dagger rapid decay} to bound the integral on this line. Along with contributions from residues of poles at all zeros $\rho = \beta + i\gamma$ of $L(s, \chi)$ in the critical strip, we obtain that
\begin{align}
\label{Lambdasumbound}
\sum_{n \in G} \Lambda_{[i]} (n) \, \Phi_{s_0}\left(\frac{N(n)}{X}\right) \chi(n) &= \sum_{\substack{L(\rho,\chi) = 0 \\ 0 \leq \beta \leq 1}} X^\rho \widehat \Phi \left(\rho + \frac{s_0}{2}\right) + O \left(\frac{(\log X)^{O(1)}}{X^{1/2}} \right).
\end{align}

  To bound the right side of \eqref{eq: regime II, reduction to zero counting}, we introduce some notations here. For a primitive Hecke character $\chi$ modulo $q$ of trivial infinite type, let $N(T,\chi)$ denote the number of zeros of $L(s,\chi)$ in the rectangle
\begin{align*}
0 \leq \beta \leq 1, \ \ \ \ \ \ \  |\gamma | \leq T.
\end{align*}
 Then we have \cite[Theorem 5.8]{iwakow} for $T \geq 2$,
\begin{align}
\label{eq: bound on N T chi}
N(T,\chi) \ll T \log(N(q)T).
\end{align}
 We further define for $\frac{1}{2}\leq \alpha \leq 1$,
\begin{align*}
N(\alpha,Q,T) &= \sum_{\substack{\chi \in S(Q)}} N(\alpha,T,\chi),
\end{align*}
 where $N(\alpha,T,\chi)$ denotes the number of zeros $\rho = \beta +i\gamma$ of $L(s,\chi)$ in the rectangle
\begin{align*}
\alpha \leq \beta \leq 1, \ \ \ \ \ \ \ |\gamma| \leq T.
\end{align*}

  We note the following analogue of Heath-Brown's zero-density estimate for $L$-functions of quadratic Dirichlet characters \cite[Theorem 3]{DRHB} to see that for $\frac 12 < \alpha \leq 1$,
\begin{align}
\label{Nbound}
N(\alpha,Q,T) \ll (QT)^{\varepsilon}(T^{3/2}Q)^{(2-2\alpha)/(3/2-\alpha)}.
\end{align}

  In fact, the above estimation follows by modifying the proof of \cite[Corollary 1.6]{BGL} using the earlier mentioned large sieve result of K. Onodera \cite{Onodera} on quadratic residue symbols in the Gaussian field. More specifically, we replace the bound given in \cite[(5.4)]{BGL} by
\begin{align*}
%%\label{eq: Jutila zero density}
 (QTX)^{\varepsilon}TY^{1/2-\alpha}(QT)^{1/2}(X^{1/2}+Q^{1/2}).
\end{align*}
  Also, we replace the bound given in \cite[(5.5)]{BGL} by
\begin{align*}
%%\label{eq: Jutila zero density}
 (QY)^{\varepsilon}(QX^{1-2\alpha}+Y^{2-2\alpha}),
\end{align*}
 where $1 \leq X \leq Y$. Then by setting $X=Q, Y=(T^{3/2}Q)^{1/(3/2-\alpha)}$, we deduce \eqref{Nbound}.

  Now, combining \eqref{eq: f dagger rapid decay} and \eqref{eq: bound on N T chi}, we see that the contribution in \eqref{Lambdasumbound} to $\mathcal{E}(Q)$ from those $\rho$ with $|\gamma| > Q^{\varepsilon}$ is
\begin{align*}
\ll X (\log X)^{A}Q^{1+\varepsilon}Q^{-A \epsilon},
\end{align*}
 for any large number $A$. We compute the contribution of this bound and the error term in \eqref{Lambdasumbound} to $\mathcal{E}(Q)$ by noting that $Q \gg \exp(w \sqrt{\log X})$ to obtain that
\begin{align}\label{eq: regime II, reduction to zero counting}
\mathcal{E}(Q) \ll X\exp(-w \sqrt{\log X}) + Q^{-\frac{1}{2}+\varepsilon}\sum_{\substack{\chi \in S(Q)}} \ \sum_{\substack{L(\rho,\chi)=0 \\ 0 \leq \beta \leq 1 \\ |\gamma| \leq Q^{\varepsilon}}} X^\beta.
\end{align}

In \eqref{eq: regime II, reduction to zero counting}, we separate the zeros $\rho$ according to whether $\beta < \frac{1}{2}+\varepsilon_0$ or $\beta \geq \frac{1}{2}+\varepsilon_0$ for a suitable small $\varepsilon_0>0$ such that when $Q<X^{1-\varepsilon}$, we have by \eqref{eq: bound on N T chi},
\begin{align}
\label{eq: regime II bound for zeros less than 4 5}
Q^{-\frac{1}{2}+\varepsilon}\sum_{\substack{\chi \in S(Q)}} \ \sum_{\substack{L(\rho,\chi)=0 \\ 0 \leq \beta < \frac{1}{2}+\varepsilon_0 \\ |\gamma| \leq Q^{\varepsilon}}} X^\beta \ll X^{\frac{1}{2}+\varepsilon_0}Q^{1/2+\varepsilon}<X^{1-\varepsilon}.
\end{align}

For those zeros with $\beta \geq \beta_0$ we apply partial summations to see that
\begin{align}
\label{exceptionalzerobound}
\begin{split}
& Q^{-\frac{1}{2}+\varepsilon}\sum_{\substack{\chi \in S(Q)}} \ \sum_{\substack{L(\rho,\chi)=0 \\ \beta_0 \leq \beta \leq 1 \\ |\gamma| \leq Q^{\varepsilon}}} X^\beta =-Q^{-\frac{1}{2}+\varepsilon}\int^1_{\beta_0}X^{\alpha}d N(\alpha,Q,Q^{\varepsilon}) \\
\ll &  Q^{-\frac{1}{2}+\varepsilon}X^{\beta_0} N(\beta_0,Q,Q^{\varepsilon})+
(\log X)Q^{-\frac{1}{2}+\varepsilon}\int^1_{\beta_0}X^{\alpha} N(\alpha,Q,Q^{\varepsilon}) d\alpha .
\end{split}
\end{align}

  Note that when $\alpha \leq 1$, we have $1/(3/2-\alpha) \leq 2$, so that the bound \eqref{Nbound} implies that
\begin{align*}
N(\alpha,Q,Q^{\varepsilon}) \ll Q^{4(1-\alpha)+\varepsilon}.
\end{align*}

  Replacing this bound for $N(\alpha,Q,Q^{\varepsilon})$ in the last integral in \eqref{exceptionalzerobound}, we see that when $Q \geq X^{1/4}$, we see that the integrand in the right-hand side expression of \eqref{exceptionalzerobound} is maximized when $\alpha=\beta_0$, it follows from this and \eqref{exceptionalzerobound} that
\begin{align}
\label{zerodensityfirstcase}
Q^{-\frac{1}{2}+\varepsilon}\sum_{\substack{\chi \in S(Q)}} \ \sum_{\substack{L(\rho,\chi)=0 \\ \beta_0 \leq \beta \leq 1 \\ |\gamma| \leq Q^{\varepsilon}}} X^\beta \ll X^{\beta_0}Q^{4(1-\beta_0)-1/2+\varepsilon} \ll X^{\beta_0}X^{4(1-\beta_0)-1/2+\varepsilon} \ll X^{1-\varepsilon},
\end{align}
  when $\beta_0>\beta_1+\varepsilon$ with $\beta_1=5/6$. When $Q < X^{1/4}$, the integrand in the right-hand side expression of \eqref{exceptionalzerobound} is maximized when $\alpha=1$, in which case we have
\begin{align}
\label{Qlargeest}
Q^{-\frac{1}{2}+\varepsilon}\sum_{\substack{\chi \in S(Q)}} \ \sum_{\substack{L(\rho,\chi)=0 \\ \beta_0 \leq \beta \leq 1 \\ |\gamma| \leq Q^{\varepsilon}}} X^\beta \ll X^{\beta_0}Q^{4(1-\beta_0)-1/2+\varepsilon} + XQ^{4(1-1)-1/2+\varepsilon} \ll X^{1-\varepsilon}+XQ^{-1/2+\varepsilon},
\end{align}
  since we have that $Q \ll \exp(w \sqrt{\log X})$.

  We now argue inductively by setting
\begin{align}
\label{betaseq}
  \beta_{n+1}=\frac {\frac 32\beta_n-\frac 14}{\frac 12+\beta_n}
\end{align}
  to see that when $1/2<\beta_0 <\beta_n+\varepsilon$, we have $1/(3/2-\alpha) \leq 1/(3/2-\beta_n)+\varepsilon$ so that when $Q \geq X^{(3/2-\beta_n)/2}$, the estimation \eqref{zerodensityfirstcase} is valid when $\beta_0 > \beta_{n+1}+\varepsilon$. On the other hand, the estimation \eqref{Qlargeest} is valid when $Q < X^{(3/2-\beta_n)/2}$ and $\beta_0 > \beta_{n+1}+\varepsilon$.

  One checks that the sequence $\{ \beta_n \}_{n \geq 1}$ defined in \eqref{betaseq} is decreasing and satisfies $1/2<\beta_n<1$ for all $n \geq 1$ (recall that $\beta_1=5/6$). It follows that $\lim_{n \rightarrow \infty}\beta_n$ exists and a little computation shows that
\begin{align*}
\lim_{n \rightarrow \infty}\beta_n=\frac 12.
\end{align*}

  It follows that, for our choice of $\varepsilon_0>0$ above, we can achieve, after finitely number of iterations, that
\begin{align*}
Q^{-\frac{1}{2}+\varepsilon}\sum_{\substack{\chi \in S(Q)}} \ \sum_{\substack{L(\rho,\chi)=0 \\ 1/2+\varepsilon_0 \leq  \beta_0 \leq \beta \leq 1 \\ |\gamma| \leq Q^{\varepsilon}}} X^\beta \ll X^{1-\varepsilon}+XQ^{-1/2+\varepsilon}.
\end{align*}

  We now combine \eqref{eq: regime II, reduction to zero counting}, \eqref{eq: regime II bound for zeros less than 4 5} and the above estimation to see that for some positive constant $c_8$,
\begin{align*}
%%\label{eq: first moment bound for cal E Q}
\mathcal{E}(Q) \ll XQ^{-1/2+\varepsilon}+X^{1-\varepsilon} \ll  X \exp(-c_8 w \sqrt{\log X}).
\end{align*}

   Applying the above bound to \eqref{E2bound}, we see that for some absolute constant $c > 0$,
\begin{equation}
\label{eq: first moment bounds for E2}
\begin{split}
E_2 &\ll X \exp(-c w \sqrt{\log x}).
\end{split}
\end{equation}

\subsection{Conclusion}

   We now combine \eqref{eq: S1 square gives main term}, \eqref{eq: first moment decompose S1 neq into E1 and E2}, \eqref{eq: first moment bounds for E1} and \eqref{eq: first moment bounds for E2} to see that Proposition \ref{prop: asymptotic for S1} follows.

\section{The mollified second moment}\label{sec:mollified second moment}

   We now begin our proof of Proposition \ref{prop:upper bound for S2}. As a preparation, we first include some results from the sieve methods.
\subsection{Tools from sieve methods}\label{sec:sieves}

  We denote $\mathbf{1}_\mathcal{A}(n)$ for the indicator function of a set $\mathcal{A}$ of algebraic integers in $\mathcal O_K$, so that $\mathbf{1}_\mathcal{A}(n)=1$ when $n \in \mathcal A$ and $\mathbf{1}_\mathcal{A}(n)=0$ otherwise. Then we have \begin{equation}
\label{basicsieveinequality}
\mathbf{1}_{\{n:n\text{ prime}\}} \leq \mathbf{1}_{\{n:(n,P(z_0))=1\}}\mathbf{1}_{\{n:(n,P(R)/P(z_0))=1\}},
\end{equation}
  where $R$ is defined as in \eqref{eq:outline section, defn of M, length of mollifier} and
\begin{align}
\label{eq:sieve section, defn of z0}
z_0 = \exp((\log X)^{1/3}), \quad
P(y) = \prod_{\substack{ \varpi \in G \\ N(\varpi) \leq y}} \varpi, \quad y>2.
\end{align}

  We write $\omega_{[i]}(n)$ for the number of distinct prime ideal factors of $(n)$. We apply Brun's upper bound sieve condition (see \cite[(6.1)]{FI10}) to see that
\begin{equation}\label{Brun}
\mathbf{1}_{\{n:(n,P(z_0))=1\}}(n)\leq \sum_{\substack{b \in G \\ b|(n,P(z_0)) \\ \omega_{[i]}(b)\leq 2r_0 }} \mu_{[i]}(b),
\end{equation}
where
\begin{align*}
 r_0 = \lfloor (\log X)^{1/3} \rfloor.
\end{align*}

  Let $G(t)$ be a non-negative smooth function, compactly supported on $[-1,1]$ satisfying $|G(t)| \ll 1, |G^{(j)}(t)| \ll_j (\log \log X)^{j-1}$ for $j \geq 1$ and  $G(t) = 1-t$ for $0 \leq t \leq 1 - (\log \log X)^{-1}$. Then we have
\begin{align}\label{Selberg}
\mathbf{1}_{\{n:(n,P(R)/P(z_0))=1\}}(n) &\leq \Bigg(\sum_{\substack{d \in G \\ d \mid n \\ (d,P(z_0))=1}} \mu_{[i]}(d) G \left( \frac{\log N(d)}{\log R}\right) \Bigg)^2 \\
&=\mathop{\sum\sum}_{\substack{j, k \in G \\ N(j),N(k) \leq R \\ [j,k]|n \\ (jk,P(z_0))=1 }} \mu_{[i]}(j)\mu_{[i]}(k)G\left( \frac{\log N(j)}{\log R}\right)G\left( \frac{\log N(k)}{\log R}\right). \nonumber
\end{align}

 Now \eqref{basicsieveinequality}, \eqref{Brun} and \eqref{Selberg} implies that
\begin{equation}\label{sieveinequality}
\mathbf{1}_{\{n:n\text{ prime}\}}(n) \leq \sum_{\substack{ d\in G\\ d|n}} \lambda_d,
\end{equation}
where the coefficients $\lambda_d$ are defined by
\begin{equation}\label{lambda}
\lambda_d= \sum_{\substack{b\in G \\ b|P(z_0) \\ \omega_{[i]}(b)\leq 2r_0 }} \mathop{\sum\sum}_{\substack{m,n \in G \\ N(m),N(n)\leq R \\ b[m,n]=d \\ (mn,P(z_0))=1 }} \mu_{[i]}(b)\mu_{[i]}(m)\mu_{[i]}(n)G\left( \frac{\log N(m)}{\log R}\right)G\left( \frac{\log N(n)}{\log R}\right).
\end{equation}
  Similar to what is pointed out in the paragraph above \cite[(5.9)]{B&P}, we have $\lambda_d\neq 0$ for $N(d) \leq D$, where
\begin{equation}\label{Ddef}
D=R^2 \exp(2(\log X)^{2/3}) \ll_{\varepsilon} R^2X^{\varepsilon}.
\end{equation}

 We end this section by listing a few lemmas needed in the paper, which are analogues to \cite[Lemma 5.1-5.4]{B&P}.
\begin{lemma}\label{fundamentallemma}
Let $0 <\delta < 1$ be a fixed constant, $r$ a positive integer with $r \asymp (\log X)^{\delta}$, and $z_0$ as in \eqref{eq:sieve section, defn of z0}. Let $G$ be the set of generators of ideals in $\mathcal O_K$ chosen in Section \ref{sec2.4}. Suppose that $g$ is a multiplicative function on $G$ such that uniformly for all primes $\varpi \in G$, we have $|g(\varpi)|\ll 1$. Then uniformly for all $\ell \in \mathcal O_K$,
\begin{equation*}
\sum_{\substack{b \in G \\ b \mid P(z_0) \\ \omega_{[i]}(b)\leq r \\ (b,\ell)=1}}\frac{\mu_{[i]}(b)}{N(b)}g(b) = \prod_{\substack{\varpi \in G \\ N(\varpi) \leq z_0\\ \varpi \nmid \ell}}\Bigg( 1-\frac{g(\varpi)}{N(\varpi)}\Bigg) +O\Big( \exp(-r\log \log r )\Big).
\end{equation*}
\end{lemma}

\begin{lemma}\label{Selbergsieve}
Let $z_0=\exp((\log X)^{1/3})$. Let $G(t)$ be as above and $G$ be the set of generators of ideals in $\mathcal O_K$ chosen in Section \ref{sec2.4}. Suppose $h$ is a function on $G$ such that uniformly for all primes $\varpi \in G$, $|h(\varpi)|\ll_\varepsilon N(\varpi)^{-\varepsilon}$ . For a fixed real number $A > 0$, there exists a function $E_0(X)$ depending only on $X,G$ and $\vartheta$ with $E_0(X) \rightarrow 0$ as $X \rightarrow \infty$, such that uniformly for $N(\ell) \ll X^{O(1)}$,
\begin{align}\label{Selbergsieve2}
\begin{split}
& \mathop{\sum\sum}_{\substack{m, n \in G \\ N(m),N(n) \leq R \\ (mn, \ell P(z_0))=1}}
  \frac{\mu_{[i]}(m)\mu_{[i]}(n)}{N([m,n])} \ G\left( \frac{\log N(m)}{\log R}\right)G\left( \frac{\log N(n)}{\log R}\right)\prod_{\substack{\varpi \in G \\ \varpi|mn}}\Big( 1+h(\varpi)\Big)\\
 =& \frac {4}{\pi} \cdot \frac{1+E_0(X)}{\log R}\prod_{\substack{\varpi \in G \\ N(\varpi) \leq z_0}}\left( 1-\frac{1}{N(\varpi)}\right)^{-1} + \ O_{\varepsilon,A}\left( \frac{1}{(\log R)^{A}}\right).
\end{split}
\end{align}
\end{lemma}

\begin{lemma}\label{sieve}
Let $\lambda_d$ and $D$ be as defined in \eqref{lambda} and \eqref{Ddef}, respectively. Let $G$ be the set of generators of ideals in $\mathcal O_K$ chosen in Section \ref{sec2.4}. Suppose that  $g$ is a multiplicative function on $G$ such that $g(\varpi)=1+O(N(\varpi)^{-\varepsilon})$ for all primes $\varpi \in G$. Then with $E_0(X)$ as in Lemma \ref{Selbergsieve} we have uniformly in $N(\ell) \ll X^{O(1)}$,
\begin{equation*}
\begin{split}
\sum_{\substack{d \in G \\ N(d) \leq D\\ (d,\ell)=1}}\frac{\lambda_d}{N(d)} g(d) = \frac 4{\pi} \cdot \frac{1+E_0(X)}{\log R}\prod_{\substack{\varpi \in G \\ N(\varpi) \leq z_0\\ \varpi \nmid \ell}}\Bigg( 1-\frac{g(\varpi)}{N(\varpi)}\Bigg)\prod_{\substack{\varpi \in G \\ N(\varpi) \leq z_0}}\left( 1-\frac{1}{N(\varpi)}\right)^{-1} + O_\varepsilon\left( \frac{1}{(\log R)^{2020}}\right).
\end{split}
\end{equation*}
\end{lemma}

\begin{lemma}\label{sievewithsum}
Let $G$ be the set of generators of ideals in $\mathcal O_K$ chosen in Section \ref{sec2.4}. Let $\lambda_d,D,g$ be as in Lemma~\ref{sieve}. Suppose that $h$ is a function on $G$ such that $|h(\varpi)|\ll_\varepsilon N(\varpi)^{-1+\varepsilon}$ for all primes $\varpi \in G$. Then with $E_0(X)$ as in Lemma \ref{Selbergsieve} we have
\begin{equation*}
\begin{split}
\sum_{\substack{d \in G \\ N(d)\leq D\\ (d,\ell)=1}}\frac{\lambda_d}{N(d)} g(d)\sum_{\substack{\varpi \in G \\ \varpi|d}} h(\varpi)=
& - \frac 4{\pi} \cdot \frac{1+E_0(X)}{\log R}\prod_{\substack{ \varpi \in G \\ N(\varpi) \leq z_0}}\left( 1-\frac{1}{N(\varpi)}\right)^{-1} \\
& \ \times \sum_{\substack{\varpi \in G \\ N(\varpi) \leq z_0 \\ \varpi \nmid \ell}} \frac{g(\varpi)h(\varpi)}{N(\varpi)} \prod_{\substack{\varpi' \in G \\ \varpi' \text{prime} \\ N(\varpi') \leq z_0\\ \varpi' \nmid \varpi\ell}}\Bigg( 1-\frac{g(\varpi')}{N(\varpi')}\Bigg) + O_\varepsilon\left( \frac{1}{(\log R)^{2020}}\right),
\end{split}
\end{equation*}
uniformly for all $\ell \in \mathcal O_K$ such that $\log N(\ell)\ll \log X$. (Here, the index $\varpi'$ runs over primes $\varpi'$.)
\end{lemma}

   As the above lemmas can be established similarly to \cite[Lemma 5.1-5.4]{B&P}, we omit the proofs by only pointing out that constant $4/\pi$ in \eqref{Selbergsieve2} (and hence in Lemma \ref{sieve} and \ref{sievewithsum}) comes from expanding
\begin{equation*}
%%\label{Selbergsieve5}
\begin{split}
 \frac{\zeta_K\left(1+\frac{2+iz_1+iz_2}{\log R}\right)}{\zeta_K\left(1+\frac{1+iz_1}{\log R}\right) \zeta_K\left(1+\frac{1+iz_2}{\log R}\right)}
\end{split}
\end{equation*}
 into Laurent series and noting that the residue of $\zeta_K(s)$ at $s = 1$ equals $\pi/4$.

\subsection{Initial treatment}
   Now we are ready to estimate $S_2$. As $\Phi$ is supported on $[\frac{1}{2},1]$, we have $\log N(\varpi) \leq \log X$, so that by positivity we may apply the sieve given \eqref{sieveinequality} to see that
\begin{equation*}
S_2\leq (\log X) S^+,
\end{equation*}
where
\begin{equation}\label{S+def}
S^+ = \sum_{\substack{ (n,2)=1 }} \mu_{[i]}^2(n) \Bigg(\sum_{\substack{d \in G \\ d \mid n \\ N(d) \leq D}} \lambda_d \Bigg) \Phi\left( \frac{N(n)}{X}\right) L(\tfrac{1}{2},\chi_{(1+i)^5n})^2 M(n)^2.
\end{equation}
  As $d | n$ and $n$ is odd, we know that $d$ is also odd. Thus, $d \in G$ implies that $d$ is primary.  Also,  we may write $\lambda_d=\mu_{[i]}^2(d)\lambda_d$ since $\lambda_d\neq 0$ only for square-free $d$ by \eqref{lambda}. We further write
\begin{equation}\label{mu2approx}
\mu_{[i]}^2(n) = N_Y(n) + R_Y(n),
\end{equation}
where
\begin{equation}\label{MYRY}
N_Y(n) = \sum_{\substack{\ell \in G \\ \ell^2 \mid n \\ N(\ell) \leq Y}} \mu_{[i]}(\ell), \ \ \ \ \ R_Y(n) = \sum_{\substack{\ell \in G \\ \ell^2 \mid n \\ N(\ell) > Y}} \mu_{[i]}(\ell),
\end{equation}
 for some parameter $Y$ to be determined later. Now, we apply Lemma \ref{lem:AFE} and \eqref{eq:Vdef} to write $L(\tfrac{1}{2},\chi_{(1+i)^5n})^2 = \mathcal{D}_2(n)$, where
\begin{align}
\label{D2}
\begin{split}
\mathcal{D}_2(n) =&  2\sum_{\substack{\nu \equiv 1 \bmod {(1+i)^3}}} \frac {d_{[i],2}(\nu)}{N(\nu)^{\frac{1}{2}}} V_2
\left(\frac{N(\nu)}{N(n)} \right) \left( \frac{(1+i)^5n}{\nu}\right)\\
=&\frac{2 }{2\pi i}\int\limits_{(c)}\left(\frac{2^{5/2}}{\pi}\right)^{2s}
\left ( \frac {\Gamma(\frac{1}{2}+s)}{\Gamma(\frac{1}{2})} \right )^2 N(n)^s L^2 \left( \frac{1}{2}+s,\chi_{(1+i)^5h}\right) \mathcal{E}(s,2) \frac{ds}{s}.
\end{split}
\end{align}
 Here $c > 1/2$ and
\begin{align*}
\mathcal{E}(s,k) = \prod_{\substack{ \varpi \in G \\ \varpi \mid k}} \left(1 - \frac{\chi_{(1+i)^5h}(\varpi)}{N(\varpi)^{1/2+s}}\right)^2.
\end{align*}

  Applying \eqref{mu2approx} and \eqref{D2} in \eqref{S+def}, we see that
\begin{equation}\label{S+SMSR}
S^+ = S^+_N + S^+_R,
\end{equation}
where
\begin{align}
\label{SMdef}
\begin{split}
S^+_N =&  \sum_{\substack{ (n,2)=1 }} N_Y(n) \Bigg(\sum_{\substack{d \mid n \\ N(d) \leq D \\ d \equiv 1 \bmod {(1+i)^3}}}\mu_{[i]}^2(d) \lambda_d \Bigg) \Phi\left( \frac{N(n)}{X}\right) \mathcal{D}_2(n) M(n)^2, \\
S^+_R =& \sum_{\substack{ (n,2)=1 }} R_Y(n) \Bigg(\sum_{\substack{d \mid n \\ N(d) \leq D \\ d \equiv 1 \bmod {(1+i)^3}}}\mu_{[i]}^2(d) \lambda_d \Bigg) \Phi\left( \frac{N(n)}{X}\right) \mathcal{D}_2(n) M(n)^2.
\end{split}
\end{align}

\subsection{Evaluation of $S_R^+$}\label{subsec:contrib of R_Y}

In this section we evaluate $S_R^+$.   We apply the divisor bound and the observation that $|\lambda_d|\ll N(d)^{\varepsilon}$ by \eqref{lambda} to deduce that
\begin{align*}
|R_Y(n)| \ll N(n)^\varepsilon, \ \ \ \ \ \ \  \ \Bigg| \sum_{\substack{d \mid n \\ N(d) \leq D \\ d \equiv 1 \bmod {(1+i)^3}}}\mu_{[i]}^2(d) \lambda_d \Bigg| \ll N(n)^\varepsilon,
\end{align*}

  Further note that $R_Y(n)=0$ unless $n = \ell^2 h$ with $N(\ell) > Y$ and $h$ square-free. This together with the above bounds implies that, via  Cauchy-Schwarz,
\begin{align}
\label{eq:up bound for S R plus after cauchy}
\begin{split}
S_R^+ \ll& X^\varepsilon \sum_{\substack{Y < N(\ell) \leq \sqrt{X} }} \ \sum_{\substack{X/2N(\ell)^2 <N(h) \leq X/N(\ell)^2 }} \mu_{[i]}^2(h) |M(\ell^2h)^2\mathcal{D}_2(\ell^2h)| \\
\ll& X^\varepsilon \sum_{\substack{Y < N(\ell) \leq \sqrt{X} }} \Bigg(\sum_{\substack{X/2N(\ell)^2 < N(h) \leq X/N(\ell)^2 }} \mu_{[i]}^2(h) |M(\ell^2h)^2|^2 \Bigg)^{1/2}\Bigg(\sum_{\substack{X/2N(\ell)^2 < N(h) \leq X/N(\ell)^2 }} \mu_{[i]}^2(h) |\mathcal{D}_2(\ell^2h)|^2 \Bigg)^{1/2}.
\end{split}
\end{align}

  Now, we write $M(\ell^2h)^2$ as
\begin{align*}
M(\ell^2h)^2 = \sum_{\substack{N(m) \leq M^2 \\ m \equiv 1 \bmod {(1+i)^3} \\ (m,\ell)=1}} \frac{\alpha(m)}{\sqrt{N(m)}}\left(\frac{(1+i)^5h}{m} \right),
\end{align*}
 where $|\alpha(m)| \ll N(m)^\varepsilon$. Then Lemma \ref{lem: estimates for character sums} allows us to deduce that
\begin{align}\label{eq:up bound S R + mollifier squared}
\sum_{\substack{X/2N(\ell)^2 < N(h) \leq X/N(\ell)^2 }} \mu_{[i]}^2(h) |M(\ell^2h)^2|^2 \ll X^\varepsilon \left(\frac{X}{N(\ell)^2} + M^2 \right).
\end{align}

  Next, we deduce from \eqref{D2} that
\begin{align*}
\mathcal{D}_2(\ell^2h) = \frac{2 }{2\pi i}\int\limits_{(c)}\left(\frac{2^{5/2}}{\pi}\right)^{2s}
\left ( \frac {\Gamma(\frac{1}{2}+s)}{\Gamma(\frac{1}{2})} \right )^2 N(\ell^2 h)^s L^2 \left( \frac{1}{2}+s,\chi_{(1+i)^5h}\right) \mathcal{E}(s,2\ell) \frac{ds}{s}.
\end{align*}
 We evaluate $\mathcal{D}_2(\ell^2h)$ by moving the line of integration to $c = \frac{1}{\log X}$ without encountering any poles. We then apply Cauchy-Schwarz to see that
\begin{align*}
|\mathcal{D}_2(\ell^2h)|^2 \ll X^\varepsilon\int\limits_{(\frac{1}{\log X})} \left|\Gamma(\frac{1}{2}+s)\right|^2 \left|L \left( \frac{1}{2}+s,\chi_{(1+i)^5h}\right) \right|^4 |ds|.
\end{align*}
 It then follows from Lemma \ref{lem:2.3} that we have
\begin{align}\label{eq:S_R contrib from fourth moment}
\sum_{\substack{X/2N(\ell)^2 < N(h) \leq X/N(\ell)^2 }} \mu_{[i]}^2(h) |\mathcal{D}_2(\ell^2h)|^2 \ll \frac{X^{1+\varepsilon}}{N(\ell)^2}.
\end{align}
 We thus deduce from \eqref{eq:up bound for S R plus after cauchy}, \eqref{eq:up bound S R + mollifier squared} and \eqref{eq:S_R contrib from fourth moment} that
\begin{align}\label{eq:bound on S_R^+}
S_R^+ \ll X^\varepsilon \left(\frac{X}{Y} + X^{1/2} M \right).
\end{align}

\subsection{Decomposition of $S_N^+$}

 In this section, we begin to evaluate $S_N^+$.  We apply \eqref{eq: defn of mollifier} and \eqref{D2} in \eqref{SMdef} to see that
\begin{equation*}
%%\label{SM1}
\begin{split}
S^+_N = 2 \sum_{\substack{ N(d) \leq D \\ d \equiv 1 \bmod {(1+i)^3} }} \mu_{[i]}^2(d)\lambda_d
& \mathop{\sum\sum}_{\substack{N(m_1),N(m_2)\leq M \\ m_1,m_2 \equiv 1 \bmod {(1+i)^3} }} \frac{b_{m_1}b_{m_2}}{\sqrt{N(m_1m_2)}}\sum_{\substack{\nu \equiv 1 \bmod {(1+i)^3}}} \frac {d_{[i],2}(\nu)}{N(\nu)^{\frac{1}{2}}} \left( \frac{(1+i)^5}{m_1m_2 \nu }\right) Z(d,\nu,m_1m_2;X,Y),
\end{split}
\end{equation*}
 where
\begin{align}
\label{Zdef}
\begin{split}
Z =& Z(d,\nu,m_1m_2;X,Y) = \sum_{\substack{ (n,2)=1 \\ d|n }} N_Y(n) \Phi\left( \frac{N(n)}{X}\right)  V_2
\left(\frac{N(\nu)}{N(n)} \right) \left( \frac{n}{m_1m_2 \nu}\right) \\
=& \sum_{\substack{ N(\alpha) \leq Y \\ \alpha \equiv 1 \bmod {(1+i)^3}}} \mu_{[i]}(\alpha) \sum_{\substack{ (n,2)=1 \\ [\alpha^2,d]|n }} F_{N(\nu)}\left( \frac{N(n)}{X}\right)  \left( \frac{n}{m_1m_2 \nu}\right).
\end{split}
\end{align}
 Here $F_{y}(t)$ is defined in \eqref{Fdef} and the last equality above follows from \eqref{MYRY} by noting that when $n$ is odd, $\ell |n$ and $\ell \in G$ implies that $\ell $ is primary.

As both $\alpha$ and $d$ are primary and square-free, we have $[\alpha^2,d]=\alpha^2 d_1$, where $d_1$ is defined in \eqref{eq: defn of d 1}.
Therefore, we can rewrite $n$ as $\alpha^2d_1m$ in \eqref{Zdef} to recast $Z$ as
\begin{align*}
\begin{split}
Z =& \sum_{\substack{ N(\alpha) \leq Y \\ \alpha \equiv 1 \bmod {(1+i)^3} \\ (\alpha,m_1m_2\nu)=1}} \mu_{[i]}(\alpha)  \left( \frac{d_1}{m_1m_2\nu}\right) \sum_{\substack{ (m,2)=1 }} F_{N(\nu)}\left( \frac{N(\alpha^2d_1m)}{X}\right)  \left( \frac{m}{m_1m_2 \nu}\right)\\
= & \frac{X}{2N(m_1m_2\nu)} \sum_{\substack{ N(\alpha) \leq Y \\ \alpha \equiv 1 \bmod {(1+i)^3} \\ (\alpha,m_1m_2\nu)=1}} \frac {\mu_{[i]}(\alpha)}{N(\alpha^2d_1)}  \left( \frac{(1+i)d_1}{m_1m_2\nu}\right) \sum_{k \in \mathcal O_K}(-1)^{N(k)}\widetilde{F}_{N(\nu)}\left( \sqrt{\frac{N(k)X}{2N(\alpha^2d_1m_1m_2\nu)}}\right)g(k,m_1m_2\nu),
\end{split}
\end{align*}
 where the last equality above follows from Lemma \ref{Poissonsumformodd}.

Applying the above expression of $Z$ in \eqref{Zdef}, we deduce that
\begin{align}
\label{SM2}
\begin{split}
S_N^+ =& X\sum_{\substack{ N(d) \leq D \\ d \equiv 1 \bmod {(1+i)^3} }} \mu_{[i]}^2(d)\lambda_d
 \mathop{\sum\sum}_{\substack{N(m_1),N(m_2)\leq M \\ m_1,m_2 \equiv 1 \bmod {(1+i)^3} \\ (m_1m_2,d)=1}} \frac{b_{m_1}b_{m_2}}{N(m_1m_2)^{3/2}} \sum_{\substack{\nu \equiv 1 \bmod {(1+i)^3} \\ (\nu,d)=1}} \frac {d_{[i],2}(\nu)}{N(\nu)^{\frac{3}{2}}} \\
& \times \sum_{\substack{ N(\alpha) \leq Y \\ \alpha \equiv 1 \bmod {(1+i)^3} \\ (\alpha,m_1m_2\nu)=1}} \frac {\mu_{[i]}(\alpha)}{N(\alpha^2d_1)} \left( \frac{d_1}{m_1m_2\nu}\right) \sum_{k \in \mathcal O_K}(-1)^{N(k)}\widetilde{F}_{N(\nu)}\left( \sqrt{\frac{N(k)X}{2N(\alpha^2d_1m_1m_2\nu)}}\right)g(k,m_1m_2\nu) \\
=& \mathcal{T}_0 + \mathcal{B},
\end{split}
\end{align}
 where $\mathcal{T}_0$ singles out the term $k=0$ of the first expression on the right-hand side of \eqref{SM2} while $\mathcal{B}$ being the rest.

\subsection{Evaluation of $\mathcal{T}_0$ }\label{subsec:k = 0}

 It follows from Lemma \ref{Gausssum} that $g(0, n)=\varphi_{[i]}(n)$ if $n=\square$ and $g(0,n)=0$ otherwise. This implies that
\begin{equation}\label{eq: T0 start}
\begin{split}
\mathcal{T}_0 =&  X \sum_{\substack{ N(d) \leq D \\ d \equiv 1 \bmod {(1+i)^3} }} \mu_{[i]}^2(d)\lambda_d
 \mathop{\sum\sum}_{\substack{N(m_1),N(m_2) \leq M \\ m_1,m_2 \equiv 1 \bmod {(1+i)^3} \\ (m_1m_2,d)=1}} \frac{b_{m_1}b_{m_2}}{N(m_1m_2)^{3/2}} \sum_{\substack{\nu \equiv 1 \bmod {(1+i)^3} \\ (\nu,d)=1 \\ m_1m_2\nu=\square}} \frac {d_{[i],2}(\nu)}{N(\nu)^{\frac{3}{2}}} \\
& \times \sum_{\substack{ N(\alpha) \leq Y \\ \alpha \equiv 1 \bmod {(1+i)^3} \\ (\alpha,m_1m_2\nu)=1}} \frac {\mu_{[i]}(\alpha)}{N(\alpha^2d_1)} \widetilde{F}_{N(\nu)}\left( 0\right)\varphi_{[i]}(m_1m_2\nu).
\end{split}
\end{equation}

We now extend the sum over $\alpha$ to all elements in $\mathcal O_K$. As $\varphi_{[i]}(n)\leq N(n)$, the error term introduced is
\begin{align}
\label{eq: error in extending T0 alpha sum}
\begin{split}
\ll & X\sum_{\substack{ N(d) \leq D \\ d \equiv 1 \bmod {(1+i)^3} }} |\lambda_d| \mathop{\sum\sum}_{\substack{N(m_1),N(m_2) \leq M \\ m_1,m_2 \equiv 1 \bmod {(1+i)^3} }} \frac{|b_{m_1}b_{m_2}|}{\sqrt{N(m_1m_2)}} \sum_{\substack{\nu \equiv 1 \bmod {(1+i)^3}  \\ m_1m_2\nu=\square}} \frac {d_{[i],2}(\nu)}{N(\nu)^{\frac{1}{2}}} \sum_{N(\alpha)>Y} \frac{1}{N(\alpha^2d_1)} |\widetilde{F}_{N(\nu)}(0)| \\
\ll &  X^{1+\varepsilon}\sum_{N(d)\leq D}  \mathop{\sum\sum}_{N(m_1),N(m_2) \leq M} \frac{1}{\sqrt{N(m_1m_2)}} \sum_{\substack{N(\nu) \leq X^{1+\varepsilon}\\ m_1m_2\nu=\square}}\frac{1}{\sqrt{N(\nu)}} \sum_{N(\alpha)>Y} \frac{1}{N(\alpha^2d_1)} + \exp\left(-X^{\varepsilon}\right),
\end{split}
\end{align}
 where the last estimation above follows from the observation that $\widetilde{F}_{N(\nu)}(0)\ll 1$ for all $N(\nu)>0$ and $\widetilde{F}_{N(\nu)}(0)\ll \exp(-\frac{ N(\nu)}{8X})$ for $N(\nu)>X^{1+\varepsilon}$, together with the bounds that $|\lambda_d|\ll N(d)^{\varepsilon}$ by \eqref{lambda} and $|b_m|\ll 1$ by \eqref{eq: defn of mollifier coeffs bm}.

As $m_1m_2\nu=\square$, the sum over $m_1,m_2,\nu$ in \eqref{eq: error in extending T0 alpha sum} is $\ll X^{\varepsilon}$. Also, it follows from the definition of $d_1$ in \eqref{eq: defn of d 1} that
\begin{align}
\label{Nalphasum}
\sum_{N(\alpha)>Y} \frac{1}{N(\alpha^2d_1)} = \frac{1}{N(d)} \sum_{j|d} \varphi_{[i]}(j) \sum_{\substack{N(\alpha)>Y \\ j|\alpha}} \frac{1}{N(\alpha)^{2}} \ll \frac{1}{N(d)^{1-\varepsilon}Y}.
\end{align}

We thus conclude that the expressions in \eqref{eq: error in extending T0 alpha sum} are further bounded by $O(X^{1+\varepsilon}/Y)$. Hence by \eqref{eq: T0 start}, we have
\begin{equation*}
\begin{split}
\mathcal{T}_0 = & X\sum_{\substack{ N(d) \leq D \\ d \equiv 1 \bmod {(1+i)^3} }} \mu_{[i]}^2(d)\lambda_d
 \mathop{\sum\sum}_{\substack{N(m_1),N(m_2) \leq M \\ m_1,m_2 \equiv 1 \bmod {(1+i)^3} \\ (m_1m_2,d)=1}} \frac{b_{m_1}b_{m_2}}{N(m_1m_2)^{3/2}} \sum_{\substack{\nu \equiv 1 \bmod {(1+i)^3} \\ (\nu,d)=1 \\ m_1m_2\nu=\square}} \frac {d_{[i],2}(\nu)}{N(\nu)^{\frac{3}{2}}} \\
& \times \sum_{\substack{ \alpha \equiv 1 \bmod {(1+i)^3} \\ (\alpha,m_1m_2\nu)=1}} \frac {\mu_{[i]}(\alpha)}{N(\alpha^2d_1)} \widetilde{F}_{\nu}\left( 0\right)\varphi_{[i]}(m_1m_2\nu)+ O\left( \frac{X^{1+\varepsilon}}{Y}\right).
\end{split}
\end{equation*}

  Now we express the sum on $\alpha$ in terms of an Euler product to arrive that
\begin{align}\label{eq: T0 before sieve}
\begin{split}
\mathcal{T}_0 =& \frac{4X}{3\zeta_K(2)}\sum_{\substack{ N(d) \leq D \\ d \equiv 1 \bmod {(1+i)^3} }}  \frac{\mu_{[i]}^2(d)\lambda_d}{N(d)} \prod_{\substack{ \varpi \in G \\ \varpi|d}} \left(\frac{N(\varpi)}{N(\varpi)+1}\right)\mathop{\sum\sum}_{\substack{N(m_1),N(m_2) \leq M \\ m_1,m_2 \equiv 1 \bmod {(1+i)^3} \\ (m_1m_2,d)=1}} \frac{b_{m_1}b_{m_2}}{N(m_1m_2)^{1/2}} \\
& \times  \sum_{\substack{\nu \equiv 1 \bmod {(1+i)^3}  \\ (\nu, d)=1 \\ m_1m_2\nu=\square}} \frac {d_{[i],2}(\nu)}{N(\nu)^{\frac{1}{2}}}  \widetilde{F}_{N(\nu)}(0) \prod_{\substack{\varpi \in G \\ \varpi|m_1m_2\nu}}\left( \frac{N(\varpi)}{N(\varpi)+1}\right)   + O\left( \frac{X^{1+\varepsilon}}{Y}\right).
\end{split}
\end{align}

  We apply Lemma~\ref{sieve} to see that
\begin{align}\label{eq:k=0 subsection, sum on d}
\begin{split}
 \sum_{\substack{ N(d) \leq D \\ (d, m_1m_2\nu)=1 \\ d \equiv 1 \bmod {(1+i)^3} }}  \frac{\mu_{[i]}^2(d)\lambda_d}{N(d)} \prod_{\substack{\varpi \in G \\\varpi|d}} \left(\frac{N(\varpi)}{N(\varpi)+1}\right) =& \frac 4{\pi} \cdot \frac{1+E_0(X)}{\log R}\prod_{\substack{\varpi \in G\\ \varpi| 2m_1m_2\nu \\ N(\varpi) \leq z_0 }}\left( 1+\frac{1}{N(\varpi)}\right)\prod_{\substack{\varpi \in G \\ N(\varpi) \leq z_0}}\left( \frac{N(\varpi)^2}{N(\varpi)^2-1}\right) \\
& + O\left( (\log R)^{-2020}\right).
\end{split}
\end{align}

 We now write $o(1)$ for $E_0(X)$ as $E_0(X) \rightarrow 0$ when $X \rightarrow \infty$.  It also follows from trivial estimation that we can omit the condition $N(\varpi) \leq z_0$ in \eqref{eq:k=0 subsection, sum on d}. Applying these in \eqref{eq:k=0 subsection, sum on d} and then to \eqref{eq: T0 before sieve}, we see that
\begin{equation}\label{eq: T0 after sieve}
\begin{split}
\mathcal{T}_0 =  \frac {8}{\pi} X\frac{1+o(1)}{\log R} \Upsilon_0 + O\left( \frac{X}{(\log R)^{2020}}+\frac{X^{1+\varepsilon}}{Y}\right),
\end{split}
\end{equation}
 where
\begin{equation}\label{eq: defn of Upsilon 0}
\Upsilon_0 = \mathop{\sum\sum}_{\substack{N(m_1),N(m_2) \leq M \\ m_1,m_2 \equiv 1 \bmod {(1+i)^3} }} \frac{b_{m_1}b_{m_2}}{N(m_1m_2)^{1/2}}    \sum_{\substack{\nu \equiv 1 \bmod {(1+i)^3}  \\ m_1m_2\nu=\square}} \frac {d_{[i],2}(\nu)}{N(\nu)^{\frac{1}{2}}}  \widetilde{F}_{N(\nu)}(0)  .
\end{equation}

  We apply \eqref{eq:Vdef} and \eqref{Fdef} to see that
\begin{align}
\label{F0}
\begin{split}
  \widetilde{F}_{N(\nu)} \left(0\right) =& \int\limits^{\infty}_{-\infty}\int\limits^{\infty}_{-\infty}\Phi \left(N(x+yi) \right)
V_{2}  \left(  \frac {N(\nu)}{X N(x+yi)} \right ) \dif x \dif y \\
=&  \frac {1}{2\pi
   i} \int\limits\limits_{(2)}\left(\frac{2^{5/2}}{\pi}\right)^{2s}
\left ( \frac {\Gamma(\frac{1}{2}+s)}{\Gamma(\frac{1}{2})} \right )^2 \left(  \frac {X } {N(\nu)}\right )^{s}
 \left (\int\limits^{\infty}_{-\infty}\int\limits^{\infty}_{-\infty}\Phi \left(N(x+yi) \right) N(x+yi)^{s}
\dif x \dif y \right )  \frac { ds}{s} \\
=& \frac {\pi }{2\pi
   i} \int\limits\limits_{(2)} \left(\frac{2^{5/2}}{\pi}\right)^{2s}
\left ( \frac {\Gamma(\frac{1}{2}+s)}{\Gamma(\frac{1}{2})} \right )^2 \left(  \frac {X } {N(\nu)}\right )^{s}
 \widehat{\Phi}(1+s)\frac {ds}{s},
\end{split}
\end{align}
  since
\begin{align*}
  \int\limits^{\infty}_{-\infty}\int\limits^{\infty}_{-\infty}\Phi \left(N(x+yi) \right) N(x+yi)^{s}
\dif x \dif y =\int^{2\pi}_0\int^{\infty}_0\Phi (r^2)r^{2s}rdrd\theta =\pi \widehat{\Phi}(1+s).
\end{align*}

 Applying \eqref{eq: defn of mollifier coeffs bm}, \eqref{eq: write H as Fourier integral} and \eqref{F0} in \eqref{eq: defn of Upsilon 0}, we obtain for $c=\frac{1}{\log X}$,
\begin{align}\label{eq: Upsilon 0 after Fourier inversion}
\begin{split}
\Upsilon_0 =& \frac{\pi}{2\pi i} \int\limits_{(c)} \left(\frac{2^{5/2}}{\pi}\right)^{2s}
\left ( \frac {\Gamma(\frac{1}{2}+s)}{\Gamma(\frac{1}{2})} \right )^2 X^{s} \widehat{\Phi}(1+s) \int_{-\infty}^{\infty}\int_{-\infty}^{\infty}h(z_1)h(z_2)  \\
& \times \mathop{\sum\sum\sum}_{\substack{ m_1,m_2, \nu \equiv 1 \bmod {(1+i)^3} \\ m_1m_2\nu=\square}} \frac{\mu_{[i]}(m_1)\mu_{[i]}(m_2)d_{[i],2}(\nu)}{N(m_1m_2\nu)^{\frac{1}{2}} N(m_1)^{\frac{1+iz_1}{\log M}}N(m_2)^{\frac{1+iz_2}{\log M}}N(\nu)^s}  \,dz_1dz_2\frac{ds}{s}.
\end{split}
\end{align}
  We write the sum over $m_1,m_2,\nu$ as an Euler product to see that
\begin{align}
\label{eq: defn of Q for T0}
\begin{split}
& \mathop{\sum\sum\sum}_{\substack{ m_1,m_2, \nu \equiv 1 \bmod {(1+i)^3} \\ m_1m_2\nu=\square}} \frac{\mu_{[i]}(m_1)\mu_{[i]}(m_2)d_{[i],2}(\nu)}{N(m_1m_2\nu)^{\frac{1}{2}} N(m_1)^{\frac{1+iz_1}{\log M}}N(m_2)^{\frac{1+iz_2}{\log M}}N(\nu)^s} \\
 =& \zeta_K^3(1+2s)\zeta_K\Big(1+\tfrac{2+iz_1+iz_2}{\log M} \Big) \zeta_K^{-2} \Big( 1+ \tfrac{1+iz_1}{\log M} +s\Big) \zeta_K^{-2}\Big( 1+\tfrac{1+iz_2}{\log M}+s \Big)Q\left(\tfrac{1+iz_1}{\log M},\tfrac{1+iz_2}{\log M},s \right),
\end{split}
\end{align}
where $Q(w_1,w_2,s)$ is holomorphic and uniformly bounded in the region $\Re(w_1), \Re(w_2), \Re(s) \geq -\varepsilon$, which satisfies
\begin{equation}\label{eq: Q at 000}
Q(0,0,0)=1.
\end{equation}

  Applying \eqref{eq: defn of Q for T0} in \eqref{eq: Upsilon 0 after Fourier inversion}, we deduce that
\begin{equation*}
\begin{split}
 \Upsilon_0 =&  \frac{\pi}{2\pi i} \int\limits_{(c)} \left(\frac{2^{5/2}}{\pi}\right)^{2s}
\left ( \frac {\Gamma(\frac{1}{2}+s)}{\Gamma(\frac{1}{2})} \right )^2 X^{s} \widehat{\Phi}(1+s) \zeta_K^3(1+2s) \int_{-\infty}^{\infty}\int_{-\infty}^{\infty}h(z_1)h(z_2) \\
& \times \zeta_K\Big(1+\tfrac{2+iz_1+iz_2}{\log M} \Big) \zeta_K^{-2} \Big( 1+ \tfrac{1+iz_1}{\log M} +s\Big) \zeta_K^{-2}\Big( 1+\tfrac{1+iz_2}{\log M}+s \Big)Q\left(\tfrac{1+iz_1}{\log M},\tfrac{1+iz_2}{\log M},s \right)  \,dz_1dz_2\frac{ds}{s}.
\end{split}
\end{equation*}

 We may truncate the integrals above to $|z_1|,|z_2|\leq \sqrt{\log M}$ and $|\Im(s)|\leq (\log X)^2$ due to the rapid decay of the gamma function in vertical strips and \eqref{eq: h has rapid decay}. Notice further that similar to \cite[Theorem 6.7]{MVa1}, we can show that there exists a constant  $c'$  such that when $\Re(z)\geq -c'/\log|\Im(z)|$ and $|\Im(z)|\geq 1$, we have
\begin{equation*}
%%\label{eq: zeta bound in zero free region}
\zeta_K(1+z)\ll \log |\Im(z)| \ \ \ \ \text{and} \ \ \ \ \frac{1}{\zeta_K(1+z)} \ll \log |\Im(z)|.
\end{equation*}
We then change the contour of integration over $s$ to the path consisting of the line segment $L_1$ from $\frac{1}{\log X}-i(\log X)^2$ to $-\frac{c'}{\log\log X}-i(\log X)^2$, the line segment $L_2$ from $-\frac{c'}{\log\log X}-i(\log X)^2$ to $-\frac{c'}{\log\log X}+i(\log X)^2$, and the line segment $L_3$ from $-\frac{c'}{\log\log X}+i(\log X)^2$ to $\frac{1}{\log X}+i(\log X)^2$.  The contributions of the integrals on the new lines are negligible due to the rapid decay of the gamma function on $L_1$ and $L_3$ and the estimation $X^s\ll \exp\left( -c'\frac{\log X}{\log\log X}\right)$ on $L_2$. We are then left with the contribution from a residue of the pole at $s=0$, which we present as an integral along a circle centered at $0$ to obtain
\begin{align*}
\begin{split}
\Upsilon_0 =& \frac{\pi}{2\pi i} \oint\limits_{|s|=\frac{1}{\log X}} \left(\frac{2^{5/2}}{\pi}\right)^{2s}
\left ( \frac {\Gamma(\frac{1}{2}+s)}{\Gamma(\frac{1}{2})} \right )^2 X^{s} \widehat {\Phi}(1+s) \zeta_K^3(1+2s) \  \\
& \times \mathop{\int \int}_{|z_i| \leq \sqrt{\log M}}h(z_1)h(z_2)\zeta_K\Big(1+\tfrac{2+iz_1+iz_2}{\log M} \Big) \zeta_K^{-2} \Big( 1+ \tfrac{1+iz_1}{\log M} +s\Big) \zeta_K^{-2}\Big( 1+\tfrac{1+iz_2}{\log M}+s \Big) \times Q\left(\tfrac{1+iz_1}{\log M},\tfrac{1+iz_2}{\log M},s \right)  \,dz_1dz_2\frac{ds}{s} \\
& + O\left( \frac{1}{(\log X)^{2020}}\right).
\end{split}
\end{align*}
  The main contribution to $\Upsilon_0$ comes from the first terms of the Laurent expansions of the zeta functions and $Q$. We then deduce via \eqref{eq: Q at 000} that
\begin{equation*}
\begin{split}
\Upsilon_0 =& \frac{\pi}{16\pi i} \oint\limits_{|s|=\frac{1}{\log X}} \left(\frac{2^{5/2}}{\pi}\right)^{2s}
\left ( \frac {\Gamma(\frac{1}{2}+s)}{\Gamma(\frac{1}{2})} \right )^2 X^{s} \widehat{\Phi}(1+s) \mathop{\int \int}_{|z_i| \leq \sqrt{\log M}}h(z_1)h(z_2)\\
& \times \left(\frac{\log M}{2+iz_1+iz_2} \right) \left( \frac{1+iz_1}{\log M} +s\right)^2 \left(\frac{1+iz_2}{\log M}+s \right)^2\,dz_1dz_2\frac{ds}{s^4} + O\left( \frac{1}{(\log X)^{1-\varepsilon}}\right).
\end{split}
\end{equation*}
 In the above expression, we extend the integrals over $z_1,z_2$ to $\mathbb{R}^2$ with a negligible error by \eqref{eq: h has rapid decay}. Then applying the relation (see \cite[(7.3.12)]{B&P})
\begin{equation*}
%%\label{eq: integrals of derivatives of H}
\begin{split}
& \int_{-\infty}^{\infty}\int_{-\infty}^{\infty}  h(z_1)h(z_2)\frac{(1+iz_1)^j(1+iz_2)^k}{2+iz_1+iz_2}\, dz_1dz_2 = (-1)^{j+k}\int_0^{\infty} H^{(j)}(t) H^{(k)}(t)\,dt,
\end{split}
\end{equation*}
 we deduce that
\begin{align*}
\begin{split}
\Upsilon_0 =&\frac{\pi}{16\pi i} \oint\limits_{|s|=\frac{1}{\log X}} \left(\frac{2^{5/2}}{\pi}\right)^{2s}
\left ( \frac {\Gamma(\frac{1}{2}+s)}{\Gamma(\frac{1}{2})} \right )^2 X^{s} \widehat{\Phi}(1+s) \Bigg\{ \frac{1}{(\log M)^3} \int_0^1 H''(t)^2\,dt \\
& -  \frac{4s}{(\log M)^2} \int_0^1 H'(t) H''(t)\,dt + \frac{2s^2}{\log M} \int_0^1 H(t) H''(t)\,dt + \frac{4s^2}{\log M} \int_0^1 H'(t)^2 \,dt \\
&- 4s^3\int_0^1 H(t)H'(t) \,dt + s^4\log M \int_0^1 H(t)^2\,dt \Bigg\}\,\frac{ds}{s^4}  + O\left( \frac{1}{(\log X)^{1-\varepsilon}}\right).
\end{split}
\end{align*}
  We now apply \eqref{eq: residue as derivative} to evaluate the above integral as a residue to arrive that
\begin{equation*}
\begin{split}
 \Upsilon_0 =& \frac{\pi \widehat{\Phi}(1)}{8} \Bigg\{ \frac{1}{6}\left(\frac{\log X}{\log M}\right)^3 \int_0^1 H''(t)^2\,dt  -  2\left(\frac{\log X}{\log M}\right)^2 \int_0^1 H'(t) H''(t)\,dt \\
& + 2\frac{\log X}{\log M} \int_0^1 H(t) H''(t)\,dt + 4\frac{\log X}{\log M} \int_0^1 H'(t)^2 \,dt  - 4\int_0^1 H(t)H'(t) \,dt  \Bigg\}  + O\left( \frac{1}{(\log X)^{1-\varepsilon}}\right).
\end{split}
\end{equation*}
  Substituting the above into \eqref{eq: T0 after sieve}, and noting that
\begin{align*}
\widehat{\Phi}(1) = \frac{1}{2} + O\left( \frac{1}{\log X}\right),
\end{align*}
 we obtain that
\begin{equation}\label{eq: T0 after mollifier}
\begin{split}
\mathcal{T}_0 =&  2X  \frac{1+o(1)}{\log R} \Bigg\{ \frac{1}{24}\left(\frac{\log X}{\log M}\right)^3 \int_0^1 H''(t)^2\,dt  \\
&-  \frac{1}{2}\left(\frac{\log X}{\log M}\right)^2 \int_0^1 H'(t) H''(t)\,dt +
 \frac{\log X}{2\log M} \int_0^1 H(t) H''(t)\,dt + \frac{\log X}{\log M} \int_0^1 H'(t)^2 \,dt  \\
&- \int_0^1
 H(t)H'(t) \,dt  \Bigg\} + O\left( \frac{X}{(\log X)^{1-\varepsilon}}+ \frac{X^{1+\varepsilon}}{Y}\right).
\end{split}
\end{equation}

\subsection{Evaluation of $\mathcal{B}$: the principal terms}

 In this section, we begin to evaluate $\mathcal{B}$. We recall from \eqref{SM2} that we have
\begin{align}
\label{Bexpression}
 \mathcal{B} =& X\sum_{\substack{ N(d) \leq D \\ d \equiv 1 \bmod {(1+i)^3} }} \mu_{[i]}^2(d)\lambda_d
 \mathop{\sum\sum}_{\substack{N(m_1),N(m_2)\leq M \\ m_1,m_2 \equiv 1 \bmod {(1+i)^3} \\ (m_1m_2,d)=1}} \frac{b_{m_1}b_{m_2}}{N(m_1m_2)^{3/2}} \mathcal{Q},
\end{align}
 where
\begin{equation*}
%%\label{eq: defn of Q1 star}
\mathcal{Q} = \sum_{\substack{\nu \equiv 1 \bmod {(1+i)^3} \\ (\nu,d)=1}} \frac {d_{[i],2}(\nu)}{N(\nu)^{\frac{3}{2}}} \ \sum_{\substack{ N(\alpha) \leq Y \\ \alpha \equiv 1 \bmod {(1+i)^3} \\ (\alpha,m_1m_2\nu)=1}} \frac {\mu_{[i]}(\alpha)}{N(\alpha^2d_1)} \left( \frac{d_1}{m_1m_2\nu}\right) \sum_{\substack{k \in \mathcal O_K \\ k \neq 0}}(-1)^{N(k)}\widetilde{F}_{N(\nu)}\left( \sqrt{\frac{N(k)X}{2N(\alpha^2d_1m_1m_2\nu)}}\right)g(k,m_1m_2\nu).
\end{equation*}
  We apply Mellin inversion to see that for any $c>1$,
\begin{align}\label{Q12}
\begin{split}
 \mathcal{Q} =&  \sum_{\substack{ N(\alpha) \leq Y \\ \alpha \equiv 1 \bmod {(1+i)^3} \\ (\alpha,m_1m_2)=1}} \frac {\mu_{[i]}(\alpha)}{N(\alpha^2d_1)} \left( \frac{d_1}{m_1m_2}\right)  \sum_{\substack{k \in \mathcal O_K \\ k \neq 0}}(-1)^{N(k)}\\
& \times  \frac{1}{2\pi i} \int\limits_{(c)}  h(\frac{N(k)X}{2N(\alpha^2d_1m_1m_2)}, w) \sum_{\substack{\nu \equiv 1 \bmod {(1+i)^3} \\ (\nu,\alpha d)=1}} \frac {d_{[i],2}(\nu)}{N(\nu)^{\frac{3}{2}+w}} \left( \frac{d_1}{\nu}\right) g(k,m_1m_2\nu) \,dw ,
\end{split}
\end{align}
  where $h$ is defined in \eqref{h}. We further apply Lemma \ref{lem: nu-sum as an Euler product} to recast $\mathcal{Q}$ as
\begin{align*}
%%\label{Q12'}
\begin{split}
 \mathcal{Q}  = & \sum_{\substack{ N(\alpha) \leq Y \\ \alpha \equiv 1 \bmod {(1+i)^3} \\ (\alpha,m_1m_2)=1}} \frac {\mu_{[i]}(\alpha)}{N(\alpha^2d_1)} \left( \frac{d_1}{m_1m_2}\right)  \sum_{\substack{k \in \mathcal O_K \\ k \neq 0}}(-1)^{N(k)} \\
& \times  \frac{1}{2\pi i} \int\limits_{(c)}  h(\frac{N(k)X}{2N(\alpha^2d_1m_1m_2)}, w) L(1+w,\chi_{ik_1})^2\mathcal{G}_{0}(1+w;k,m_1m_2,\alpha, d) \,dw.
\end{split}
\end{align*}

 Observe that integration by parts shows that when $\Re(w) \geq -\frac 12+\varepsilon$, we have for any integer $j \geq 1$,
\begin{align}
\label{Phibound}
 \widehat \Phi(1+w) \ll_j \frac{1}{(1+|w|)^j}.
\end{align}
  It follows Lemma \ref{lem: nu-sum as an Euler product}, Lemma \ref{lem: properties of h(xi,w)} and the above estimation that we can move the line of integration of the $w$-integral in \eqref{Q12} to $c=-\frac{1}{2}+\varepsilon$. We encounter a pole at $w=0$ only when $\chi_{ik_1}$ is a principal character, which holds if and only if $k_1=\pm i$ and this is further equivalent to $kd_1=\pm i j^2$ for some $j \in G$ by \eqref{eq: defn of k1 and k2}. We thus deduce that
\begin{equation}\label{Q14}
\mathcal{Q} = \mathcal{P}_{+}+\mathcal{P}_{-} + \mathcal{R},
\end{equation}
where
\begin{align}\label{eq: defn of P1}
\begin{split}
\mathcal{P}_{\pm} =& \underset{w=0}{\mbox{Res}} \  \sum_{\substack{ N(\alpha) \leq Y \\ \alpha \equiv 1 \bmod {(1+i)^3} \\ (\alpha,m_1m_2)=1}} \frac {\mu_{[i]}(\alpha)}{N(\alpha^2d_1)} \left( \frac{d_1}{m_1m_2}\right)   \sum_{\substack{k \in \mathcal O_K \\ k \neq 0 \\ k_1=\pm i}}(-1)^{N(k)}  \\
& \times  h(\frac{N(k)X}{2N(\alpha^2d_1m_1m_2)}, w) \zeta_K(1+w)^2\mathcal{G}_{0}(1+w;k,m_1m_2,\alpha, d), \\
 \mathcal{R} = & \sum_{\substack{ N(\alpha) \leq Y \\ \alpha \equiv 1 \bmod {(1+i)^3} \\ (\alpha,m_1m_2)=1}} \frac {\mu_{[i]}(\alpha)}{N(\alpha^2d_1)} \left( \frac{d_1}{m_1m_2}\right)   \sum_{\substack{k \in \mathcal O_K \\ k \neq 0 }}(-1)^{N(k)}   \\
& \times   \frac{1}{2\pi i} \int\limits_{(-\frac{1}{2}+\varepsilon)}   h(\frac{N(k)X}{2N(\alpha^2d_1m_1m_2)}, w) L(1+w,\chi_{ik_1})^2\mathcal{G}_{0}(1+w;k,m_1m_2,\alpha, d) \,dw .
\end{split}
\end{align}

 We treat $\mathcal{P}_{\pm}$ first. Note that $d_1$ is square-free by \eqref{eq: defn of d 1} since $d$ is square-free. It follows that $kd_1=\pm i j^2$ for some $j \in G$ if and only if $k=\pm i d_1 {j'}^2$ for some $j' \in G$ . Thus we relabel $k$ as $\pm i d_1j^2$ with $j$ running through all elements in $G$ in \eqref{eq: defn of P1} and apply Lemma~\ref{lem: properties of h(xi,w)} to see that for $c>\frac 12$, we have
\begin{align}\label{P12}
\begin{split}
\mathcal{P}_{\pm} =& \underset{w=0}{\mbox{Res}} \  \sum_{\substack{ N(\alpha) \leq Y \\ \alpha \equiv 1 \bmod {(1+i)^3} \\ (\alpha,m_1m_2)=1}} \frac {\mu_{[i]}(\alpha)}{N(\alpha^2d_1)}  \zeta_K(1+w)^2 \widehat{\Phi}(1+w) X^w  \frac{\pi }{2\pi i} \int\limits_{(c)}  \left(\frac{2^{5/2}}{\pi}\right)^{2s}
\left ( \frac {\Gamma(\frac{1}{2}+s)}{\Gamma(\frac{1}{2})} \right )^2(\pi)^{-2s+2w}\frac{\Gamma (s-w)}{\Gamma (1-s+w)}  \\
& \times (2N(\alpha^2m_1m_2))^{s-w}  \sum_{\substack{j \in G }}(-1)^{N(j)} N(j)^{-2s+2w}   \left( \frac{d_1}{m_1m_2}\right) \mathcal{G}_{0}(1+w; \pm i d_1j^2,m_1m_2,\alpha, d)\,\frac{ds}{s}.
\end{split}
\end{align}

Note that part (ii) of Lemma~\ref{Gausssum} implies that for $j \in G$,  $\varpi \nmid 2\alpha d$ and $\beta\geq 1$,
$$
\left( \frac{d_1}{\varpi^{\beta}}\right) g(\pm i d_1j^2, \varpi^{\beta}) = g(\pm i j^2, \varpi^{\beta}).
$$
 This allows us to deduce from the definition of $\mathcal{G}_0$ given in Lemma~\ref{lem: nu-sum as an Euler product} that
\begin{align}
\label{Gdef}
 \left( \frac{d_1}{m_1m_2}\right) \mathcal{G}_{0}(1+w; \pm i d_1j^2,m_1m_2,\alpha, d) = \mathcal{G}(1+w; \pm i j^2,m_1m_2,\alpha d),
\end{align}
 where for any $k, \ell, \alpha \in \mathcal O_K$ and $s \in \mc$, we define $\mathcal{G}(s;k, \ell, \alpha)=\prod_{\varpi \in G}\mathcal{G}_{\varpi}(s;k, \ell, \alpha)$ similar to that given in \cite[(5.8)]{sound1} by
\begin{align}
\label{Gprod}
\begin{split}
\mathcal{G}_{\varpi}(s;k,\ell,\alpha ) \ =& \ \Bigg( 1-\frac{1}{N(\varpi)^{s}}\Bigg(\frac{ik_1}{\varpi}\Bigg)\Bigg)^2 \ \ \ \ \text{if }\varpi|2\alpha , \\
\mathcal{G}_{\varpi}(s;k,\ell,\alpha ) \ =& \ \Bigg( 1-\frac{1}{N(\varpi)^{s}}\Bigg(\frac{i k_1}{\varpi}\Bigg)\Bigg)^2\sum_{r=0}^{\infty} \frac{r+1 }{N(\varpi)^{rs}}\frac{g(k, \varpi^{r+\text{\upshape{ord}}_\varpi(\ell)})}{N(\varpi)^{r/2}}\ \ \ \ \text{if }\varpi \nmid 2\alpha .
\end{split}
\end{align}
 Here $k_1$ is the unique element in $\mathcal O_K$ that we have $k=k_1k^2_2$ with $k_1$ being square-free and $k_2 \in G$.

  It follows from Lemma \ref{Gausssum} and the above expression of $\mathcal{G}(s;k, \ell, \alpha)$ that we have $\mathcal{G}(s; i j^2, \ell, \alpha)=\mathcal{G}(s; -i j^2, \ell, \alpha)$ for $j \in G$. Thus applying \eqref{Gdef} to \eqref{P12}, we see that
\begin{align}
\begin{split}
\label{P1H}
\mathcal{P}_{\pm} =& \underset{w=0}{\mbox{Res}} \  \sum_{\substack{ N(\alpha) \leq Y \\ \alpha \equiv 1 \bmod {(1+i)^3} \\ (\alpha,m_1m_2)=1}} \frac {\mu_{[i]}(\alpha)}{N(\alpha^2d_1)}  \zeta_K(1+w)^2 \widehat{\Phi}(1+w) X^w  \\
& \times \frac{\pi }{2\pi i} \int\limits_{(c)}  \left(\frac{2^{5/2} }{\pi}\right)^{2s}N(\alpha)^{2s-2w}\left ( \frac {\Gamma(\frac{1}{2}+s)}{\Gamma(\frac{1}{2})} \right )^2
\Gamma_1 (s-w)  \mathcal{H}(s-w, 1+w; m_1m_2,\alpha d) \frac{ds}{s} ,
\end{split}
\end{align}
  where
\begin{equation*}
\begin{split}
\Gamma_1(s)=(2^{-1/2}\pi)^{-2s}\frac{\Gamma (s)}{\Gamma (1-s)},
\end{split}
\end{equation*}
  and where we define for $|v-1| \leq 1 /\log X$, and any $u$ with $\Re(u) > 1/2$,
\begin{align*}
\begin{split}
 \mathcal{H}(u, v; \ell,\alpha)
 =& N(l)^u \sum_{\substack{j \in G }}(-1)^{N(j)} N(j)^{-2u}  \mathcal{G}(v;  i j^2,l,\alpha).
\end{split}
\end{align*}

    We further note that around $s=1$,
\begin{align}
\label{zetaKexpan}
  \zeta_K(s)=\frac {\pi}{4}\cdot \frac 1{s-1}+\gamma_K+O(s-1),
\end{align}
  where $\gamma_K$ is a constant. This allows us to deduce from \eqref{P1H} and \eqref{zetaKexpan}  that
\begin{align}
\label{P1expression}
\begin{split}
\mathcal{P}_{\pm} =& \frac {\pi^3}{4^2}  \sum_{\substack{ N(\alpha) \leq Y \\ \alpha \equiv 1 \bmod {(1+i)^3} \\ (\alpha,m_1m_2)=1}} \frac {\mu_{[i]}(\alpha)}{N(\alpha^2d_1)}  \widehat{\Phi}(1)  \\
& \times \frac{1 }{2\pi i} \int\limits_{(c)}  \left(\frac{2^{5/2}N(\alpha)}{\pi}\right)^{2s}\left ( \frac {\Gamma(\frac{1}{2}+s)}{\Gamma(\frac{1}{2})} \right )^2
\Gamma_1 (s)  \mathcal{H}(s, 1; m_1m_2,\alpha d) \\
& \times \left ( \log \frac {X}{N(\alpha)^2}+\frac {\widehat{\Phi}'(1)}{\widehat{\Phi}(1)}+\frac {8}{\pi}\gamma_K-\frac {\frac {\partial}{\partial s}\mathcal{H}(s, 1; m_1m_2,\alpha d)}{\mathcal{H}(s, 1; m_1m_2,\alpha d)}+\frac {\frac {\partial}{\partial w}\mathcal{H}(s, w; m_1m_2,\alpha d)}{\mathcal{H}(s, 1; m_1m_2,\alpha d)}\Big |_{w=1}-\frac {\Gamma'_1(s)}{\Gamma_1(s)} \right )\frac{ds}{s}  .
\end{split}
\end{align}

  Note that we have
\begin{align*}
\begin{split}
& \sum_{\substack{j \in G }}(-1)^{N(j)} N(j)^{-2u}  \mathcal{G}(v;  j^2,\ell,\alpha)  \\
 = & -\sum_{\substack{j \in G \\ (j, 1+i) =1 }}N(j)^{-2u}  \mathcal{G}(v;  i j^2,\ell,\alpha )+\sum_{\substack{j \in G \\ 1+i|j }}N(j)^{-2u}  \mathcal{G}(v;  i j^2,\ell,\alpha ) \\
=& -\sum_{\substack{j \in G  }}N(j)^{-2u}  \mathcal{G}(v; i j^2,\ell,\alpha )+2\sum_{\substack{j \in G \\ 1+i|j }}N(j)^{-2u}  \mathcal{G}(v;  i j^2,\ell,\alpha ) \\
=& -\sum_{\substack{j \in G  }}N(j)^{-2u}  \mathcal{G}(v;  i j^2,\ell,\alpha )+2^{1-2u}\sum_{\substack{j \in G }}N(j)^{-2u}  \mathcal{G}(v;  i (1+i)^2j^2,\ell,\alpha ) \\
=& -(1-2^{1-2u})\sum_{\substack{j \in G  }}N(j)^{-2u}  \mathcal{G}(v;  j^2,\ell,\alpha ),
\end{split}
\end{align*}
  where the last step above follows by noting that we have $\mathcal{G}(v; i (1+i)^2j^2,\ell,\alpha )=\mathcal{G}(v;  i j^2,\ell,\alpha )$.
 We deduce from this that
\begin{align}
\label{Hexpression}
\begin{split}
 \mathcal{H}(u, v; \ell,\alpha)=& -N(l)^u (1-2^{1-2u})\sum_{\substack{j \in G  }}N(j)^{-2u} \mathcal{G}(v; i j^2,l,\alpha).
\end{split}
\end{align}

   We now write
\begin{equation*}
 \ell = \ell'_1 {\ell'_2}^2, \ \ \ \ \mu_{[i]}^2(\ell'_1)=1, \ \ell'_2\in G,
\end{equation*}
  and apply the above definition of $\mathcal{G}$ in \eqref{Gprod} and Lemma \ref{Gausssum} to obtain from \eqref{Hexpression} that
\begin{align}\label{eq: write inner j sum as Euler product}
\begin{split}
 \mathcal{H}(u, v; \ell,\alpha )
=& -N(l)(1-2^{1-2u}) N(\ell'_1)^{u-\frac{1}{2}} \zeta_K(2u)\zeta_K(2u+2v-1) \mathcal{H}_1(u,v; \ell,\alpha),
\end{split}
\end{align}
  where
\begin{align}
\begin{split}
 \mathcal{H}_1(u,v; \ell,\alpha)=\prod_{\varpi \in G}\mathcal{H}_{1,\varpi}(u,v; \ell,\alpha),
\end{split}
\end{align}
  with
\begin{equation}
\label{H1}
\mathcal{H}_{1,\varpi}(u,v; \ell,\alpha) =
\begin{cases}
\left( 1-\frac{1}{N(\varpi)^{v}}\right)^2\left( 1-\frac{1}{N(\varpi)^{2u+2v-1}}\right) & \text{if } \varpi|2\alpha, \\
\frac{\left( 1-\frac{1}{N(\varpi)^{v}}\right)^2}{ \left( 1- \frac{1}{N(\varpi)^{2u+2v-1}}\right)}\left( 1+\frac{2}{N(\varpi)^{v}} -\frac{2}{N(\varpi)^{2u+v}}+\frac{1}{N(\varpi)^{2u+2v-1}}-\frac{3}{N(\varpi)^{2u+2v}}+\frac{1}{N(\varpi)^{4u+4v-1}}\right) & \text{if } \varpi \nmid 2\alpha \ell,  \\
\frac{\left( 1-\frac{1}{N(\varpi)^{v}}\right)^2}{ \left( 1- \frac{1}{N(\varpi)^{2u+2v-1}}\right)}\left( 1-\frac{1}{N(\varpi)^{2u}} +\frac{2}{N(\varpi)^{2u+v-1}}-\frac{2}{N(\varpi)^{2u+v}}+\frac{1}{N(\varpi)^{2u+2v-1}}-\frac{1}{N(\varpi)^{4u+2v-1}}\right) & \text{if } \varpi| \ell'_1, \\
\frac{\left( 1-\frac{1}{N(\varpi)^{v}}\right)^2}{ \left( 1- \frac{1}{N(\varpi)^{2u+2v-1}}\right)}\left( 1-\frac{1}{N(\varpi)} +\frac{2}{N(\varpi)^{v}}-\frac{2}{N(\varpi)^{2u+v}}+\frac{1}{N(\varpi)^{2u+2v-1}}-\frac{1}{N(\varpi)^{2u+2v}}\right) & \text{if } \varpi|\ell, \ \varpi \nmid \ell'_1.
\end{cases}
\end{equation}

It follows from \eqref{eq: write inner j sum as Euler product} to \eqref{H1} that
\begin{align*}
\begin{split}
 \mathcal{H}(s,1; \ell,\alpha), \ \frac {\partial}{\partial s}\mathcal{H}(s, 1; \ell,\alpha ), \ \frac {\partial}{\partial w}\mathcal{H}(s, w; \ell,\alpha ) \Big |_{w=1}
\end{split}
\end{align*}
 are holomorphic for $\Re(s)>0$. Moreover, by the convexity bounds for $\zeta_K(s), \zeta'_K(s)$ given in \eqref{zetaconvexitybounds},
we deduce that when $\Re(s) \geq  \frac 1{\log X}$, we have
\begin{align*}
\begin{split}
 \mathcal{H}(s,1; \ell,\alpha), \ \frac {\partial}{\partial s}\mathcal{H}(s, 1; \ell,\alpha ), \ \frac {\partial}{\partial w}\mathcal{H}(s, w; \ell,\alpha ) \Big |_{w=1} \ll N(l)^{1+\varepsilon}N(l'_1)^{\Re(s)-1/2}(N(\alpha)X)^{\varepsilon}(1+|s|)^{1+\varepsilon}.
\end{split}
\end{align*}

  This allows us to shift the line of integration in \eqref{P1expression} to $\Re(s)=\frac{1}{\log X}$ to see that the integral on the new line is bounded by
\begin{align*}
%%\label{residuebound}
\begin{split}
\ll & N(m_1m_2)N(\ell_1)^{-\frac{1}{2}+\varepsilon} N(\alpha)^{\varepsilon} X^{\varepsilon}\int\limits_{\left( \frac{1}{\log X}\right)} \left|\Gamma\left( \tfrac{1}{2}+s\right)\right|^2 \max(|\Gamma_1(s)|,|\Gamma_1'(s)| )(1+|s|)^{1+\varepsilon} \,|ds| \\
\ll & N(m_1m_2)N(\ell_1)^{-\frac{1}{2}+\varepsilon} N(\alpha)^{\varepsilon} X^{\varepsilon},
\end{split}
\end{align*}
  where we write
\begin{equation*}
%%\label{eq: defn of ell 1}
m_1m_2 = \ell_1 \ell_2^2, \ \ \ \ \mu_{[i]}^2(\ell_1)=1, \ \ell_2\in G.
\end{equation*}

 Note that \eqref{P1H} is equivalent to \eqref{P1expression}, so we can now set  $c=1/\log X$ there. We then drop the condition $N(\alpha) \leq Y$ in the sum over $\alpha$ and apply an argument similar to that in \eqref{Nalphasum} to estimate the sum over $N(\alpha) > Y$ in \eqref{P1H} to see that
\begin{align}
\label{P1K}
\begin{split}
\mathcal{P}_{\pm} =& \underset{w=0}{\mbox{Res}} \   \zeta_K(1+w)^2 \widehat{\Phi}(1+w) X^w  \\
& \times \frac{\pi }{2\pi i} \int\limits_{(\frac 1{\log X})}  \left(\frac{2^{5/2}}{\pi}\right)^{2s}\left ( \frac {\Gamma(\frac{1}{2}+s)}{\Gamma(\frac{1}{2})} \right )^2
\Gamma_1 (s-w)  \mathcal{K}(s-w, 1+w; m_1m_2) \frac{ds}{s} + O\Bigg( \frac{N(m_1 m_2) N(\ell_1)^{-\frac{1}{2}+\varepsilon} X^{\varepsilon}}{N(d)^{1-\varepsilon}Y^{1-\varepsilon}} \Bigg),
\end{split}
\end{align}
  where
\begin{align*}
\begin{split}
 \mathcal{K}(s, w; \ell, d)=\sum_{\substack{  \alpha \equiv 1 \bmod {(1+i)^3} \\ (\alpha,\ell)=1}} \frac {\mu_{[i]}(\alpha)}{N(\alpha)^{2-2s}N(d_1)}  \mathcal{H}(s, w; \ell,\alpha d).
\end{split}
\end{align*}

   We can recast $\mathcal{P}_{\pm}$ given in \eqref{P1K} more explicitly as
\begin{align}
\label{P1Kexplicit}
\begin{split}
\mathcal{P}_{\pm} =& \frac {\pi^3}{4^2} \widehat{\Phi}(1)  \frac{1 }{2\pi i} \int_{(c)}  \left(\frac{2^{5/2} }{\pi}\right)^{2s}\left ( \frac {\Gamma(\frac{1}{2}+s)}{\Gamma(\frac{1}{2})} \right )^2
\Gamma_1 (s)  \mathcal{K}(s, 1; m_1m_2,d) \\
& \times \left ( \log X+\frac {\widehat{\Phi}'(1)}{\widehat{\Phi}(1)}+\frac {8}{\pi}\gamma_K-\frac {\frac {\partial}{\partial s}\mathcal{K}(s, 1; m_1m_2, d)}{\mathcal{K}(s, 1; m_1m_2,d)}+\frac {\frac {\partial}{\partial w}\mathcal{K}(s, w; m_1m_2,d)}{\mathcal{K}(s, 1; m_1m_2, d)} \Big |_{w=1}-\frac {\Gamma'_1(s)}{\Gamma_1(s)} \right )\frac{ds}{s} \\
& + O\Bigg( \frac{N(m_1 m_2) N(\ell_1)^{-\frac{1}{2}+\varepsilon} X^{\varepsilon}}{N(d)^{1-\varepsilon}Y^{1-\varepsilon}} \Bigg)  .
\end{split}
\end{align}

  Since $\lambda_d\ll N(d)^{\varepsilon}$ by \eqref{lambda} and $b_m\ll 1$ by \eqref{eq: defn of mollifier coeffs bm}, it thus follows from \eqref{eq: 2nd bound for R over dyadic intervals} that
\begin{equation}
\label{eq: bound for contribution of error term of P1}
\sum_{\substack{ N(d) \leq D \\ d \equiv 1 \bmod {(1+i)^3} }} \mu_{[i]}^2(d)\lambda_d
 \mathop{\sum\sum}_{\substack{N(m_1),N(m_2)\leq M \\ m_1,m_2 \equiv 1 \bmod {(1+i)^3} \\ (m_1m_2,d)=1}} \frac{b_{m_1}b_{m_2}}{N(m_1m_2)^{3/2}}  \, O\Bigg( \frac{N(m_1 m_2) N(\ell_1)^{-\frac{1}{2}+\varepsilon} X^{\varepsilon}}{N(d)^{1-\varepsilon}Y^{1-\varepsilon}} \Bigg) \ll \frac{X^{\varepsilon}D^{\varepsilon}M^{\varepsilon}  }{Y^{1-\varepsilon}}.
\end{equation}

\subsection{Evaluation of $\mathcal{B}$: the remainder terms}\label{subsec: error term}

In this section, we estimate $\mathcal{R}$ given in \eqref{eq: defn of P1}. We denote
\begin{align}\label{eq: defn of R}
\begin{split}
 \mathcal{R}(\ell,d) =& \frac{1}{N(\ell)} \sum_{\substack{ N(\alpha) \leq Y \\ \alpha \equiv 1 \bmod {(1+i)^3} \\ (\alpha,\ell)=1}} \frac {\mu_{[i]}(\alpha)}{N(\alpha^2d_1)} \left( \frac{d_1}{\ell}\right)   \sum_{\substack{k \in \mathcal O_K \\ k \neq 0 }}(-1)^{N(k)}   \\
& \times    \frac{1}{2\pi i} \int\limits_{(-\frac{1}{2}+\varepsilon)}   h(\frac{N(k)X}{2N(\alpha^2d_1\ell)}, w) L(1+w,\chi_{ik_1})^2\mathcal{G}_{0}(1+w;k,\ell,\alpha, d) \,dw
\end{split}
\end{align}
to see that $\mathcal{R}=N(m_1m_2)\mathcal{R}(m_1m_2,d)$.  We let $\beta_{\ell,d} =\overline{\mathcal{R}(\ell,d)}/|\mathcal{R}(\ell,d)|$ if $\mathcal{R}(\ell,d)\neq 0$, and $\beta_{\ell,d}=1$ otherwise. Thus $|\mathcal{R}(\ell,d)|=\beta_{\ell,d} \mathcal{R}(\ell,d)$ and we deduce from \eqref{eq: defn of R} that for integers $J,V\geq 1$,
\begin{equation}\label{eq: bound for R over dyadic intervals}
\sum_{\substack{N(d)=V \\ (d,2)=1}}^{2V-1}\sum_{\substack{N(\ell)=J \\ (\ell,2d)=1}}^{2J-1} |\mathcal{R}(\ell,d)| = \sum_{\substack{N(d)=V \\ (d,2)=1}}^{2V-1} \sum_{\substack{N(\ell)=J \\ (\ell,2d)=1}}^{2J-1} \beta_{\ell,d}\mathcal{R}(\ell,d) \ll\sum_{\substack{ N(\alpha) \leq Y \\ \alpha \equiv 1 \bmod {(1+i)^3} }} \frac {1}{N(\alpha)^2}  \sum_{\substack{k \in \mathcal O_K \\ k \neq 0 }}\int\limits_{(-\frac{1}{2}+\varepsilon)} U(\alpha,k,w)\,|dw|,
\end{equation}
where
\begin{equation*}
U(\alpha,k,w) = \sum_{\substack{N(d)=V \\ (d,2)=1}}^{2V-1} \frac{1}{N(d_1)}|L(1+w,\chi_{i k_1})|^2\Bigg|  \sum_{\substack{N(\ell)=J \\ (\ell,2\alpha d)=1}}^{2J-1}  \frac{\beta_{\ell,d}}{N(\ell)}\left( \frac{d_1}{\ell}\right) \mathcal{G}_0(1+w;k,\ell,\alpha,d) h(\frac{N(k)X}{2N(\alpha^2d_1\ell)}, w)\Bigg|.
\end{equation*}
We apply the Cauchy-Schwarz inequality inequality to see that for an integer $K \geq 1$,
\begin{align}\label{eq: applying Cauchy to sum of U}
\begin{split}
 & \sum_{\substack{ K\leq N(k)< 2K }}U(\alpha,k,w) \\
\ll & \Bigg(\sum_{\substack{N(d)=V \\ (d,2)=1}}^{2V-1} \frac{1}{N(d_1)} \sum_{\substack{ K\leq N(k)<2K }} N(k_2)  | L(1+w,\chi_{i k_1})|^4  \Bigg)^{\frac{1}{2}} \\
& \times \Bigg(\sum_{\substack{N(d)=V \\ (d,2)=1}}^{2V-1} \frac{1}{N(d_1)} \sum_{\substack{ K\leq N(k)< 2K  }}\frac{1}{N(k_2)} \Bigg| \sum_{\substack{N(\ell)=J \\ (\ell,2\alpha d)=1}}^{2J-1} \frac{\beta_{\ell,d}}{N(\ell)}\left( \frac{d_1}{\ell}\right) \mathcal{G}_0(1+w;k,\ell,\alpha,d) h(\frac{N(k)X}{2N(\alpha^2d_1\ell)}, w)\Bigg|^2 \Bigg)^{\frac{1}{2}},
\end{split}
\end{align}
where $k_2$ is defined by \eqref{eq: defn of k1 and k2}. Note that \eqref{eq: defn of d 1} implies $N(d_1) \geq N(d)/N(\alpha)$, so that we have
\begin{align*}
\begin{split}
& \sum_{\substack{N(d)=V \\ (d,2)=1}}^{2V-1} \frac{1}{N(d_1)} \sum_{\substack{ K\leq N(k)<2K  }} N(k_2)  | L(1+w,\chi_{ik_1})|^4 \\
\ll &  \frac{N(\alpha)}{V} \sum_{\substack{0 \neq N(k_1) \ll KV }} | L(1+w,\chi_{k_1})|^4 \sum_{N(k_2) \ll \sqrt{\frac{KV}{N(k_1)}} }N(k_2) \sum_{\substack{N(d)=V \\ (d,2)=1 \\ d_1|k_1k_2^2}}^{2V-1} 1\\
\ll&  N(\alpha)^{1+\varepsilon} K^{1+\varepsilon} V^{\varepsilon} \sum_{\substack{0 \neq N(k_1) \ll KV }} \frac{1}{N(k_1)}| L(1+w,\chi_{i k_1})|^4  \\
\ll& N(\alpha)^{1+\varepsilon} K^{1+\varepsilon} V^{\varepsilon} (1+|w|^2)^{1+\varepsilon},
\end{split}
\end{align*}
  where the last estimation above follows from Lemma~\ref{lem:2.3} and partial summation. We then deduce from this and \eqref{eq: applying Cauchy to sum of U} that
\begin{align}\label{eq:using moment estimates of Heath Brown}
\begin{split}
& \sum_{\substack{ K\leq N(k)<2K }}U(\alpha,k,w) \\
\ll & \Bigg(N(\alpha)^{1+\varepsilon} K^{1+\varepsilon} V^{\varepsilon} (1+|w|^2)^{1+\varepsilon}  \Bigg)^{\frac{1}{2}} \\
& \times \Bigg(\sum_{\substack{N(d)=V \\ (d,2)=1}}^{2V-1} \frac{1}{N(d_1)} \sum_{\substack{ K\leq N(k)<2K  }}\frac{1}{N(k_2)} \Bigg| \sum_{\substack{N(\ell)=J \\ (\ell,2\alpha d)=1}}^{2J-1} \frac{\beta_{\ell,d}}{N(\ell)}\left( \frac{d_1}{\ell}\right) \mathcal{G}_0(1+w;k,\ell,\alpha,d)h(\frac{N(k)X}{2N(\alpha^2d_1\ell)}, w) \Bigg|^2 \Bigg)^{\frac{1}{2}}.
\end{split}
\end{align}

  Applying \eqref{Phibound} together with the first bound of Lemma~\ref{lem: version of Lemma 5.4 of Sound} implies that we may restrict the sum of the right-hand side of \eqref{eq:using moment estimates of Heath Brown} to $K=2^j \leq N(\alpha)^2 V J(1+|w|^2)(\log X)^4$, in which case we apply the second bound in Lemma~\ref{lem: version of Lemma 5.4 of Sound} to \eqref{eq:using moment estimates of Heath Brown} to see that
\begin{equation*}
\begin{split}
\sum_{\substack{ K\leq N(k)<2K }}U(\alpha,k,w)
& \ll_{\varepsilon} (1+|w|^2)^{\frac{1}{2}+\varepsilon}|\widehat{\Phi}(1+w)|N(\alpha  JKVX)^{\varepsilon} \Bigg(  \frac{ N(\alpha)^3 V(JK+J^2)}{X} \Bigg)^{\frac{1}{2}} \\
& \ll_{\varepsilon} (1+|w|^2)^{1+\varepsilon}|\widehat{\Phi}(1+w)|(N(\alpha)  JKVX)^{\varepsilon} \frac{\alpha^{\frac{5}{2}} VJ}{X^{\frac{1}{2}}}.
\end{split}
\end{equation*}
 Substituting this into \eqref{eq: bound for R over dyadic intervals} and summing over $K=2^j, K\leq N(\alpha)^2 V J(1+|w|^2)(\log X)^4$, we deduce via \eqref{Phibound} that
\begin{equation}\label{eq: 2nd bound for R over dyadic intervals}
\sum_{\substack{N(d)=V \\ (d,2)=1}}^{2V-1}\sum_{\substack{N(\ell)=J \\ (\ell,2d)=1}}^{2J-1} |\mathcal{R}(\ell,d)| \ll \frac{ V^{1+\varepsilon}J^{1+\varepsilon} Y^{\frac{3}{2}+\varepsilon}}{X^{\frac{1}{2}-\varepsilon}}.
\end{equation}
 As $\mathcal{R}=N(m_1m_2)\mathcal{R}(m_1m_2,d)$, we apply the bounds that $\lambda_d\ll N(d)^{\varepsilon}$ by \eqref{lambda} and $b_m\ll 1$ by \eqref{eq: defn of mollifier coeffs bm} to derive from \eqref{eq: 2nd bound for R over dyadic intervals} that
\begin{equation}\label{eq: bound for contribution of R1}
\sum_{\substack{ N(d) \leq D \\ d \equiv 1 \bmod {(1+i)^3} }} \mu_{[i]}^2(d)\lambda_d
 \mathop{\sum\sum}_{\substack{N(m_1),N(m_2)\leq M \\ m_1,m_2 \equiv 1 \bmod {(1+i)^3} \\ (m_1m_2,d)=1}} \frac{b_{m_1}b_{m_2}}{N(m_1m_2)^{3/2}}  \,\mathcal{R}  \ll \frac{ D^{1+\varepsilon}M^{1+\varepsilon} Y^{\frac{3}{2}+\varepsilon}}{X^{\frac{1}{2}-\varepsilon}}.
\end{equation}

\subsection{Gathering the terms}

We deduce from \eqref{Bexpression}, \eqref{Q14}, \eqref{P1Kexplicit}, \eqref{eq: bound for contribution of error term of P1} and \eqref{eq: bound for contribution of R1} that
\begin{align}
\label{Bexp}
\begin{split}
 \mathcal{B}=& 2X\sum_{\substack{ N(d) \leq D \\ d \equiv 1 \bmod {(1+i)^3} }} \mu_{[i]}^2(d)\lambda_d
 \mathop{\sum\sum}_{\substack{N(m_1),N(m_2)\leq M \\ m_1,m_2 \equiv 1 \bmod {(1+i)^3} \\ (m_1m_2,d)=1}} \frac{b_{m_1}b_{m_2}}{N(m_1m_2)^{3/2}} \, \frac {\pi^3}{4^2} \widehat{\Phi}(1)  \\
& \times \frac{1 }{2\pi i} \int\limits_{(\frac 1{\log X})}  \left(\frac{2^{5/2} }{\pi}\right)^{2s}\left ( \frac {\Gamma(\frac{1}{2}+s)}{\Gamma(\frac{1}{2})} \right )^2
\Gamma_1 (s)  \mathcal{K}(s, 1; m_1m_2,d) \\
& \times \left ( \log X+\frac {\widehat{\Phi}'(1)}{\widehat{\Phi}(1)}+\frac {8}{\pi}\gamma_K-\frac {\frac {\partial}{\partial s}\mathcal{K}(s, 1; m_1m_2, d)}{\mathcal{K}(s, 1; m_1m_2,d)}+\frac {\frac {\partial}{\partial w}\mathcal{K}(s, w; m_1m_2,d)}{\mathcal{K}(s, 1; m_1m_2, d)}\Bigg|_{w=1}-\frac {\Gamma'_1(s)}{\Gamma_1(s)} \right )\frac{ds}{s} \\
&+ O\Bigg( \frac{X^{1+\varepsilon}D^{\varepsilon}M^{\varepsilon}  }{Y^{1-\varepsilon}} + X^{\frac{1}{2}+\varepsilon}D^{1+\varepsilon}M^{1+\varepsilon} Y^{\frac{3}{2}+\varepsilon} \Big).
\end{split}
\end{align}

 We now require that the values of $\theta, \vartheta$ defined in \eqref{eq:outline section, defn of M, length of mollifier} to satisfy
\begin{equation*}
\theta+2\vartheta<\frac{1}{2}.
\end{equation*}
 Also, for small $\delta = \delta(\theta,\vartheta)$, we set the parameter $Y$ defined in \eqref{mu2approx} to be
\begin{equation}
\label{Y}
Y=X^{\delta}.
\end{equation}

  This way, we see that
\begin{align}
\label{bigOest}
\begin{split}
 O\Bigg( \frac{X^{1+\varepsilon}D^{\varepsilon}M^{\varepsilon}  }{Y^{1-\varepsilon}} + X^{\frac{1}{2}+\varepsilon}D^{1+\varepsilon}M^{1+\varepsilon} Y^{\frac{3}{2}+\varepsilon} \Big)=O(X^{1-\varepsilon}).
\end{split}
\end{align}

  Now, a direct calculation shows that
\begin{align}
\label{Ks1}
\begin{split}
& \mathcal{K}(s, 1; \ell, d)\\
 =& -\frac {N(l)}{\sqrt{N(l_1)}}\frac 1{4} \frac {4^s+4^{-s}-\frac 52}{4^s} \zeta_K(2s)\zeta_K(2s+1) \frac {\varphi_{[i]}(d\ell)^2}{N(d\ell)^2N(d)}
\prod_{\substack{\varpi \in G\\ \varpi | \ell \\ \varpi \nmid \ell_1}}\left( 1+\frac{1}{N(\varpi)}\right) \\
& \times \prod_{\substack{\varpi \in G \\ \varpi | \ell_1 }}\left( N(\varpi)^s+N(\varpi)^{-s} \right) \prod_{\substack{\varpi \in G \\ \varpi \nmid 2 l d}} \left( 1-\frac{1}{N(\varpi)}\right)^2\left ( 1+\frac{2}{N(\varpi)}+ \frac{1}{N(\varpi)^3} -\frac{1}{N(\varpi)^{2}}\left( \frac{1}{N(\varpi)^{2s}}+\frac{1}{N(\varpi)^{-2s}} \right)\right ) \\
&\times \prod_{\substack{\varpi \in G \\ \varpi | d}}\left( 1-\frac{1}{N(\varpi)^{2s+1}}\right)\left( 1-\frac{1}{N(\varpi)^{-2s+1}}\right).
\end{split}
\end{align}

   Moreover, we use the functional equation \eqref{fcneqnforzeta} for $\zeta_K(s)$ and the relation (see \cite[\S 10, (3)]{Da})
\begin{align*}
%%\label{doublegamma}
 \Gamma(s)\Gamma(s+\frac 12)=2^{1-2s}\pi^{1/2}\Gamma(2s)
\end{align*}
   to see that we have
\begin{align}
\label{Keven}
\begin{split}
 \left(\frac{2^{5/2} }{\pi}\right)^{2s}\left ( \frac {\Gamma(\frac{1}{2}+s)}{\Gamma(\frac{1}{2})} \right )^2
\Gamma_1 (s)  \mathcal{K}(s, 1; m_1m_2,d) =\left(\frac{2^{5/2} }{\pi}\right)^{-2s}\left ( \frac {\Gamma(\frac{1}{2}-s)}{\Gamma(\frac{1}{2})} \right )^2
\Gamma_1 (-s)  \mathcal{K}(-s, 1; m_1m_2, d).
\end{split}
\end{align}

  Further calculation shows that
\begin{align}
\label{Kpartialdifference}
\begin{split}
& \frac {\frac {\partial}{\partial w}\mathcal{K}(s, w; \ell, d)}{\mathcal{K}(s, 1; \ell,d )}\Big |_{w=1}-\frac {\frac {\partial}{\partial s}\mathcal{K}(s, 1; \ell,d)}{\mathcal{K}(s, 1; \ell,d)} =-\log N(\ell_1)-2\frac {\zeta'_K}{\zeta_K}(2s)+2\frac {\zeta'_K}{\zeta_K}(2s+1)+\Psi(s),
\end{split}
\end{align}
  where
\begin{align}
\label{Psidef}
\begin{split}
 \Psi(s)=& \sum_{\substack{\varpi \in G \\ \varpi| d }}\left (\frac {2\log N(\varpi)}{N(\varpi)-1}+\frac {2 \log N(\varpi)}{ N(\varpi)^{1-2s}-1 }+\frac {2 \log N(\varpi)}{N(\varpi)^{1+2s}-1 } \right )\\
& + 2\log 2+\frac {6\log 2}{(1-2^{1+2s})(1-2^{1-2s})}+\sum_{\substack{\varpi\in G \\ \varpi | \ell }}\frac {2\log N(\varpi)}{N(\varpi)-1}-\sum_{\substack{\varpi \in G \\ \varpi | \ell \\ \varpi \nmid \ell_1}}\frac {2\log N(\varpi)}{N(\varpi)+1} \\
& +\sum_{\substack{\varpi \in G \\ \varpi \nmid 2\ell d }}\left (\frac {2\log N(\varpi)}{N(\varpi)-1}-\frac {2\log N(\varpi)}{N(\varpi)} \frac { 1+ \frac{2}{N(\varpi)^2} -\frac{1}{N(\varpi)}\left( \frac{1}{N(\varpi)^{2s}}+\frac{1}{N(\varpi)^{-2s}} \right ) }{ 1+\frac{2}{N(\varpi)}+ \frac{1}{N(\varpi)^3} -\frac{1}{N(\varpi)^{2}}\left( \frac{1}{N(\varpi)^{2s}}+\frac{1}{N(\varpi)^{-2s}} \right)  }\right ) \\
=& \Psi(-s).
\end{split}
\end{align}

  This implies that
\begin{align}
\label{Kpartialdereven}
\begin{split}
& \frac {\frac {\partial}{\partial w}\mathcal{K}(s, w; \ell)}{\mathcal{K}(s, 1; \ell )}\Big |_{w=1}-\frac {\frac {\partial}{\partial s}\mathcal{K}(s, 1; \ell)}{\mathcal{K}(s, 1; \ell)}-\frac {\Gamma'_1(s)}{\Gamma_1(s)}
= \frac {\frac {\partial}{\partial w}\mathcal{K}(-s, w; \ell)}{\mathcal{K}(s, 1; \ell )}\Big |_{w=1}-\frac {\frac {\partial}{\partial s}\mathcal{K}(-s, 1; \ell)}{\mathcal{K}(-s, 1; \ell)}-\frac {\Gamma'_1(-s)}{\Gamma_1(-s)}.
\end{split}
\end{align}

  We move the line of integration in \eqref{Bexp} to $\Re(s)=-\frac{1}{\log X}$ to encounter a pole at $s=0$. A change of variable $s\mapsto -s$ together with \eqref{Keven} and \eqref{Kpartialdereven} shows that the integral on the new line is the negative of the original integral in \eqref{Bexp}. It follows that the original integral equals half of the residue of the pole at $s=0$. Representing this residue as an integral along the circle $|s|=\frac{1}{\log X}$ and applying \eqref{bigOest}, we see that
\begin{equation}\label{eq: B as a residue}
\begin{split}
\mathcal{B}=& X\sum_{\substack{ N(d) \leq D \\ d \equiv 1 \bmod {(1+i)^3} }} \mu_{[i]}^2(d)\lambda_d
 \mathop{\sum\sum}_{\substack{N(m_1),N(m_2)\leq M \\ m_1,m_2 \equiv 1 \bmod {(1+i)^3} \\ (m_1m_2,d)=1}} \frac{b_{m_1}b_{m_2}}{N(m_1m_2)^{3/2}} \, \frac {\pi^3}{4^2} \widehat{\Phi}(1)  \\
& \times  \frac{1}{2\pi i} \oint\limits_{|s|=\frac{1}{\log X}} \left(\frac{2^{5/2} }{\pi}\right)^{2s}\left ( \frac {\Gamma(\frac{1}{2}+s)}{\Gamma(\frac{1}{2})} \right )^2
\Gamma_1 (s)  \mathcal{K}(s, 1; m_1m_2,d) \\
& \times \left ( \log X+\frac {\widehat{\Phi}'(1)}{\widehat{\Phi}(1)}+\frac {8}{\pi}\gamma_K-\frac {\frac {\partial}{\partial s}\mathcal{K}(s, 1; m_1m_2, d)}{\mathcal{K}(s, 1; m_1m_2,d)}+\frac {\frac {\partial}{\partial w}\mathcal{K}(s, w; m_1m_2,d)}{\mathcal{K}(s, 1; m_1m_2, d)} \Big |_{w=1}-\frac {\Gamma'_1(s)}{\Gamma_1(s)} \right )\frac{ds}{s} + O(X^{1-\varepsilon}).
\end{split}
\end{equation}

  We proceed to sum over $d$ in \eqref{eq: B as a residue}. By doing so, we apply \eqref{Ks1}, \eqref{Kpartialdifference} and \eqref{Psidef} to  encounter the following sums:
\begin{align*}
%%\label{eq: defn of Sigma 1}
\begin{split}
\Sigma_1 =& \sum_{\substack{ N(d) \leq D \\ d \equiv 1 \bmod {(1+i)^3} \\ (d, m_1m_2)=1 }} \mu_{[i]}^2(d)\lambda_d
 \frac{\varphi_{[i]}(d)^2}{N(d)^3} \prod_{\substack{\varpi \in G \\ \varpi \nmid 2 m_1m_2 d}} \left( 1-\frac{1}{N(\varpi)}\right)^2\left ( 1+\frac{2}{N(\varpi)}+ \frac{1}{N(\varpi)^3} -\frac{1}{N(\varpi)^{2}}\left( \frac{1}{N(\varpi)^{2s}}+\frac{1}{N(\varpi)^{-2s}} \right)\right ) \\
&\times \prod_{\substack{\varpi \in G \\ \varpi | d}} \left( 1-\frac{1}{N(\varpi)^{2s+1}}\right)\left( 1-\frac{1}{N(\varpi)^{-2s+1}}\right),   \\
\Sigma_2 =& \sum_{\substack{ N(d) \leq D \\ d \equiv 1 \bmod {(1+i)^3} \\ (d, m_1m_2)=1 }} \mu_{[i]}^2(d)\lambda_d
 \frac{\varphi_{[i]}(d)^2}{N(d)^3}\prod_{\substack{\varpi \in G \\ \varpi \nmid 2 m_1m_2 d}} \left( 1-\frac{1}{N(\varpi)}\right)^2\left ( 1+\frac{2}{N(\varpi)}+ \frac{1}{N(\varpi)^3} -\frac{1}{N(\varpi)^{2}}\left( \frac{1}{N(\varpi)^{2s}}+\frac{1}{N(\varpi)^{-2s}} \right)\right ) \\
&\times \prod_{\substack{\varpi \in G \\ \varpi | d}} \left( 1-\frac{1}{N(\varpi)^{2s+1}}\right)\left( 1-\frac{1}{N(\varpi)^{-2s+1}}\right) \sum_{\substack{\varpi \in G \\ \varpi |d}}J(\varpi,s),
\end{split}
\end{align*}
where
\begin{equation*}
%%\label{eq: defn of J}
J(\varpi,s) = \frac{2\log N(\varpi)}{N(\varpi)^{1+2s}-1} + \frac{2\log N(\varpi)}{N(\varpi)^{1-2s}-1} + \left( \frac{2\log N(\varpi)}{N(\varpi)}\right)\frac{1 + \frac{2}{N(\varpi)^2} -\frac{1}{N(\varpi)}\left( N(\varpi)^{2s}+N(\varpi)^{-2s}\right) }{1 + \frac{2}{N(\varpi)} + \frac{1}{N(\varpi)^3} - \frac{1}{N(\varpi)^2}\left( N(\varpi)^{2s}+N(\varpi)^{-2s} \right) }.
\end{equation*}
  Both $\Sigma_1$ and $\Sigma_2$ can be evaluated similar to the sums defined in \cite[(7.9.5), (7.9.6)]{B&P}, using Lemma~\ref{sieve} and Lemma~\ref{sievewithsum}, respectively. The result is
\begin{align}
\begin{split}
\label{eq: Sigma 1 simplified}
\Sigma_1 =&  \frac 4{\pi} \cdot \frac{2N(m_1m_2)}{\varphi_{[i]}(m_1m_2)}\prod_{\varpi \nmid 2m_1m_2 } \left( 1-\frac{1}{N(\varpi)^2}\right)   \frac{1+o(1)}{\log R} + O\left( \frac{1}{(\log R)^{2020}}\right), \\
\Sigma_2 = &   -\frac 4{\pi} \cdot \frac{2N(m_1m_2)}{\varphi_{[i]}(m_1m_2)}\prod_{\varpi \nmid 2m_1m_2 } \left( 1-\frac{1}{N(\varpi)^2}\right)   \frac{1+o(1)}{\log R} \\
& \times \sum_{ \varpi \nmid 2m_1m_2 } \frac{J(\varpi,s)}{N(\varpi)+1}\left( 1-\frac{1}{N(\varpi)^{1+2s}}\right) \left( 1-\frac{1}{N(\varpi)^{1-2s}}\right)  + O\left( \frac{1}{(\log R)^{2020}}\right).
\end{split}
\end{align}

  Applying \eqref{eq: Sigma 1 simplified} in \eqref{eq: B as a residue}, together with \eqref{Ks1}, \eqref{Kpartialdifference}, we see that
\begin{equation}\label{eq: B after evaluation of Sigma 1 and 2}
\begin{split}
 \mathcal{B}=&  \frac {\pi^2}{6\zeta_K(2)} \widehat{\Phi}(1) X \frac{1+o(1)}{\log R}  \mathop{\sum\sum}_{\substack{N(m_1),N(m_2)\leq M \\ m_1,m_2 \equiv 1 \bmod {(1+i)^3} }} \frac{b_{m_1}b_{m_2}}{N(m_1m_2 \ell_1)^{1/2}} \,  \prod_{\substack{\varpi \in G \\ \varpi|\ell_1}}\left( \frac{N(\varpi)}{N(\varpi)+1}\right)\\
& \times \frac{1}{2\pi i} \oint\limits_{|s|=\frac{1}{\log X}}\prod_{\substack{\varpi \in G \\ \varpi|\ell_1}} (N(\varpi)^s+N(\varpi)^{-s})  \Gamma_1(s)\left(\frac{2^{3/2} }{\pi}\right)^{2s}\left ( \frac {\Gamma(\frac{1}{2}+s)}{\Gamma(\frac{1}{2})} \right )^2 \zeta_K(2s)\zeta_K(2s+1)\\
&\times  \left(\frac{5}{2} - 4^s-4^{-s}   \right)    \Bigg ( \log \left( \frac{X}{N(\ell_1)}\right) +\frac {\widehat{\Phi}'(1)}{\widehat{\Phi}(1)}+\frac {8}{\pi}\gamma_K  -\frac{\Gamma_1'(s)}{\Gamma_1(s)} - 2\frac{\zeta'_K}{\zeta_K}(2s) + 2 \frac{\zeta'_K}{\zeta_K}(2s+1) \\
&   + 2\log 2+\frac {6\log 2}{(1-2^{1+2s})(1-2^{1-2s})} +\sum_{\substack{\varpi \in G \\ \varpi \neq 1+i}} \eta_1(\varpi,s) +\sum_{\substack{\varpi \in G \\ \varpi|m_1m_2 \\ \varpi\nmid \ell_1}} \eta_2(\varpi,s)+ \sum_{\substack{ \varpi \in G \\\varpi|\ell_1}} \eta_3(\varpi,s)\Bigg ) \,\frac{ds}{s}+ O\left( \frac{X}{(\log R)^{2020}}\right),
\end{split}
\end{equation}
where
\begin{align*}
%%\label{eq: defn of eta 2}
\begin{split}
\eta_1(\varpi,s) =& \frac{2\log N(\varpi)}{N(\varpi)-1} - \left( \frac{2\log N(\varpi)}{N(\varpi)}\right)\frac{1 + \frac{2}{N(\varpi)^2} -\frac{1}{N(\varpi)}\left( N(\varpi)^{2s}+N(\varpi)^{-2s}\right) }{1 + \frac{2}{N(\varpi)} + \frac{1}{N(\varpi)^3} - \frac{1}{N(\varpi)^2}\left( N(\varpi)^{2s}+N(\varpi)^{-2s} \right) } \\
&  - \frac{J(\varpi,s)}{N(\varpi)+1}\left( 1-\frac{1}{N(\varpi)^{1+2s}}\right) \left( 1-\frac{1}{N(\varpi)^{1-2s}}\right), \\
\eta_2(\varpi,s) =& \frac{2\log N(\varpi)}{N(\varpi)-1} - \frac{2\log N(\varpi)}{N(\varpi)+1} -\eta_1(\varpi,s), \\
\eta_3(\varpi,s) =& \frac{2\log N(\varpi)}{N(\varpi)-1} - \eta_1(\varpi,s).
\end{split}
\end{align*}

  For $|s|=\frac{1}{\log X}$, we further introduce the following sums:
\begin{align*}
%%\label{eq: defn of Upsilon 1}
\begin{split}
\Upsilon_1=& \mathop{\sum\sum}_{\substack{N(m_1),N(m_2)\leq M \\ m_1,m_2 \equiv 1 \bmod {(1+i)^3} }} \frac{b_{m_1}b_{m_2}}{N(m_1m_2 \ell_1)^{1/2}} \,  \prod_{\substack{\varpi \in G \\ \varpi|\ell_1}}\left( \frac{N(\varpi)}{N(\varpi)+1}\right)\prod_{\substack{\varpi \in G \\ \varpi|\ell_1}} (N(\varpi)^s+N(\varpi)^{-s}), \\
\Upsilon_2=&- \mathop{\sum\sum}_{\substack{N(m_1),N(m_2)\leq M \\ m_1,m_2 \equiv 1 \bmod {(1+i)^3} }} \frac{b_{m_1}b_{m_2}}{N(m_1m_2 \ell_1)^{1/2}} \,  \prod_{\substack{\varpi \in G \\ \varpi|\ell_1}}\left( \frac{N(\varpi)}{N(\varpi)+1}\right)\prod_{\substack{\varpi \in G \\ \varpi|\ell_1}} (N(\varpi)^s+N(\varpi)^{-s})\log N(\ell_1), \\
\Upsilon_3=& \mathop{\sum\sum}_{\substack{N(m_1),N(m_2)\leq M \\ m_1,m_2 \equiv 1 \bmod {(1+i)^3} }} \frac{b_{m_1}b_{m_2}}{N(m_1m_2 \ell_1)^{1/2}} \,  \prod_{\substack{\varpi \in G \\ \varpi|\ell_1}}\left( \frac{N(\varpi)}{N(\varpi)+1}\right)\prod_{\substack{\varpi \in G \\ \varpi|\ell_1}} (N(\varpi)^s+N(\varpi)^{-s}) \sum_{\substack{\varpi \in G \\ \varpi|m_1m_2 \\ \varpi \nmid \ell_1}} \eta_2(\varpi,s), \\
\Upsilon_4=& \mathop{\sum\sum}_{\substack{N(m_1),N(m_2)\leq M \\ m_1,m_2 \equiv 1 \bmod {(1+i)^3} }} \frac{b_{m_1}b_{m_2}}{N(m_1m_2 \ell_1)^{1/2}} \,  \prod_{\substack{\varpi \in G \\ \varpi|\ell_1}}\left( \frac{N(\varpi)}{N(\varpi)+1}\right)\prod_{\substack{\varpi \in G \\ \varpi|\ell_1}} (N(\varpi)^s+N(\varpi)^{-s}) \sum_{\substack{\varpi \in G \\ \varpi | \ell_1}} \eta_3(\varpi,s).
\end{split}
\end{align*}
   These sums are similar to those defined in \cite[(7.9.14)-(7.9.17)]{B&P} and can be evaluated accordingly. Again by keeping in mind that the residue of $\zeta_K(s)$ at $s = 1$ equals $\pi/4$, we have that
\begin{align*}
%%\label{eq: Upsilon 1 evaluated}
\begin{split}
& \Upsilon_1 =  6\zeta_K(2)(\frac 4{\pi})^3 \Bigg(\frac{1}{\log^3 M} \int_0^1 H''(t)^2\,dt  - \frac{2s^2}{\log M}\int_0^1 H(t)H''(t)\,dt
+  s^4 \log M \int_0^1 H(t)^2\,dt\Bigg) + O\left( \frac{1}{(\log X)^{4-\varepsilon}}\right),  \\
& \Upsilon_2 = 6\zeta_K(2)(\frac 4{\pi})^3\left( -\frac{4}{\log^2 M} \int_0^1 H'(t)H''(t)\,dt + 4s^2 \int_0^1 H(t)H'(t)\,dt \right)  + O\left( \frac{1}{(\log X)^{3-\varepsilon}}\right), \\
& \Upsilon_3, \Upsilon_4  \ll  \frac{1}{(\log X)^{3/2-\varepsilon}}.
\end{split}
\end{align*}

 We now sum over $m_1,m_2$ in \eqref{eq: B after evaluation of Sigma 1 and 2} and apply the above estimations together with the observation that $\zeta_K(2s)$ is bounded when $s$ is near $0$ and $\zeta_K(1+2s) \ll \log X$ when $|s|=1/\log X$ to deduce that
\begin{equation}
\label{Bint}
\begin{split}
 \mathcal{B}=& \frac {4^3}{\pi} \widehat{\Phi}(1)X  \frac{1+o(1)}{\log R} \frac{1}{2\pi i} \oint\limits_{|s|=\frac{1}{\log X}}  \Gamma_1(s)\left(\frac{2^{3/2} }{\pi}\right)^{2s}\left ( \frac {\Gamma(\frac{1}{2}+s)}{\Gamma(\frac{1}{2})} \right )^2 \zeta_K(2s)\zeta_K(2s+1)\\
&\times  \left(\frac{5}{2} - 4^s-4^{-s}   \right)    \Bigg ( \Big ( \log X  +\frac {\widehat{\Phi}'(1)}{\widehat{\Phi}(1)}+\frac {8}{\pi}\gamma_K  -\frac{\Gamma_1'(s)}{\Gamma_1(s)} - 2\frac{\zeta'_K}{\zeta_K}(2s) + 2 \frac{\zeta'_K}{\zeta_K}(2s+1) \\
&   + 2\log 2+\frac {6\log 2}{(1-2^{1+2s})(1-2^{1-2s})} +\sum_{\substack{\varpi \in G \\ \varpi \neq 1+i}} \eta_1(\varpi,s) \Big ) \Big ( \frac{1}{\log^3 M} \int_0^1 H''(t)^2\,dt  \\
& - \frac{2s^2}{\log M}\int_0^1 H(t)H''(t)\,dt  +  s^4 \log M \int_0^1 H(t)^2\,dt \Big ) -\frac{4}{\log^2 M} \int_0^1 H'(t)H''(t)\,dt \\
&  + 4s^2 \int_0^1 H(t)H'(t)\,dt\Bigg ) \,\frac{ds}{s}+ O\left( \frac{X}{(\log X)^{3/2-\varepsilon}}\right).
\end{split}
\end{equation}

   We further use  the fact that $\zeta_K(0)=-\frac 14$ (see \cite[p. 27]{Gao2} for an computation of $\zeta_K(0)$) and that when $s \rightarrow 0$
\begin{align*}
 & \Gamma_1(s)\zeta_K(2s+1)=\frac {\pi}{8}\cdot \frac 1{s^2}+O(\frac 1s)
\end{align*}
  to evaluate the integral in \eqref{Bint} to obtain that
\begin{equation*}
\begin{split}
\mathcal{B} = 4X\widehat{\Phi}(1)\frac{1+o(1)}{\log R}\Bigg ( \frac{\log X}{2\log M} \int_0^1 H(t)H''(t)\,dt -\int_0^1 H(t)H'(t)\,dt \Bigg )
 +  O\left(X(\log X)^{-3/2+\varepsilon}\right).
\end{split}
\end{equation*}
 We then conclude from \eqref{S+SMSR}, \eqref{eq:bound on S_R^+}, \eqref{SM2},  \eqref{eq: T0 after mollifier} and the above expression that
\begin{equation*}
\begin{split}
S^+ =
& 2X \frac{1+o(1)}{\log R} \Bigg\{ \frac{1}{24}\left(\frac{\log X}{\log M}\right)^3 \int_0^1 H''(t)^2\,dt  \\
& -  \frac{1}{2}\left(\frac{\log X}{\log M}\right)^2 \int_0^1 H'(t) H''(t)\,dt + \frac{\log X}{\log M} \int_0^1 H(t) H''(t)\,dt + \frac{\log X}{\log M} \int_0^1 H'(t)^2 \,dt  \\
& - \ \ 2\int_0^1 H(t)H'(t) \,dt  \Bigg\} + O\left( \frac{X}{(\log X)^{3/2-\varepsilon}}+\frac{X^{1+\varepsilon}}{Y}+ X^{\frac{1}{2}+\varepsilon}M\right).
\end{split}
\end{equation*}
 Recalling the definition of $M, R$ from \eqref{eq:outline section, defn of M, length of mollifier} and $Y$ from \eqref{Y}, we see that the assertion of  Proposition \ref{prop:upper bound for S2} follows for the above expression of $S^+$.

\section{Proof of Theorem \ref{thm: pos prop}}\label{sec:finish proof of pos prop thm}

  In this section, we complete the proof of Theorem \ref{thm: pos prop} by first applying \eqref{eq: Cauchy Schwarz first second moment inequality}, Proposition \ref{prop: asymptotic for S1} and Proposition \ref{prop:upper bound for S2} to see that, for fixed sufficiently small $\delta_0 > 0$,
\begin{align}\label{eq:first pos prop inequality}
 \frac{2X}{(1+\delta_0)} \cdot \vartheta\frac{\left(H(0) - \frac{1}{2\theta}H'(0)\right)^2}{\mathfrak{I}} \leq \sum_{\substack{(\varpi, 2)=1 \\ L(\frac{1}{2},\chi_{(1+i)^5\varpi}) \neq 0}} \log N(\varpi)\Phi \left( \frac{N(\varpi)}{X}\right) & \leq  \log X \sum_{\substack{(\varpi, 2)=1 \\ X/2 \leq N(\varpi) \leq X \\ \\ L(\frac{1}{2},\chi_{(1+i)^5\varpi})  \neq 0}} 1.
\end{align}

 As the left side of \eqref{eq:first pos prop inequality} is an increasing function of $\vartheta$, we take $\vartheta= \frac{1}{2}(\frac{1}{2} - \theta) - \varepsilon$ to be the largest possible value allowed by the condition of Proposition \ref{prop:upper bound for S2}. This leads to
\begin{align*}
%%\label{eq:percent lower bound with cal Z}
\sum_{\substack{(\varpi, 2)=1 \\ X/2 \leq N(\varpi) \leq X \\ \\ L(\frac{1}{2},\chi_{(1+i)^5\varpi})  \neq 0}} 1 \ &\geq \frac{2X}{(1+2\delta_0) \log X} \cdot \varrho,
\end{align*}
 with
\begin{align*}
\varrho &= \frac{1}{2}\left(\frac{1}{2}-\theta \right)\frac{\left(H(0) - \frac{1}{2\theta}H'(0)\right)^2}{\mathfrak{I}}.
\end{align*}
 The optimal $H$ that maximizes $\varrho$ is determined in \cite[Section 8]{B&P} to be a smooth approximation on $[0, 1]$ of the function
$H_*(x)$ given by
\begin{align*}
H_*(x)  =(1-x)^2\left(2 + \frac{3}{2\theta}+x \right).
\end{align*}

  We then proceed as in \cite[Section 8]{B&P} to see that
\begin{align}\label{eq:percent lower bound in terms of Z theta}
\sum_{\substack{(\varpi, 2)=1 \\ X/2 \leq N(\varpi) \leq X \\ \\ L(\frac{1}{2},\chi_{(1+i)^5\varpi})  \neq 0}} 1 \ &\geq \frac{2X}{(1+O(\delta_0)) \log X} \cdot \rho(\theta),
\end{align}
  where
\begin{align*}
\rho(\theta) = \frac{1}{2} \left(\frac{1}{2} - \theta \right)\left(1 - \frac{1}{(1+2\theta)^3} \right).
\end{align*}

 By setting $\theta = \theta_0$ to be the unique positive root $\theta_0= 0.17409\ldots$ of the polynomial $16\theta^4 + 32\theta^3+24\theta^2+12\theta-3$ that maximizes $\rho(\theta)$ on $(0,\frac{1}{2})$, we have
\begin{align*}
%%\label{eq:val of Z theta at special theta}
\rho(\theta_0) = 0.09645\ldots.
\end{align*}
 We substitute this in \eqref{eq:percent lower bound in terms of Z theta} and sum over dyadic intervals. As
\begin{align*}
\sum_{\substack{(\varpi, 2)=1 \\ N(\varpi) \leq X }} 1 = (1+o(1))\frac{4X}{(\log X)},
\end{align*}
 the assertion of Theorem \ref{thm: pos prop} then follows.

\vspace*{.5cm}

\noindent{\bf Acknowledgments.} P. G. is supported in part by NSFC grant 11871082.

\bibliography{biblio}
\bibliographystyle{amsxport}

\vspace*{.5cm}

\end{document}